\documentclass[11pt,a4paper,oneside,titlepage]{book}
\usepackage{fancyhdr,amsfonts,amssymb}
\usepackage{amsmath}
\usepackage{amsthm}
\usepackage{algorithmic}
\usepackage{ifthen}
\usepackage{longtable}
\usepackage{array}
\usepackage{url}
\usepackage{multirow}
\usepackage{tikz}
\usetikzlibrary{decorations.pathreplacing, matrix,arrows,decorations.pathmorphing}
\usepackage{graphicx}
\renewcommand{\sectionmark}[1]{}
 
\fancypagestyle{plain}

\usepackage{hyperref}

\newtheorem{definition}{Definition}
\newtheorem{theorem}[definition]{Theorem}
\newtheorem{lemma}[definition]{Lemma}
\newtheorem{corollary}[definition]{Corollary}
\newtheorem{proposition}[definition]{Proposition}

\newtheorem{example}[definition]{Example}
\newtheorem{hypothesis}[definition]{Hypothesis}

\newtheorem{remark}[definition]{Remark}
\numberwithin{definition}{section}
\numberwithin{figure}{chapter}
\makeatletter\def\blfootnote{\xdef\@thefnmark{}\@footnotetext}\makeatother

\begin{document}
\thispagestyle{empty}
\begin{titlepage}
\begin{center}

\vspace{0.3cm}
{\large Dipartimento di Matematica e Informatica} \\
\vspace{0.3cm}
{\large\bf Scuola di Dottorato ISIMR - Matematica e Informatica}\\

\vspace{0.15cm}

{\large\bf in co-tutela con Technische Universit\"at Graz} \\

\vspace{0.3cm}

{\sc XXVII ciclo}

\rule[0.1cm]{\textwidth}{0.1mm}
{ S.S.D. MAT/05 -- ANALISI MATEMATICA}

\vspace{2.3cm}

{\sc Tesi di Dottorato}\\

\vspace {0.9cm}

{\LARGE\bf
Low-discrepancy sequences:\\
Theory and Applications\\
}

\vspace{2.cm}
{\Large Maria Rita Iac\`{o}}

\vspace{2.4cm}
\begin{tabular}{ccc}
{\large \bf Supervisore} & \hspace{3.1cm} & {\large \bf Direttore} \vspace{.25cm} \\
{\large Dott. Ingrid Carbone} & & {\large Prof. Nicola Leone} \vspace{.7cm}\\
{\large \bf Supervisore} & &\\
{\large Prof. Robert Tichy} \vspace{.7cm}& &\\
{\large \bf Supervisore} & &\\
{\large Prof. Aljo\v{s}a Vol\v{c}i\v{c}} & &\\

\end{tabular}

\vspace{1.8cm}
\rule[0.1cm]{\textwidth}{0.1mm}

{ A.A. 2013 -- 2014}
\end{center}
\end{titlepage}

\thispagestyle{empty}
\frontmatter
\chapter*{Abstract}

The main topic of this present thesis is the study of the asymptotic behaviour of sequences modulo 1.\\
In particular, by using ergodic and dynamical methods, a new insight to problems concerning the asymptotic behaviour of multidimensional sequences can be given, and a criterion to construct new multidimensional uniformly distributed sequences is provided.

More precisely, one considers a uniformly distributed sequence (u.d.) $(x_n)_{n\in\mathbb{N}}$ in the unit interval, i.e. a sequence satisfying the relation
\begin{equation*}
  \lim_{N \rightarrow \infty} \frac{1}{N} \sum_{n = 1}^N f(x_n) =
\int_0^1 f(x) dx\ ,
\end{equation*}
for every continuous function $f$ defined on $[0,1]$.\\
This relation suggests the possibility of a numerical approximation of the integral on the right-hand side by means of u.d. sequences, even if it does not give any information on the quality of the estimator.\\
The following quantity 
\begin{equation*}
 D_N^* = D_N^*(x_1, \ldots, x_N) = \sup_{a \in[0,1[} \left| \frac{\sum_{n=1}^N \mathbf{1}_{[0,a[}
(x_n)}{N} - \lambda([0,a[)\right|\ ,
\end{equation*}
called the \emph{star discrepancy}, has been introduced in order to have a quantitative insight on the rate of convergence.\\
One of the most important results about the integration error of this approximation technique is given by the Koksma-Hlawka inequality which states that the error can be bounded by the product of the variation of $f$ (in the
sense of Hardy and Krause), denoted by $V(f)$, and the star-discrepancy $D_N^*$ of the point sequence $(x_n)_{n \in
\mathbb{N}}$:
\begin{equation*}
 \left| \frac{1}{N} \sum_{n = 1}^N f(x_n) - \int_0^1
f(x) dx \right| \leq V(f) D_N^*(x_n)\ .
\end{equation*}
Thus in order to minimize the integration error we have to use point sequences
with small discrepancy, that is, sequences which
achieve a star-discrepancy of order $\mathcal{O}(N^{-1} (\log N))$. These sequences are called
low-discrepancy sequences and they turn out to be very useful especially for the approximation of multidimensional integrals. In this context, the error in the approximation is smaller than the probabilistic one of the standard Monte Carlo method, where a sequence of random points instead of deterministic points, is used. Methods using low-discrepancy sequences, often called quasi-random sequences, are called Quasi-Monte Carlo methods (QMC).\\

However, to construct low-discrepancy sequences, especially multidimensional ones, and to compute the discrepancy of a given sequence are in general not easy tasks.\\

The aim of this thesis is to provide a full description of these problems as well as the methods used to handle them.
The idea was to consider tools from ergodic theory in order to produce new low-discrepancy sequences. The starting point to do this, is to look at the orbit, i.e. the sequence of iterates, of a continuous transformation $T$ defined on the unit interval which has the property of being uniquely ergodic.\\
The unique ergodicity of the transformation has the following consequence:\\

If $T:[0,1]\rightarrow [0,1]$ is uniquely ergodic, then for every $f\in \mathcal{L}^{1}(X)$ 
\begin{equation*}\label{birkhoff}
\lim_{N\to\infty}\frac{1}{N}\sum_{n=0}^{N-1}f(T^{n}x)=\int_Xf(x)d\mu(x)\ ,
\end{equation*}
for every $x\in X$.\\
So the orbit of $x$ under $T$ is a uniformly distributed sequence.\\

In this respect, we devoted the first chapter entirely on classical topics in uniform distribution theory and ergodic theory. \\
This provides the basic requirements for a complete understanding of the following chapters, even to a reader who is not familiar with the subject.\\

Chapter 2 deals with a countable family of low-discrepancy sequences, namely the $LS$-sequences of points. In particular, one of these sequences will be considered in full detail in Chapter 3.\\

The content of Chapters 3, 4 and 5 is based on three published papers that I co-authored.\\

In Chapter 3, the method used to construct the transformation $T$ is the so-called
\lq\lq cutting-stacking\rq\rq\ technique. In particular, we were able to prove the ergodicity of $T$ (a weaker property than unique ergodicity), and that the orbit of the origin under this map coincides with an $LS$-sequence which turns out to be a low-discrepancy one.\\

In Chapter 4, another approach from ergodic theory is used. This approach is based on the study of dynamical systems arising from numeration systems defined by linear recurrences. In this way we could not only prove that the transformation $T$ defined in Chapter 3 is uniquely ergodic, but we could also construct multidimensional uniformly distributed sequences.\\

Finally, we go back to the problem of finding an approximation for integrals. In Chapter 5 we tried to find bounds for integrals of two-dimensional, piecewise constant functions with respect to copulas. Copulas are functions that can be viewed as asymptotic distribution functions with uniform margins. To solve this problem, we had to draw a connection to linear assignment
problems, which can be solved efficiently in polynomial time. The approximation technique was applied to problems in financial mathematics and uniform distribution theory.

\tableofcontents
\mainmatter
\chapter{Preliminary topics}
 \thispagestyle{empty}
In this chapter we collect classical results about uniform distribution, discrepancy
and ergodic theory, which are the essential background for the following chapters.

\section{Uniform distribution and discrepancy of sequences}
In this first section we recall definitions and basic properties of uniform distribution and discrepancy of sequences. For a full treatment of the subject and for proofs we refer to \cite{Drmota_Tichy,kuipers_niederreiter,Niederreiter,Porubski_Strauch}. \\ The structure of the section is the following: a first part is devoted to the definition of uniformly distributed (u.d.) sequences of points in the unit interval. \\
Furthermore we provide some historical examples constructed in $[0,1[$, such as Kronecker and van der Corput sequences.\\
Then we introduce the notion of uniform distribution and discrepancy for sequences of partitions defined on $[0,1[$, showing how to get associated uniformly distributed sequences of points from them.\\
Finally we extend the definitions and the results to the multidimensional unit cube $[0,1[^s$ and we see how this theory can be applied to Quasi-Monte Carlo methods to find numerical approximation of integrals. 
\subsection{Preliminary definitions and results in $[0,1[$}
The starting point of uniform distribution is given by the following definition introduced by H. Weyl \cite{weyl2, weyl1}.
\begin{definition}
A sequence $(x_n)_{n\in\mathbb{N}}$ of points in $[0,1[$ is said to be uniformly distributed (u.d.) if for every interval $[a, b[\ \subseteq [0,1[$ we have
\begin{equation}\label{1.1}
 \lim_{N \rightarrow \infty} \frac{1}{N}\sum_{n = 1}^N \mathbf{1}_{[a, b[} (x_n) = b-a\ ,
\end{equation}
where $\mathbf{1}_{[a, b[}$ is the characteristic function of $[a,b[$.
\end{definition}
Observing that $\lambda([a,b[)=\int_0^1\mathbf{1}_{[a, b[} (x)\ dx$, \eqref{1.1} can be written in the equivalent form
\begin{equation}
\lim_{N \rightarrow \infty} \frac{1}{N}\sum_{n = 1}^N \mathbf{1}_{[a, b[} (x_n) =\int_0^1\mathbf{1}_{[a, b[} (x)\ dx\ .
\end{equation}
This remark together with a classical approximation process of continuous functions by means of step functions led to the following criterion that can be also found in \cite{weyl2, weyl1}.
\begin{theorem}[Weyl's Theorem]\label{weyl_thm}
A sequence $(x_n)_{n\in\mathbb{N}}$ of points in $[0, 1[$ is u.d. if and only if for every real-valued continuous function $f$ defined on $[0, 1]$
\begin{equation}\label{1.3}
\lim_{N \rightarrow \infty} \frac{1}{N}\sum_{n = 1}^N f(x_n) =\int_0^1f(x)\ dx
\end{equation}
holds.
\end{theorem}
As it will appear clear at the end of the chapter, Theorem \ref{weyl_thm} suggests a numerical approximation of integrals by means of averages over uniformly distributed sequences.\\

Two more general results can be found as corollaries of the previous theorem. The first one is a generalization to Riemann-integrable functions, while the second one to complex-valued continuous functions.
\begin{corollary}\label{riemann}
A sequence $(x_n)_{n\in\mathbb{N}}$ of points in $[0, 1[$ is u.d. if and only if for every Riemann-integrable function $f$ defined on $[0, 1]$ equation \eqref{1.3} holds.
\end{corollary}
\begin{corollary}
A sequence $(x_n)_{n\in\mathbb{N}}$ of points in $[0, 1[$ is u.d. if and only if for every complex-valued continuous function $f$ defined on $\mathbb{R}$ with period 1 equation \eqref{1.3} holds.
\end{corollary}
This second corollary is particularly important since funtions of the type $f(x)=e^{2\pi ihx}$, where $h$ is a non-zero integer, satisfy its assumptions. Hence, these functions give a criterion to determine if a sequence of points is u.d..
\begin{theorem}[Weyl's Criterion]
A sequence $(x_n)_{n\in\mathbb{N}}$ of points in $[0,1[$ is u.d. if and only if 
\begin{equation}\label{1.4}
\lim_{N \rightarrow \infty} \frac{1}{N}\sum_{n = 1}^Ne^{2\pi ihx_n}=0
\end{equation}
for every integer $h \neq 0$.
\end{theorem}
The first proof of this theorem has been given by Weyl \cite{weyl2, weyl1}, but several proofs can be found in literature. 
As we will see later, Weyl applied this theorem to show that Kronecker's sequence $(\lbrace n\theta\rbrace)_{n\in\mathbb{N}}$, where $\theta$ is irrational and $\lbrace x\rbrace$ is the fractional part of $x$, is u.d. in $[0,1[$. This criterion is of particular interest since it emphasizes the connection between uniform distribution theory and the estimation of exponential sums.\\

Now we would like to introduce a quantity, the discrepancy, which measures the maximal deviation between the empirical distribution of a sequence and the uniform distribution. This term has been introduced by van der Corput and the first intensive study of discrepancy is due to van der Corput and Pisot \cite{vdc_pisot}.\\ For a detailed survey about discrepancy of sequences and applications we refer to \cite{Drmota_Tichy}.
\begin{definition}[Discrepancy]
Let $\omega = \{x_1,\dots,x_N\}$ be a finite set of real numbers in $[0,1[$. The quantity
\begin{equation}\label{discrepancy}
D_N(\omega)=\sup_{0\leq a< b\leq 1} \left|\frac{1}{N} \sum_{n=1}^N \mathbf{1}_{[a,b[}(x_n)-(b-a)\right|
\end{equation}
is called the discrepancy of the given set $\omega$.
\end{definition}
If $(x_n)_{n\in\mathbb{N}}$ is an infinite sequence of points in $[0,1[$, we associate to it the sequence of positive real numbers $D_N (\{x_1, x_2, \dots, x_N \})$. So, the notation $D_N(x_n)$ denotes the discrepancy of the initial segment $\{x_1, x_2, \dots, x_N\}$ of the infinite sequence.\\
The following result shows how discrepancy is related to uniform distribution.
\begin{theorem}\label{thm3}
A sequence $(x_n)_{n\in\mathbb{N}}$ of points in $[0, 1[$ is u.d. if and only if 
\begin{equation}\label{1.6}
\lim_{N\to\infty} D_N(x_n)=0\ .
\end{equation}
\end{theorem}
Sometimes it is useful to restrict the family of intervals considered in the definition of discrepancy to intervals of the form $[0, a[$ with $0 < a \leq 1$. This leads to the following definition of star-discrepancy.
\begin{definition}[Star discrepancy]
Let $\omega = \{x_1,\dots,x_N\}$ be a finite set of real numbers in $[0,1[$. We define star discrepancy of $\omega$ the quantity
\begin{equation}
D_N^*(\omega)=\sup_{0< a\leq 1} \left|\frac{1}{N} \sum_{n=1}^N \mathbf{1}_{[0,a[}(x_n)-a\right|\ .
\end{equation}
\end{definition}
This definition extends to an infinite sequence in the same way as it has been done for the discrepancy. \\
Moreover, discrepancy and star discrepancy are related by the following relation.
\begin{theorem}
For any sequence $(x_n)_{n\in\mathbb{N}}$ of points in $[0,1[$ we have that
\begin{equation}
D_N^*(x_n)\leq D_N(x_n)\leq 2D_N^*(x_n)\ .
\end{equation}
\end{theorem}
As a consequence, it follows that Theorem \ref{thm3} holds also for star discrepancy, i.e. $\lim_{N\to\infty}D_N^*(x_n)=0$ if and only if $(x_n)_{n\in\mathbb{N}}$ is u.d.. \\

The most well-known problem in the theory of irregularities of distributions is to determine the optimal lower bound for the discrepancy $D_N(x_n)$. Discrepancy bounds from below are essentially quantitative measures for the irregularity of point distributions.\\
A very easy observation gives the following trivial lower bound.
\begin{proposition}\label{TrivBound}
For any finite point set $\omega=\{x_1,\dots,x_N\}$ in $[0,1[$ we have that
\begin{equation*}
\frac{1}{N}\leq D_N(\omega)\leq 1 \ . 
\end{equation*} 
\end{proposition}

The finite sequence $(x_n)=\left(\frac{n}{N}\right)$ with $n=0,\dots,N-1$ satisfies $D_N(x_n)=\frac{1}{N}$. But sequences of this kind, i.e. showing a discrepancy of this order, can only exist in the one-dimensional case, as it will be clear in \S \ref{multi} as direct consequence of Roth's Theorem (Theorem \ref{roth}).\\ 
According to the previous Proposition, one could ask whether the lower bound is really attained by an infinite sequence in the unit interval. Van der Corput conjectured that such a sequence does not exist, i.e. there is no real sequence $(x_n)_{n\in\mathbb{N}}$ satisfying $D_N(x_n)=\mathcal{O}\left(\frac{1}{N}\right)$ as $N\to\infty$.\\
Although this problem was first solved by van Aardenne-Ehrenfest \cite{vAardenne_Ehrenfest1,vAardenne_Ehrenfest2}, the conjecture was completely proved from a quantitative point of view as shown in the following important result due to Schmidt \cite{Schmidt}. 
\begin{theorem}[Schmidt's Theorem]\label{SchmidtThm}
For any sequence $(x_n)_{n\in\mathbb{N}}$ in $[0,1[$ we have
\begin{equation*}
D_N(x_n)\geq C \frac{\log N}{N}
\end{equation*}
for infinitely many $N$, where $C>0$ is an absolute constant.
\end{theorem}
Sequences having discrepancy of the order $\mathcal{O}\left(\frac{\log N}{N}\right)$ are called low discrepancy sequences or quasi random sequences. As we will see in the last section, they play an important role in applications and they are especially used in the implementation of Quasi-Monte Carlo methods.\\
In particular, when dealing with the implementation of an algorithm, one actually considers point sets instead of infinite sequences. We now state two results which turn out to be useful for estimating the discrepancy of a given point set.
\begin{theorem}\label{thm6}
If $0\leq x_1\leq x_2 \leq \dots \leq x_N < 1$, then
\begin{equation*}
D_N(x_1,\dots, x_N)=\frac{1}{N}+\max_{1\leq n\leq N}\left(\frac{n}{N}-x_n\right)-\min_{1\leq n\leq N}\left(\frac{n}{N}-x_n\right)\ .
\end{equation*}
\end{theorem}
\begin{theorem}\label{ThmDiscrDecomp}
Let $\omega=\{x_1,\ldots, x_N\}$ be a finite set of $N$ points in $[0,1[$. For $1\leq j\leq r$ let $\omega_j$ be a subset of $\omega$ consisting of $N_j$ elements such that its discrepancy is $D_{N_j}(\omega_j)$, its star discrepancy is $D^*_{N_j}(\omega_j)$, with $\omega_j\cap \omega_i=\emptyset$ for all $j\neq i $ and $N= N_1+\ldots+N_r$. Then
\begin{equation*}
D_N(\omega)\leq\sum_{j=1}^{r}\frac{N_j}{N}D_{N_j}(\omega_j)
\end{equation*}
and also
\begin{equation*}
D^*_N(\omega)\leq\sum_{j=1}^{r}\frac{N_j}{N}D^*_{N_j}(\omega_j).
\end{equation*}
\end{theorem}
Now we would like to show two classical examples of low discrepancy sequences.

The two sequences that we are going to present got considerable attention since their introduction at the beginning of 1900. 
\begin{example}[Kronecker sequence]
Let $\theta$ be an irrational number. Then the sequence $(\lbrace n\theta\rbrace)_{n\in\mathbb{N}}$ is known as Kronecker sequence. It is an easy application of Weyl's criterion to show that it is u.d..\\
In fact,
\begin{equation*}
\left| \frac{1}{N} \sum_{n=1} ^{N} e^{2\pi ihn\theta} \right|=\left| \frac{1}{N}\left( \sum_{n=0} ^{N} e^{2\pi ihn\theta}-1 \right)\right|=\left| \frac{1}{N}\left( \frac{1-( e^{2\pi ih\theta})^{N+1}}{1-e^{2\pi ih\theta}}-1\right)\right| \medskip
\end{equation*}
\begin{equation*}
=\frac{1}{N}\left| \frac{1-e^{2\pi ih\theta(N+1)}-1+ e^{2\pi ih\theta}}{1-e^{2\pi ih\theta}}\right| = \frac{1}{N}\left| e^{2\pi ih\theta}\right|\left| \frac{1-e^{2\pi ihN \theta}}{1-e^{2\pi ih\theta}} \right| =\medskip
\end{equation*}
\begin{equation*}
=\frac{1}{N}\left| \frac{1-e^{2\pi ihN \theta}}{1-e^{2\pi ih\theta}} \right|\ .
\end{equation*}
\medskip
Now if we let $\alpha = \pi h \theta$, we can observe that
\begin{eqnarray*}
|1-e^{2 i \alpha}|&=& |1-\cos(2 \alpha)-i \sin(2\alpha)|=\sqrt{(1-\cos(2\alpha))^{2}+\sin^{2}(2\alpha)} \\
& = & \sqrt{2(1-\cos(2\alpha))}=\sqrt{2(1-\cos^{2}\alpha+\sin^{2}\alpha)}=2|\sin \alpha|\ .
\end{eqnarray*}
This is enough to see that \eqref{1.4} holds since
\begin{equation*}
\left| \frac{1}{N} \sum_{n=1} ^{N} e^{2\pi ihn\theta} \right|= \frac{1}{N}\left|\frac{1-e^{2\pi ihN \theta}}{2\sin( \pi h \theta)} \right| \leq \frac{1}{N|\sin (\pi h \theta)|}
\end{equation*}
which tends to zero when $N$ tends to infinity for an integer $h\neq 0$.\\

This result has been found independently by Bohl \cite{Bohl}, Sierpi\'nski \cite{Sierpinski} and Weyl \cite{weyl1} in 1909-1910.\\
The sequence takes the name of Kronecker since this results refines a theorem due to Kronecker showing that the points $e^{in\theta}$ are dense in the unit circle, whenever $\theta$ is an irrational multiple of $\pi$ (Kronecker's approximation theorem).\\
The following result shows that bounds for the discrepancy of the Kronecker sequence are strictly connected with the partial quotients of the continued fraction expansion of $\theta$ and therefore with the Ostrowski expansion. The proof requires some basics about continued fractions and the related Ostrowski expansion and so we do not provide it. We refer to \cite{Niederreiter} for a proof.
\begin{theorem}
The discrepancy $D_N(x_n)$ of the Kronecker sequence satisfies
\begin{equation}
D_N(x_n)<\frac{1}{N}\sum_{\underset{c_k\neq 0}{k=0}}^m(c_k+1)\leq \frac{1}{N}\sum_{k=1}^{m+1} a_i\ ,
\end{equation}
where the $a_i$'s are the partial quotiens of the continued fraction expansion of $\theta$ and the $c_n$'s are the integers in the Ostrowski expansion of $N$.
\end{theorem}
\begin{corollary}
 If $\theta$ is an irrational number such that $\sum_{k=1}^m a_k=\mathcal{O}(m)$, then 
 \begin{equation}
D_N(x_n)=\mathcal{O}\left(\frac{\log(N)}{N}\right)\qquad {\rm for\ all\ } N\geq 2
\end{equation}
\end{corollary}
\end{example}
\begin{example}[Van der Corput sequence]
The van der Corput sequence has been introduced by van der Corput \cite{vdc} in 1935. The main idea to provide bounds for the discrepancy of this sequence is to introduce a suitable numeration system, as the Ostrowski numeration system for the Kronecker sequence. There is one natural numeration system arising from the definition of the van der Corput sequence, namely, the classical system of numeration in base $b$, where $b\geq 2$ is apositive integer.
\begin{proposition}[$b$-ary expansion]
 Let $b\geq 2$ be a positive integer and denote by $\mathbb{Z}_b=\{0,1,\dots, b-1\}$ the least residue system mod $b$. Then every positive integer $n\geq 0$ has a unique expansion
 \begin{equation}\label{b_exp}
  n=\sum_{k=1}^ra_k(n)b^k
 \end{equation}
in base $b$, where $a_k(n)\in\mathbb{Z}_b$ for all $j\geq 0$ and $r=\lfloor \log_b n\rfloor$, where $\lfloor x \rfloor$  is the largest integer not greater than $x$.
\end{proposition}

\begin{definition}[Radical inverse function]\label{radical_inverse}
For an integer $b\geq 2$, consider the $b$-ary expansion in \eqref{b_exp} of $n\in\mathbb{N}$. The function $\varphi_b:\mathbb{N}\longrightarrow [0,1[$ defined as
\begin{equation}
\varphi_b(n) = \sum_{k=0}^r a_k(n) b^{-k-1}
\end{equation}
is called radical inverse function in base $b$.
\end{definition}
In other words, $\varphi_b(n)$ is obtained from $n$ by a symmetric reflection of the expansion \eqref{b_exp} with respect to the decimal point.
\begin{definition}
The sequence $(x_n)_{n\in\mathbb{N}}$ of general term $x_n=\varphi_2(n)$ is called van der Corput sequence.
\end{definition}
\begin{remark}
It is common use to call van der Corput sequences in base $b$, sequences $(x_n)_{n\in\mathbb{N}}$ of general term $x_n=\varphi_b(n)$, where $b$ is a fixed prime number greater than $1$.
\end{remark}
￼￼￼￼￼￼￼￼￼￼￼￼￼￼￼￼￼￼￼￼￼￼￼￼￼￼￼￼￼￼￼￼￼￼￼￼￼￼￼￼￼￼￼￼￼￼￼￼￼￼￼￼￼￼￼￼￼￼￼￼￼￼￼￼￼￼￼￼￼￼￼￼￼￼￼￼￼￼￼￼￼￼￼￼￼￼￼￼￼￼￼￼￼￼￼￼￼￼￼￼￼￼￼￼

\begin{theorem}
The discrepancy $D_N(x_n)$ of the van der Corput sequence $(\varphi_b(n))_{n\in\mathbb{N}}$ satisfies
\begin{equation*}
D_N(x_n)\leq c_b\left( \frac{\log(N+1)}{N}\right)
\end{equation*}
where $c_b>0$ is an absolute constant.
\end{theorem}
In particular, Faure \cite[Theorem 6]{Faure} established the following bounds for arbitrary $b\geq 2$:
\begin{equation*}
 \limsup_{N\to\infty}\frac{ND_N(\varphi_b(n))}{\log N}=\begin{cases}

\frac{b^2}{4(b+1)\log b} & {\rm for\ even}\ b \\
\frac{b-1}{4\log b} & {\rm for\ odd}\ b \ .
\end{cases}
\end{equation*}
An immediate consequence of this result is the fact that the van der Corput sequence which has the best asymptotic behaviour is the one in base $3$.\\

Another generalization of the van der Corput sequence, called \emph{generalized van der Corput sequence in base $b$}, which improves the asymptotic behaviour of the discrepancy, has been introduced in \cite{Faure}. This new family of sequences is defined as the sequence of general term
\begin{equation*}
 x_n=\sum_{k=0}^m \sigma(a_k(n))b^{-k-1} \qquad {\rm for\ all}\ n\geq 0 ,
\end{equation*}
where $\sigma$ is a permutation of $\mathbb{Z}_b$.\\
 In particular, if we denote by $(\varphi_b^\sigma(n))_{n\in\mathbb{N}}$ a generalized van der Corput sequence in base $b$, the following result says that the discrepancy is smaller compared to that of the classical van der Corput sequence $(\varphi_b(n))$ in base $b$.
\begin{theorem}[Corollaire 3, \cite{Faure}]
For every sequence $\varphi_b^\sigma$ we have
\begin{equation*}
D^*_N(\varphi_b^\sigma(n))\leq D_N(\varphi_b^\sigma(n))\leq D_N(\varphi_b(n))
\end{equation*}
\end{theorem}
\end{example}
\subsection{Uniform distribution of sequences of partitions of $[0,1[$}
In this section we consider a method to obtain u.d. sequences of points in the unit interval as sequences of points associated to a special family of sequences of partitions. We present new sequences of partitions recently introduced and studied by Vol\v{c}i\v{c} \cite{volcic}, which are a generalization of the sequences introduced by Kakutani \cite{kakutani} in 1976.
\begin{definition}
Let $(\pi_n)_{n\in\mathbb{N}}$ be a sequence of partitions of $[0,1[$, with
\begin{equation*}
\pi_n=\{[t_i^n, t_{i+1}^n[\, :\, 1\leq i\leq k(n)\}\ .
\end{equation*}
Then $(\pi_n)_{n\in\mathbb{N}}$ is uniformly distributed (u.d.) if for any continuous function $f$ on $[0,1]$
\begin{equation}
\lim_{n\to\infty}\frac{1}{k(n)} \sum_{i=1} ^{k(n)} f(t_i^n)=\int_0^1 f(t)\ dt\ .
\end{equation}
\end{definition}
\begin{remark}
 Given a sequence of partitions $(\pi_n)_{n\in\mathbb{N}}$, it is possible to associate to it a sequence of measures $(\mu_n)_{n\in\mathbb{N}}$ by
 \begin{equation}\label{AssMeas}
\mu_n=\frac{1}{k(n)}\sum_{i=1}^{k(n)}\delta_{t_i^n}\ ,
\end{equation}
where we denote by $\delta_t$ the Dirac measure concentrated at $t$.\\
We note that the uniform distribution property of the sequence of partitions $(\pi_n)_{n\in\mathbb{N}}$ is equivalent to the weak convergence of $(\mu_n)_{n\in\mathbb{N}}$ to the Lebesgue measure $\lambda$ on $[0,1]$. 
\end{remark}
This observation allows us to give the following equivalent definition.
\begin{definition}
A sequence of partitions $(\pi_n)_{n\in\mathbb{N}}$ is u.d. if and only if for every interval $[a,b]\subset [0,1]$
\begin{equation}
 \lim_{n\to\infty}\frac{1}{k(n)}\sum_{i=1}^{k(n)}\mathbf{1}_{[a,b[}(t_i^n)=b-a
\end{equation}
\end{definition}
Equivalently, $(\pi_n)_{n\in\mathbb{N}}$ is u.d.\ if the sequence of discrepancies
\begin{equation}\label{discrPart}
D_n=\sup_{0\leq a<b\leq 1}\bigg|\frac{1}{k(n)}\sum_{i=1}^{k(n)}
\mathbf{1}_{[a,b[}(t_i^{(n)})- (b-a)\bigg|
\end{equation}
tends to $0$ as $n\to\infty$.\\

Let us describe Kakutani's splitting procedure \cite{kakutani}, known as Kakutani's $\alpha$-refinement, which allows to construct a whole class of u.d. sequences of partitions of $[0, 1[$.
\begin{definition}[Kakutani splitting procedure]
Let $\alpha\in\ ]0,1[$ and $\pi=$\linebreak $\{[t_i, t_{i+1}[\, :\, 1\leq i\leq k\}$ be any partition of $[0,1[$, then Kakutani's $\alpha$-refinement of $\pi$ (which will be denoted by $\alpha\pi$) is obtained by splitting only the intervals of $\pi$ having maximal length in two parts, proportionally to $\alpha$ and $1-\alpha$ respectively.
\end{definition}
If we denote by $\alpha^n\pi$ the $\alpha$-refinement of $\alpha^{n-1}\pi$ for every $n\in\mathbb{N}$, the so-called Kakutani's sequence of partitions $(\kappa_n)_{n\in\mathbb{N}} = (\alpha^n\omega)_{n\in\mathbb{N}}$ is obtained by successive $\alpha$-refinements of the trivial partition $\omega = \{[0, 1]\}$.

Observe that for $\alpha=\frac{1}{2}$ Kakutani's sequence of partitions coincides with the binary sequence of partitions.\\
Kakutani proved the following result.
\begin{theorem}
For every $\alpha\in\, ]0,1[$ the sequence of partitions $(\kappa_n)_{n\in\mathbb{N}}$ is u.d..
\end{theorem}
The sequence $(\alpha^n\omega)_{n\in\mathbb{N}}$ and its properties have been investigated by many authors. For example, see \cite{brennan_durrett} and \cite{vanzwet} for a modification of the splitting procedure where the intervals of maximal length are split randomly. For a generalization of Kakutani's splitting procedure in higher dimensions see  Carbone and Vol\v{c}i\v{c} \cite{carbone_volcic}. Another further
generalization of Kakutani's splitting procedure, that we now describe, was introduced by Vol\v{c}i\v{c} \cite{volcic}.

\begin{definition}[$\rho$-refinement]
Let $\rho$ be a non-trivial finite partition of $[0, 1[$. The $\rho$-refinement of a partition $\pi$ of $[0,1[$ (denoted by $\rho\pi$) is the partition obtained by subdividing all intervals of $\pi$ having maximal length positively homotetically to $\rho$.
\end{definition}
\begin{remark}
If $\rho=\{[0,\alpha[, [\alpha,1[\}$ and $\pi=\omega$, then its $\rho$-refinement coincides with Kakutani's $\alpha$-refinement.\\
As in the case of Kakutani's splitting procedure, we can consider the $\rho$-refinement $\rho^2\pi$ of $\rho\pi$ and so on in order to obtain the sequence of partitions $(\rho^n\pi)_{n\in\mathbb{N}}$ defined as follows.
\end{remark}
\begin{definition}
Given a non-trivial finite partition $\rho$ of $[0,1[$, the sequence of $\rho$-refinements $(\rho^n\pi)_{n\in\mathbb{N}}$ of $\pi$ is defined as the sequence of partitions obtained by successive $\rho$-refinements of $\pi$.
\end{definition}

By using arguments from ergodic theory, Vol\v{c}i\v{c} \cite{volcic} proved that the sequence $(\rho^n\omega)_{n\in\mathbb{N}}$ is u.d. for every finite partition $\rho$. \\

A natural problem is to estimate the asymptotic behaviour of the discrepancy of $(\rho^n\pi)_{n\in\mathbb{N}}$.\\
A first result in this direction has been obtained by Carbone \cite{carbone}, who considered the so-called $LS$-sequences for $\rho$ consisting of $L$ intervals of length $\alpha$ and $S$ intervals of length $\alpha^2$, such that $\alpha L+\alpha^2 S=1$.\\
We will provide a complete description of this class of sequences in Chapter 2.\\
The first general results providing upper bounds for the discrepancy of sequences of arbitrary $\rho$-refinements of $\omega$ have been found by Drmota and Infusino \cite{Drmota_Infusino}. The authors consider a new approach based on the analysis of a tree evolution process, namely the Khodak algorithm, where the generation of successive nodes has the same structure as the $\rho$-refinement process.\\
Let us state the results found in \cite{Drmota_Infusino}.
\begin{definition}
Let $p_1,\dots, p_m$ be positive integers. We say that \linebreak $\log\left( \frac{1}{p_1} \right),\dots, \log\left( \frac{1}{p_m} \right)$ are rationally related if there exists a positive number $\Lambda$ such that $\log\left( \frac{1}{p_1} \right),\dots,\log\left( \frac{1}{p_m} \right)$ are integer multiples of $\Lambda$, that is
\begin{equation*}
\log\left( \frac{1}{p_j} \right)=n_j\Lambda\ , \quad {\rm with}\ n_j\in\mathbb{Z}\ {\rm for}\ j=1,\dots, m
\end{equation*}
Without loss of generality, we can assume that $\Lambda $ is as large as possible which is equivalent to assume that $\gcd(n_1,\ldots,n_m)=1$. \\
We say that $\log\left(\frac{1}{p_1}\right),\ldots,\log\left(\frac{1}{p_m}\right)$ are irrationally related if they are not rationally related.
\end{definition}
\begin{theorem}\label{DM1}
Suppose that the lengths of the intervals of a partition $\rho$ of $[0,1[$ are $p_1,\dots, p_m$ and that $\log\left( \frac{1}{p_j} \right)$ for $j=1,\dots, m$ are rationally related. Then there exist a real number $\eta >0$ and an integer $d\geq 0$ such that the discrepancy $D_n$ of $(\rho^n\omega)_{n\in\mathbb{N}}$ is bounded by
\begin{equation*}
D_n=\begin{cases}
\mathcal{O}\left(\frac{(\log k(n))^d}{k(n)^{\eta}}\right) & {\rm if}\ 0<\eta < 1\ , \\
\mathcal{O}\left(\frac{(\log k(n))^{d+1}}{k(n)^{\eta}}\right) & {\rm if}\ \eta = 1 \ ,\\
 \mathcal{O}\left(\frac{1}{k(n)}\right) & {\rm if}\ \eta > 1\ .
\end{cases}
\end{equation*}
\end{theorem}
The next theorem requires the following
\begin{definition}
 Every irrational number $x$ has an approximation constant $c(x)$ defined by
\begin{equation*}
c(x)=\liminf_{q\to \infty}q|qx-p|\ , 
\end{equation*}
where $p$ is the nearest integer to $qx$. \\
Moreover $x$ is said to be badly approximable if $c(x)>0$.
\end{definition}
\begin{theorem}\label{DM2}
Suppose that the lengths of the intervals of a partition $\rho$ of $[0,1[$ are $p$ and $q = 1 - p$. If $\frac{\log p}{\log q}\notin \mathbb{Q}$ is badly approximable, then the discrepancy $D_n$ of $(\rho^n\omega)$ is bounded by
\begin{equation*}
D_n=\mathcal{O}\left(\left(\frac{\log\log (k(n))}{\log (k(n))}\right)^{\frac{1}{4}}\right)\ ,\qquad {\rm}\ n\to \infty\ .
\end{equation*}
￼Furthermore, if $p, q$ are algebraic numbers then \begin{equation*}
￼D_n = \mathcal{O}\left(\left(\frac{\log\log (k(n))}{\log(k(n))}\right)^\kappa\right), \qquad {\rm}\ n\to \infty\\ ,
\end{equation*}
where $\kappa$ is an effectively computable positive real constant.
\end{theorem}
In this second case, we can observe that weaker upper
bounds for the discrepancy are obtained, since they depend heavily on Diophantine approximation properties of the ratio $\frac{\log p}{\log q}$. \\
Finally, the authors proved bounds for the elementary discrepancy
of the sequences of partitions constructed for a certain class of fractals.\\

A second interesting problem arising in this context concerns the uniform distribution of sequences of partitions, when the starting partition $\pi$ is not the trivial partition $\omega$.\\
As pointed out in \cite{volcic}, if we start from an arbitrary initial partition $\pi$, then the sequence of partitions $(\rho^n\pi)_{n\in\mathbb{N}}$ is in general not u.d.. 
\begin{remark}
Let us consider 
\begin{equation*}
 \pi=\left\{\left[0, \frac 25\right[, \left[\frac 25, 1\right[\right\} \quad {\rm and}\quad \rho=\left\{\left[0, \frac 12\right[, \left[\frac 12, 1\right[\right\}\ .
\end{equation*}
It is clear that the $\rho-$refinement operates alternatively on $\left[\frac 25, 1\right[$ and $\left[0, \frac 25\right[$. So, if we consider the sequence of measures $(\mu_n)_{n\in\mathbb{N}}$ associated to $(\rho^n\pi)_{n\in\mathbb{N}}$, then for any measurable set $E\subset[0,1[$ the subsequence $(\mu_{2n})_{n\in\mathbb{N}}$ converges to 
\begin{equation*}
 \lambda_{2n}(E)=\frac 54\cdot\lambda\left(E\cap\left[0, \frac 25\right[\right)
+\frac 56\cdot\lambda\left(E\cap\left[\frac 25, 1\right[\right)\ ,
\end{equation*}
while the subsequence $(\mu_{2n+1})_{n\in\mathbb{N}}$ converges to
\begin{equation*}
 \lambda_{2n+1}(E)=\frac 56\cdot\lambda\left(E\cap\left[0, \frac 25\right[\right)
+\frac {10}{9}\cdot\lambda\left(E\cap\left[\frac 25, 1\right[\right)\ .
\end{equation*}
Hence, $(\mu_n)_{n\in\mathbb{N}}$ does not converge and consequently $(\rho^n\pi)_{n\in\mathbb{N}}$ is not u.d..\\
It is worthwhile noticing that the same problem arises even in the simplest case of Kakutani's splitting procedure.\\ So, it is essential to find sufficient conditions on $\pi$ in order to guarantee the uniform distribution of $(\alpha^n\pi)_{n\in\mathbb{N}}$ or more in general of $(\rho^n\pi)_{n\in\mathbb{N}}$.
\end{remark}
The solution to this problem has been found by Aistleitner and Hofer \cite{aistleitner_hofer}, who gave necessary and sufficient conditions on $\pi$ and $\rho$ under which the sequence $(\rho^n\pi)_{n\in\mathbb{N}}$ is uniformly distributed.\\
The authors proved the following result, using methods introduced in \cite{Drmota_Infusino}.
\begin{theorem}\label{AH}
 Let $\rho$ be a partition of $[0, 1[$ consisting of $m \geq 2$ intervals of lengths $p_1,\dots, p_m$, and let $\pi$ be an initial partition of $[0, 1[$ consisting of $l \geq 2$ intervals of lengths $\alpha_1,\dots,\alpha_l$. Then the sequence $(\rho^n\pi)_{n\in\mathbb{N}}$ is uniformly distributed if and only if one of the following conditions is satisfied:
 \begin{enumerate}
\item the real numbers $\log\left( \frac{1}{p_1} \right),\dots, \log\left( \frac{1}{p_m} \right)$ are irrationally related
\item the real numbers $\log\left( \frac{1}{p_1} \right),\dots, \log\left( \frac{1}{p_m} \right)$ are rationally related with parameter $\Lambda$ and
\begin{equation*}
\alpha_i=ce^{v_i\Lambda},\ c\in\mathbb{R}^+,\ v_i\in\mathbb{Z}, i=1\dots l\ .
\end{equation*}
\end{enumerate}
\end{theorem}
\begin{remark}
 Condition 2. includes the special case when the initial partition $\pi$ consists of intervals having the same length, and in particular the case when the initial partition is the trivial partition $\omega$ and the corresponding sequence of partition is a Kakutani sequence.
\end{remark}
In particular, the following corollary gives conditions under which the Kakutani sequence of partitions of a particular initial sequence $\pi$ is u.d..
\begin{corollary}
 Let $\rho=\{[0,p[,[p,1[\}$ and $\pi=\{[0,\alpha[,[\alpha,1[\}$. Then $(\rho^n\pi)_{n\in\mathbb{N}}$ is u.d. if and only if one of the following conditions is satisfied:
 \begin{enumerate}
  \item $\frac{\log(p)}{\log(1-p)}$ is irrational
  \item $\log\left(\frac{1}{p}\right)$ and $\log\left(\frac{1}{1-p}\right)$ are rationally related with parameter $\Lambda$ and $\alpha=\frac{1}{e^{k\Lambda}+1}$ for $k \in\mathbb{Z}$.
 \end{enumerate}
\end{corollary}
Furthermore, in \cite{aistleitner_hofer} a description of the general asymptotic behaviour of a sequence of partitions, even in cases not covered by Theorem \ref{AH}, is given. More precisely, the authors proved the following
\begin{theorem}
Assume that neither condition 1. nor condition 2. of Theorem \ref{AH} is satisfied.
Then for any interval $A =[a, b]\subset [0, 1]$ which is completely contained in the $i$-th interval of the initial partition $\pi$ for some $i$, $1 \leq i \leq l$, we have
\begin{equation*}
\limsup_{n\to \infty}\frac{1}{k(n)}\sum_{j=1}^{k(n)}\mathbf{1}_{[a,b]}(t_j^n)=c_1(b-a)\ ,
\end{equation*}
\begin{equation*}
\liminf_{n\to \infty}\frac{1}{k(n)}\sum_{j=1}^{k(n)}\mathbf{1}_{[a,b]}(t_j^n)=c_2(b-a)\ ,
\end{equation*}
where
\begin{equation*}
c_1=\left( \sum_{j=1}^l\alpha_j{\rm exp}\left(-\Lambda \left\lbrace \frac{\log \alpha_j- \log \alpha_i}{\Lambda}\right\rbrace\right) \right)^{-1}>1\ ,
\end{equation*}
\begin{equation*}
c_2=\left( \sum_{j=1}^l\alpha_j{\rm exp}\left(\Lambda \left\lbrace \frac{\log \alpha_i - \log \alpha_j}{\Lambda}\right\rbrace\right) \right)^{-1}<1\ ,
\end{equation*}
are constants depending on $i$.
\end{theorem}
Hence, if either condition 1. or 2. does not hold, then the sequence $(\rho^n\pi)_{n\in\mathbb{N}}$ is not u.d..\\
Now we would like to discuss how to associate to a u.d.\! sequence of partitions a u.d.\! sequence of points. \\
The problem is the following: Assume that $(\pi_n)_{n\in\mathbb{N}}$ is a u.d.\! sequence of partitions of $[0,1[$. Is it possible to rearrange the points $t_i^n$ determining the partitions $\pi_n$, for $1\leq i\leq k(n)$, in order to get a u.d.\! sequence of points?\\ Clearly, we can reorder points in several different ways and a natural restriction is that we first reorder all the points determining $\pi_1$ then those defining $\pi_2$, and so on. A reordering of this kind is called sequential reordering.\\
We introduce a result proved by Vol\v{c}i\v{c} in \cite{volcic}, where a probabilistic answer to this problem is given.

\begin{definition}
If $(\pi_n)_{n\in\mathbb{N}}$ is a u.d.\! sequence of partitions of $[0,1[$ with
\begin{equation*}
 \pi_n=\{[t_{i-1}^n, t_i^n[: 1\leq i\leq k(n)\}\ ,
\end{equation*}
the sequential random reordering of the points $t_i^n$ is a sequence $(\varphi_m)_{m\in\mathbb{N}}$ of consecutive blocks of random variables. The $n$-th block consists of $k(n)$ random variables which have the same law and represent the drawing, without replacement, from the sample space $W_n=\left\{t_1^n,\ldots,t_{k(n)}^n\right\}$ where each singleton has probability $\frac{1}{k(n)}$.
\end{definition}

Denote by $S_n$ the set of all permutations on $W_n$, endowed with the probability $P(\tau_n)=\frac{1}{k(n)!}$ with $\tau_n\in S_n$.

Any sequential random reordering of $(\pi_n)_{n\in\mathbb{N}}$ corresponds to a random selection of $\tau_n\in S_n$ for each $n\in\mathbb{N}$. The permutation $\tau_n\in S_n$ identifies the reordered $k(n)$-tuple of random variables $\varphi_i$ with $K(n-1)\leq i\leq K(n)$, where $K(n)=\sum\limits_{i=1}^nk(i)$. Therefore, the set of all sequential random reorderings can be endowed with the natural product probability on the space $S=\prod\limits_{n=1}^\infty S_n$.

\begin{theorem} 
If $(\pi_n)_{n\in\mathbb{N}}$ is a u.d.\! sequence of partitions of $[0,1[$, then the sequential random reordering of the points $t_i^n$ defining them is almost surely a u.d.\! sequence of points in $[0,1[$.
\end{theorem}

\subsection{Multidimensional case}\label{multi}
In this section we collect results about uniform distribution and discrepancy in the multidimensional case. We will see that some of the results stated in the first subsection can be naturally extended in the $s$-dimensional unit cube $[0,1[^s$, while some others require a deeper analysis, as in the case of the estimation of the discrepancy.\\
In fact, apart from the case $s=1$, it is in general very difficult to give good quantitative estimates concerning the distribution of multidimensional sequences and several problems are still open.\\

Let $s$ be an integer with $s\geq 2$. We use the following notation: for two points $\mathbf{a},\mathbf{b} \in [0,1[^s$  we write $\mathbf{a} \leq \mathbf{b}$ and $\mathbf{a} < \mathbf{b}$ if the corresponding inequalities hold in each coordinate; furthermore, we write $[\mathbf{a}, \mathbf{b}[$ for the set $\{\mathbf{x} \in [0,1[^s: ~\mathbf{a} \leq \mathbf{x} < \mathbf{b}\}$, and we call such a set an $s$-dimensional interval. Moreover we denote by $\mathbf{1}_{J}$ the indicator function of the set $J \subseteq [0,1[^s$ and by $\lambda_s$ the $s$-dimensional Lebesgue measure (for short we write $\lambda$ instead of $\lambda_1$). Note that vectors will be written in bold fonts and we write $\mathbf{0}$ for the $s$-dimensional vector $(0,\dots,0)$.\\
\begin{definition}
A sequence $(\mathbf{x}_n)_{n \in \mathbb{N}}$ of points in $[0,1[^s$ is called uniformly distributed (u.d.) if
\begin{equation}\label{multiunif}
 \lim_{N \rightarrow \infty} \frac{1}{N}\sum_{n = 1}^N \mathbf{1}_{[\mathbf{a}, \mathbf{b}[} (\mathbf{x}_n) = \lambda_s([\mathbf{a},\mathbf{b}[)
\end{equation}
for all $s$-dimensional intervals $[\mathbf{a}, \mathbf{b}[\, \subseteq [0,1)^s$.
\end{definition}

In particular, we can introduce the following function
\begin{equation}\label{adf}
 g(x,y) = \lim_{N \rightarrow \infty} \frac{1}{N}  \sum_{n = 1}^{N} \mathbf{1}_{[0,x[ \times [0,y[}(x_n,y_n),
\end{equation}
 which will be extensively used in the last chapter. 
We call $g$ the asymptotic distribution function (a.d.f.) of a sequence $(x_n,y_n)_{n \in\mathbb{N}}$ in $[0,1[^2$ if \eqref{adf} holds for every point $(x,y)$ of continuity of $g$. Moreover, in \cite{fial} Fialov\'a and Strauch consider
\begin{equation*}
 \limsup_{N \rightarrow \infty} \frac{1}{N} \sum_{n = 1}^N f(x_n, y_n),
\end{equation*}
where $(x_n)_{n > 1}, (y_n)_{n > 1}$ are u.d.\ sequences in the unit interval and $f$ is a continuous function on $[0,1[^2$, see also \cite{Porubski_Strauch}. In this case the a.d.f.\ $g$ of $(x_n,y_n)_{n > 1}$ is called copula.\\ 

Weyl was the first who extended the uniform distribution property to the multidimensional case. His classical results \cite{weyl1,weyl2} have the following multidimensional version.
\begin{theorem}[Weyl's Theorem]\label{multiweyl}
A sequence $(\mathbf{x}_n)_{n \in \mathbb{N}}$ of points in $[0,1[^s$ is u.d.\ if and only if for every (real or complex-valued) continuous function $f$ on $[0,1[^s$ the relation
\begin{equation*}
  \lim_{N \rightarrow \infty} \frac{1}{N} \sum_{n = 1}^N f(\mathbf{x}_n) = \int_{[0,1]^s} f(\mathbf{x}) d\mathbf{x}
\end{equation*}
holds.
\end{theorem}

In particular, if $g$ is a copula, then we can write
\begin{equation}\label{inte}
 \lim_{N \rightarrow \infty} \frac{1}{N} \sum_{n = 1}^N f(x_n, y_n) = \int_0^1 \int_0^1 f(x,y) dg(x,y).
\end{equation}
Let $\mathbf{x}, \mathbf{y}\in \mathbb{R}^s$ and let us denote by $\mathbf{x}\cdotp \mathbf{y}$ the usual inner product in $\mathbb{R}^s$, i.e. $\mathbf{x}\cdotp \mathbf{y}=\sum\limits_{i=1}^sx_iy_i$. Then we can give the generalization of the Weyl's Criterion presented in the first section.
\begin{theorem}[Weyl's Criterion]
The sequence $(\mathbf{x}_n)_{n \in \mathbb{N}}$ in $[0,1[^s$ is u.d.\! if and only if
\begin{equation*}
\lim_{N\rightarrow{\infty}}\frac{1}{N}\sum_{n=1}^{N}{e^{2\pi ih\cdotp \mathbf{x}_n}}=0
\end{equation*}
for all non-zero integer lattice points $\mathbf{h}\in\mathbb{Z}^s\setminus \mathbf{0}$.
\end{theorem}

As in the one-dimensional case, a measure for the quality of the uniform distribution of a sequence is given by the discrepancy.
\begin{definition}[Discrepancy]
Let $\boldsymbol \omega=\{\mathbf{x}_1,\ldots,\mathbf{x}_N\}$ be a finite set of points in $[0,1[^s$. Then the discrepancy of $\boldsymbol\omega$ is defined as
\begin{equation}
 D_N(\boldsymbol \omega)=\sup_{J\in [0,1[^s}\Bigg|\frac{1}{N}\sum_{n=1}^N\mathbf{1}_{J}(\mathbf{x}_n)-\lambda_s(J)\Bigg|.
\end{equation}
\end{definition}
If $(\mathbf{x}_n)_{n \in \mathbb{N}}$ is an infinite sequence of points, we associate to it the sequence of discrepancies $D_N(\{\mathbf{x}_1, \mathbf{x}_2, \dots, \mathbf{x}_N\})$. So, the symbol $D_N(\mathbf{x}_n)$ denotes the discrepancy of the initial segment $\{\mathbf{x}_1, \mathbf{x}_2, \dots, \mathbf{x}_N\}$ of the sequence. 
\begin{theorem}
A sequence $(\mathbf{x}_n)_{n \in \mathbb{N}}$ of points in $[0,1[^s$ is u.d.\! if and only if
\begin{equation}
 \lim_{N\to\infty}D_N(\mathbf{x}_n)=0\ .
\end{equation}
\end{theorem}
If we restrict to subintervals of the form $J=[0,a_1[\times\cdots\times [0,a_s[$ with $0< a_i\leq 1$, we get the obvious definition of multidimensional star discrepancy.
\begin{definition}[Star discrepancy]
Let $\boldsymbol \omega=\{\mathbf{x}_1,\ldots,\mathbf{x}_N\}$ be a finite set of points in $[0,1[^s$. Then the star discrepancy of $\boldsymbol\omega$ is defined as
\begin{equation}
 D^*_N(\boldsymbol\omega)=\sup_{J}\Bigg|\frac{1}{N}\sum_{n=1}^N\mathbf{1}_{J}(\mathbf{x}_n)-\lambda_s(J)\Bigg|,
\end{equation}
where the supremum is taken over all subintervals $J\subset [0,1[^s$ of the form  $J=[0,a_1[\times\cdots\times [0,a_s[$ with $0< a_i\leq 1$.
\end{definition}
The relation between discrepancy $D_N$ and star discrepancy $D_N^*$ in several dimensions is very much similar to the one presented for $s=1$.
\begin{theorem}
For any sequence $(\mathbf{x}_n)_{n \in \mathbb{N}}$ of points in $[0,1[^s$ we have
\begin{equation*}
D^*_N(\mathbf{x}_n)\leq D_N(\mathbf{x}_n)\leq 2^s D_N^*(\mathbf{x}_n)\ .
\end{equation*}
\end{theorem}

In higher dimensions it is possible to define other forms of discrepancies. One of the most used to quantify the convergence in \eqref{multiunif} is the isotropic discrepancy, where the supremum is taken over all convex sets in the unit cube $[0,1[^s$ (see e.g. \cite{Chen_Travaglini} for estimates of the discrepancy of point sets with respect to closed convex polygons). The origin of this new definition is strictly related to the integration domain in the right hand side of \eqref{multiunif}.\\
We refer to \cite{Matousek} for a complete introduction to geometric discrepancy, i.e. discrepancy studied for classes of geometric figures other than the axis-parallel rectangles, such as the set of all balls, or the set of all polygons or polytopes, and so on. \\



In the sequel we will present known lower bounds for the discrepancy.
\begin{proposition}
For any finite set $\boldsymbol\omega=\{\mathbf{x}_1,\ldots,\mathbf{x}_N\}$ of points in $[0,1[^s$ we have that
\begin{equation}
 \frac{1}{N}\leq D_N(\boldsymbol\omega)\leq 1\ .
\end{equation}
\end{proposition}
As we have already pointed out, the lower bound is exactly attained only in the one-dimensional case. In fact, the following theorem, due to Roth \cite{Roth}, illustrates that in the higher-dimensional case examples of sequences having discrepancy equal to $\frac{1}{N}$ cannot exist.
\begin{theorem}[Roth's Theorem]\label{roth}
Let $s\geq 2$. Then the discrepancy $D_N(\boldsymbol\omega)$ of the point set $\boldsymbol\omega=\{\mathbf{x}_1,\ldots,\mathbf{x}_N\}\subset [0,1[^s$ is bounded from below by
\begin{equation}\label{BoundFinSeq}
D_N(\boldsymbol\omega)\geq c_s\left(\frac{(\log N)^{\frac{s-1}{2}}}{N}\right),
\end{equation}
where $c_s>0$ is an absolute constant given by $c_s=\frac{1}{2^{4s}((s-1)\log 2)^\frac{s-1}{2}}$.
\end{theorem}
This bound is the best known result for $s>3$ and for $s=3$ it has been slightly sharpened by Beck \cite{Beck}.\\
The first finite sequence of $N\geq 2$ points in $[0,1[^s$ showing asymptotic behaviour of the discrepancy of order $\mathcal{O}\left(\frac{(\log N)^{s-1}}{N}\right)$ has been the Hammersley point set \cite{Hammersley}, for which the constant depends only on the dimension $s$.\\
We will show how to construct this  point set in the next paragraph.\\
A stronger form of Roth's Theorem, but with a worse constant, is the following theorem proved by Beck and Chen \cite{Beck_Chen}.
\begin{theorem}
 Let $s\geq 2$ and $\boldsymbol\omega=\{\mathbf{x}_1,\ldots,\mathbf{x}_N\}$ an arbitrary finite sequence of $N>1$ points in $ \mathbb{R}^s$. Then there exists an $s$-dimensional cube $Q\subset [0,1[^s$ with sides parallel to the axes satisfying
\begin{equation*}
\left| \frac{1}{N}\sum_{i=1}^N \mathbf{1}(\mathbf{x}_i)-\lambda_s(Q)\right|\geq \frac{c_s(\log N)^{\frac{s-1}{2}}-d_s}{N}\ ,
\end{equation*}
where
\begin{equation*}
 c_s=\frac{(\log 2)^{\frac{2-s}{4}}2^{-\frac{7}{2}}3^{\frac{s}{2}}\pi^{-\frac{s}{4}}}{((s-1)!)^{\frac{1}{2}}s^{\frac{s}{4}}(2s+1)^{\frac{s-1}{2}}(6s+1)^{\frac{s}{2}}}
\end{equation*}
and
\begin{equation*}
 d_s=s^{\frac{1}{2}}2^{s+2}3^{\frac{s-1}{2}}\ .
\end{equation*}
\end{theorem}
It is clear that Roth's Theorem follows from this one and it is an interesting question whether the converse is also true. This has been proved only in the 2-dimensional case by Ruzsa \cite{Rusza}, but in general only weaker results have been found.\\

For infinite sequences in $[0,1[^s$ we only have a long-standing conjecture, stating that for every dimension $s$ there exists a constant $c_s$ such that for any infinite sequence $(\mathbf{x}_n)_{n\in\mathbb{N}}$ in $[0,1[^s$
\begin{equation}\label{BoundInfSeq}
 D_N(\mathbf{x}_n)\geq c_s \left(\frac{(\log N)^s}{N}\right)
\end{equation}
holds, for infinitely many $N$ and with a positive constant depending only on the dimension $s$.\\ It is a widely held belief that the orders of magnitude in \eqref{BoundFinSeq} and \eqref{BoundInfSeq} are best possible even if this is only known for \eqref{BoundFinSeq} in the case $s=1,2$ and for \eqref{BoundInfSeq} in the case $s=1$ (see Theorem \ref{SchmidtThm}).\\

Sequences of points in $[0,1[^s$ having discrepancy of order $\mathcal{O}\left(\frac{(\log N)^s}{N}\right)$ are called low discrepancy sequences.\\
In the last part of this section we will highlight the important role played by low-discrepancy sequences in Quasi-Monte Carlo integration and other applications to numerical analysis. \\

In the following paragraph we show some of the most well-known examples of multidimensional low discrepancy sequences and point sets, namely Kronecker sequence, Hammersley point set and Halton sequence. The first one is the obvious generalization of the one-dimensional Kronecker sequence, while the second and the third are a finite and infinite generalization of the van der Corput sequence. 

\begin{example}[Kronecker sequence]
Let $\theta_1,\dots,\theta_s$ be distinct irrational numbers. The $s$-dimensional sequence
\begin{equation*}
(\mathbf{x}_n)_{n\in\mathbb{N}}=(\{n\theta_1\},\dots,\{n\theta_s\})_{n\in\mathbb{N}}
\end{equation*}
is called Kronecker sequence.\\
It has been proved by Weyl \cite{weyl2} that this sequence is uniformly distributed if and only if $1,\theta_1,\dots,\theta_s$ are linearly independent over $\mathbb{Q}$. Indeed it is a direct application of Weyl's criterion \eqref{1.4} and the proof is exactly the same as in the one-dimensional case.\\
Concerning the problem of estimating the discrepancy of this sequence, the approach used is completely different from that in the one-dimensional case, since there is no canonical generalization of Ostrowski numeration system to higher dimensions. This is essetially due to the fact that there is no canonical generalization of the Euclidian algorithm.
To cope with the lack of a satisfactory tool replacing continued fractions, several
approaches are possible, for instance a classical one is that of best simultaneous approximations, introduced by Lagarias \cite{Lagarias1, Lagarias2} in 1982. \\
So almost every Kronecker sequence is an \lq\lq almost\rq\rq\ low-discrepancy sequence.
\end{example}
\begin{example}[Hammersley point set]
For a given $N\in\mathbb{N}$, an $s$-dimensional Hammersley point set $(\mathbf{x}_n)$ of size $N$ in $[0,1[^s$ is defined by 
\begin{equation*}
\mathbf{x}_n=\left(\frac{n}{N}, \varphi_{b_1}(n),\ldots, \varphi_{b_{s-1}}(n)\right)\qquad n=0,1,\dots,N-1\ ,
\end{equation*}
where $b_1,\ldots, b_{s-1}$ are given coprime positive integers and $\varphi_{b_i}$ is the radical inverse function in base $b_i$ defined in Definition \ref{radical_inverse}.\\

It is immediate to see that this is an $s$-dimensional generalization of the van der Corput sequence. It has been introduced by Hammersley \cite{Hammersley} in his survey paper on Monte Carlo methods in order to achieve a smaller error bound in the estimate of multidimensional integrals by means of finite sums. This problem will be discussed in detail in the next section.\\

Halton \cite{Halton} proved that the Hammersley sequence has a discrepancy of order $\mathcal{O}\left(\frac{(\log N)^{s-1}}{N}\right)$. We will see in the next paragraph that this result directly follows from the following lemma applied to the $(s-1)$-dimensional Halton sequence.\\
The following lemma turns out to be a very useful result to contruct different low-discrepancy point sets in the $s$-dimensional unit cube, as soon as we know that the sequence $(\mathbf{x}_n)_{n\in\mathbb{N}}$ is low-discrepancy in $[0,1[^{s-1}$.
\begin{lemma}\label{discr_point_set}
 For $s\geq 2$, let $\boldsymbol\omega$ be an arbitrary sequence of points $(\mathbf{x}_n)_{n\in\mathbb{N}}$ in $[0,1[^{s-1}$. For $N\geq 1$, let $P$ be the point set consisting of $\left(\frac{n}{N},\mathbf{x}_n\right)\in[0,1[^s$ for $n=0,1,\dots,N-1$. Then
\begin{equation*}
ND_N^*(P)\leq \max_{1\leq M\leq N}MD_M^*(\boldsymbol\omega)+1\ .
\end{equation*}
\end{lemma}

\begin{remark}
A Hammersley point set is a finite set of size $N$ which cannot be extended to an infinite sequence. In fact, if we want to increase the number of points by one we have to compute again all the $N$ point of the form $\frac{n}{N}$, losing the previously calculated values.\\
It is evident that this is a quite tedious work, especially for applications.\\
So now we introduce the second generalization to higher dimension of the van der Corput sequence. This new construction gives rise to an infinite sequence of points in the unit cube $[0,1[^s$.
\end{remark}
\end{example}
\begin{example}[Halton sequence]
Let $b_1,\ldots, b_s\geq 2$ be pairwise coprime integers. Then the $s-$dimensional Halton sequence $(\mathbf{x}_n)_{n\in\mathbb{N}}$ in $[0,1[^s$ is the sequence of points of general term 
\begin{equation*}
 \mathbf{x}_n=\left(\varphi_{b_1}(n),\ldots, \varphi_{b_s}(n)\right)\ ,
\end{equation*}
where $\varphi_{b_i}$ is the radical inverse function in base $b_i$ defined in \ref{radical_inverse}.\\

For $s=1$ and $b_1=2$ we just get the van der Corput sequence.\\
The following result is due to Halton \cite{Halton} and shows that the Halton sequence has low-discrepancy.
\begin{theorem}\label{Halton_discr}
 Let $(\mathbf{x}_n)_{n\in\mathbb{N}}$ be the Halton sequence in the pairwise relatively prime bases $b_1,\dots,b_s$. Then
 \begin{equation*}
 D_N^*(\mathbf{x}_n)<\frac{s}{N}+\frac{1}{N}\prod_{i=1}^s\left(\frac{b_i-1}{2\log b_i}\log N+\frac{b_i+1}{2}\right)\quad {\rm for\ all\ } N\geq 1\ .
 \end{equation*}
\end{theorem}
From this theorem and Lemma \ref{discr_point_set}, we get the following result.
\begin{corollary}
If $P$ is the $s$-dimensional Hammersley point set of size $N$ in $I^s$, then
\begin{equation*}
D_N^*(P)<\frac{s}{N}+\frac{1}{N}\prod_{i=1}^s\left(\frac{b_i-1}{2\log b_i}\log N+\frac{b_i+1}{2}\right)\ .
\end{equation*}
\end{corollary}
\end{example}
We end this paragraph with two important remarks concerning the discrepancy of the Halton sequences.
\begin{remark}
Apart from the requirement of coprimality in the choice of the bases $b_1,\dots, b_s$ in the definition of the Halton sequence, it is also possible to optimize this choice in order to get a better discrepancy bound. \\
In fact, as we have shown in Theorem \ref{Halton_discr}, an upper bound for the discrepancy of the Halton sequence is given by
\begin{equation*}
D_N^*(\mathbf{x}_n)\leq A(b_1,\dots,b_s)N^{-1}(\log N)^s+\mathcal{O}(N^{-1}(\log N)^{s-1})\ ,
\end{equation*}
where the coefficient of the main term is
\begin{equation*}
A(b_1,\dots,b_s)=\prod_{i=1}^s\frac{b_i-1}{2\log b_i}\ .
\end{equation*}
So in order to get a better upper bound, we need to minimize this coefficient. It is evident that it is minimal if we let $b_1,\dots,b_s$ be the first $s$ primes $p_1=2, p_2=3,\dots , p_s$.\\
Of course the same assumption can be made in the case of the Hammersley point set and it gives again a better bound for the discrepancy.\\
If we denote by $A_s$ the term $A(p_1,\dots,p_s)$ and consider the Prime Number Theorem, we obtain
\begin{equation*}
\lim_{s\to\infty}\frac{\log A_s}{s\log s}=1\ .
\end{equation*}
Thus $A_s$ increases superexponentially as $s\to\infty$ and this means that Halton sequences and Hammersley point sets are actually useful for applications only for small dimensions $s$.\\
This phenomenon, generically referred to as \lq \lq curse of dimensionality\rq\rq, essentially implies that when the dimension increases, the Halton sequence does not show a good asymptotic behaviour. \\
To cope with this problem, it has been introduced, first by Spanier \cite{Spanier} and later considered by several authors (see e.g. \cite{Aistleitner_Hofer2, Hellekalek, HK, HKLF, Niederreiter_Winterhof, Okten}), the notion of \emph{hybrid sequences}. Roughly speaking, they are $s$-dimensional sequences obtained by concatenating $d$-dimensional low-discrepancy sequences with $s-d$-dimensional random sequences. We will discuss the advantages of considering these new sequences in the last part of this section.

\end{remark}
\begin{remark}
As we have already pointed out, choosing the first $s$ primes instead of arbitrary coprime bases improves the bound for the discrepancy of the Halton sequence. Nevertheless, numerical calculations showed strong correlation between components. Correlations between radical inverse functions with different bases used for different dimensions reflect on poorly distributed two-dimensional projections. The first investigations on this problem are due to Braaten and Weller \cite{Braaten_Weller}.\\
For instance, if we consider the eight-dimensional Halton sequence, we know that the last two coordinates are defined by the digit espansion in base $17$ and $19$, respectively. As the following picture taken from \cite{Vandewoestyne_Ronald} shows, there is a strong correlation between the seventh and eighth coordinate.\\
\begin{figure}[h!]
\centering
\includegraphics[scale=4]{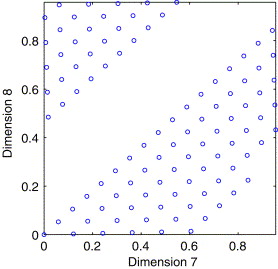}
\caption{Projection of the first 100 points of the Halton sequence onto the seventh and the eighth coordinates}
\end{figure}\\
The poor two-dimensional projections can be explained by the fact that the difference between the two primes bases $17$ and $19$ corresponding to the dimensions $7$ and $8$ is very small compared to the base size and so the first $18$ points in the two bases have the same expression.\\ One could think that one way to avoid this problem is to drop the first $20$ entries, but this is not very convenient since the same problem appears also for bigger twin primes. In order to break correlations, the most convenient solution is given by a scramble of the original Halton sequence. It is a randomized version of the Halton sequence, in the same way as it has been done to obtain the generalized van der Corput sequence in one dimension, based on digit permutations.\\
An alternative approach suggested in \cite{CMW} is to find an optimal Halton sequence within a family of scrambled sequences.\\
Other methods are considered in literature to randomize the Halton sequence and they make use of the description of the sequence by means of the von Neumann-Kakutani transformation, see e.g. \cite{Okten2, Schlier, Wang_Hickernell}.
We will analyze this transformation and its connection with u.d. sequences in the next chapter.\\
As discussed above, even if Halton sequences are low-discrepancy and therefore suitable for applications, it is sometimes better to consider point sets and sequences whose discrepancy bounds have much smaller constants.\\
In view of this consideration, point sets called $(t,m,s)$-nets and sequences called $(t,s)$-sequences have been introduced by Niederreiter \cite{Niederreiter2}.
\end{remark}

\subsection{Quasi-Monte Carlo methods}
We end this section by presenting the most interesting application of low-discrepancy sequences.\\
Roughly speaking, quasi-Monte Carlo methods are deterministic versions of Monte Carlo methods.\\ 
This last technique consists in the numerical approximation of integrals of a function $f$ defined on the $s$-dimensional unit cube $[0,1]^s$ by the average function value at the $N$ quadrature points belonging to $[0,1[^s$. More precisely, the crude Monte Carlo approximation for the integral of a function $f(\mathbf{x})$ on the unit interval $[0, 1]^s$ is
\begin{equation*}
\int_{[0,1]^s} f(\mathbf{x})d\mathbf{x}\approx \frac{1}{N}\sum_{n=1}^N f(\mathbf{x}_n)\ ,
\end{equation*}
where $\mathbf{x}_1,\dots,\mathbf{x}_N$ are random points from $[0,1[^s$ obtained by performing $N$ independent and uniformly distributed trials.\\
The first important question arising when implementing an approximation technique concernes the type of convergence and the speed of convergence.\\
For the crude Monte Carlo method, the strong law of large numbers guarantees that the numerical integration procedure converges almost surely.\\ Moreover, it follows from the central limit theorem that the integration error is
\begin{equation*}
\left|\frac{1}{N}\sum_{n=1}^N f(\mathbf{x}_n)-\int_{[0,1]^s} f(\mathbf{x})d\mathbf{x}\right|=\mathcal{O}\left(N^{-1/2}\right)\ .
\end{equation*}
It is important to note that this order of magnitude does not depend on the dimension $s$, as it happens for classical numerical integration where the error bound is of order $\mathcal{O}\left(N^{-2/s}\right)$. So the Monte Carlo method for numerical integration allows us to overcome the curse of dimensionality.\\
It is evident that for the practical implementation of the Monte Carlo method the fundamental question is how to produce a random sample. Even if there are some techniques such as tables of \lq\lq random\rq\rq\ numbers or physical devices for generating random numbers such as white noise, there are some deficiencies in the Monte Carlo method minating is usefulness. As remarked in \cite{Niederreiter}, the main deficiencies of the Monte Carlo method are:
\begin{itemize}
\item there are only probabilistic error bounds,
\item the regularity of the integrand is not reflected,
\item generating random samples is difficult.
\end{itemize} 
To cope with these three problems one should select the sample points according to a deterministic scheme that is well suited for the problem at hand. The quasi-Monte Carlo method is based on this idea.\\
Recall that Weyl's Theorem \ref{multiweyl} suggests the possibility of a numerical approximation of the integral of a continuous function $f$  defined on $[0,1]^s$ by the average function value at $N$ u.d. points in $[0,1[^s$, i.e.
\begin{equation*}
  \lim_{N \rightarrow \infty} \frac{1}{N} \sum_{n = 1}^N f(\mathbf{x}_n) = \int_{[0,1[^s} f(\mathbf{x}) d\mathbf{x}\ ,
\end{equation*}
where $\mathbf{x}_1,\dots,\mathbf{x}_N$ are u.d. in $[0,1[^s$.\\
Therefore, it is very interesting to get information on the order of this convergence.\\ Referring to this problem, a very useful estimate is provided by the Koksma-Hlawka inequality which is strictly related to the discrepancy of the sequence. \\ Before we can present this result, we need to define the variation of a function $f: [0,1]^s\to\mathbb{R}$. \\

\noindent By a partition $\pi$ of $[0,1]^s$ we mean a set of $s$ finite sequences $(\eta_i^{(0)},\ldots,\eta_i^{(m_i)})$ for $i=1,\ldots,s$ with $0=\eta_i^{(0)}\leq\eta_i^{(1)}\leq\cdots\leq\eta_i^{(m_i)}=1$. In connection with such a partition we define for each $i=1,\ldots,s$ an operator $\Delta_i$ by
\begin{eqnarray*}
\Delta_if(x_1,\ldots,x_{i-1},\eta_i^{(j)},x_{i+1},\ldots,x_s)&=&
f(x_1,\ldots,x_{i-1},\eta_i^{(j+1)},x_{i+1},\ldots,x_s)\\
&-&f(x_1,\ldots,x_{i-1},\eta_i^{(j)},x_{i+1},\ldots,x_s)
\end{eqnarray*} 
for $0\leq j<m_i$. Operators with different indices obviously commute and 
$\Delta_{i_1,\ldots,i_k}$ stands for $\Delta_{i_1}\cdots\Delta_{i_k}$. Such an operator commutes with summation over variables on which it does not act.

\begin{definition}[Function of bounded variation in the sense of Vitali]\ \\
For a function $f: [0,1]^s\to\mathbb{R}$ we set
$$V^{(s)}(f)=\sup_P\sum_{j_1=0}^{m_1-1}\cdots\sum_{j_s=0}^{m_s-1}
\left|\Delta_{1,\ldots,s}f(\eta_1^{(j_1)},\ldots,\eta_s^{(j_s)})\right|,$$
where the supremum is extended over all partitions $P$ of $[0,1]^s$. \\
If $V^{(s)}(f)$ is finite then $f$ is said to be of bounded variation on $[0,1]^s$ in the sense of Vitali.
\end{definition}

\begin{definition}[Function of bounded variation in the sense of Hardy and Krause]
Let $f: [0,1]^s\to\mathbb{R}$ and assume that $f$ is of bounded variation in the sense of Vitali.
If the restriction $f^{(F)}$ of $f$ to each face $F$ of $[0,1]^s$ of dimension $1,2,\ldots,s-1$ is of bounded variation on $F$ in the sense of Vitali, then $f$ is said to be of bounded variation on $[0,1]^s$ in the sense of Hardy and Krause.
\end{definition}

So we can state the following theorem.
\begin{theorem}[Koksma-Hlawka's Inequality]\ \\
Let $f$ be a function of bounded variation on $[0,1]^s$ in the sense of Hardy and Krause. Let $\boldsymbol\omega=(\mathbf{x}_1,\ldots, \mathbf{x}_N)$ be a finite set of points in $[0,1]^s$. Let us denote by $\boldsymbol\omega_l$ the projection of $\boldsymbol\omega$ on the $(s-l)-$dimensional face $F_l$ of $[0,1]^s$ defined by $F_l=\{(u_1,\ldots,u_s)\in [0,1]^s : u_{i_1}=\cdots$ $\cdots=u_{i_l}=1\}$. Then we have
\begin{equation}\label{KHIn}
\left|\frac{1}{N} \sum_{n=1}^N f(\mathbf{x}_n)-\int_{[0,1]^s}f(\mathbf{x})d\mathbf{x}\right|\leq
\sum_{l=0}^{s-1}\sum_{F_l}D_N^*(\mathbf{\boldsymbol\omega}_l)V^{(s-l)}(f^{(F_l)}),
\end{equation}
where the second sum is extended over all $(s-l)-$dimensional faces $F_l$ of the form $u_{i_1}=\cdots=u_{i_l}=1$. The discrepancy $D_N^*(\boldsymbol\omega_l)$ is clearly computed on the face of $[0,1]^s$ in which $\boldsymbol\omega_l$ is contained.
\end{theorem}
Hence, the Koksma-Hlawka's inequality assures that considering a low discrepancy $s$-dimensional sequence in the Quasi-Monte Carlo integration leads to an improvement on the Monte Carlo error bound.\\
However, in high dimensions the quasi-Monte Carlo method starts losing its effectiveness over the Monte Carlo method.\\
Since the error bound in the Monte Carlo approximation does not depend on the dimension, several authors found worthwhile to combine the advantages of Monte Carlo and Quasi-Monte Carlo methods, i.e., statistical error estimation and faster convergence. As we have already pointed out in the section about the convergence rate of the Halton sequence the basic idea is to consider the so-called hybrid sequences.\\

For a mixed $s$-dimensional sequence, whose elements are vectors obtained by concatenating $d$-dimensional vectors from a low-discrepancy sequence with $(s - d)$-dimensional random vectors, probabilistic upper bounds for its star discrepancy have been provided by several authors, e.g. \cite{Aistleitner_Hofer2, Okten, OTB, Gnewuch}.\\ The first deterministic bounds have been shown by Niederreiter \cite{Niederreiter4} and since then other results have been found, see e.g. \cite{Niederreiter_Winterhof, PHN}.

\section{Ergodic Theory}
In this section we recall basic definitions and results on the theory of dynamical systems. Our main references for this section will be \cite{CFS, Dajani_Kraaikamp, einsiedler_ward, Petersen, Walters}. \\
The structure of the section is the following: a first part deals with measure-theoretic aspects of the theory, namely measure-preserving transformations, property of mixing and ergodicity and examples. \\
A second part is concerned with topological aspects, such as invariant measures for continuous transformations, with a particular attention to dynamical systems associated to a numeration system.\\
\subsection{Preliminary definitions and results}
Before giving formal definitions of the mathematical objects in ergodic theory, it is worthwhile to understand what is ergodic theory about. Roughly speaking, it is a part of the theory of
dynamical systems. In its simplest form, a dynamical system is a function $T$ defined
on a set $X$. The aim of the theory is to describe the asymptotic behavior of the iterates of this map in a certain point $x\in\ X$, defined by induction in the following way: $x=T^0(x), T^1(x)=T(x), T^2(x)=T(T(x)),\dots, T^n=T^{n- 1}(T(x))$.\\
This sequence is called orbit of $x$ under $T$.

According to the different structure which $X$ and $T$ may have, the theory of dynamical systems splits into subfields:
\begin{itemize}
\item Differentiable dynamics deals with actions by differentiable maps on smooth
manifolds;
\item Ergodic theory deals with measure preserving actions of measurable maps on a
measure space, usually assumed to be finite;
\item Topological dynamics deals with actions of continuous maps on topological
spaces, usually compact metric spaces.
\end{itemize}

Our dissertation will focus only on the last two aspects of the theory.
\begin{definition}
A measure preserving transformation  is a transformation $T$ defined on a measure space $(X,\mathcal{A},\mu)$, such that 
\begin{enumerate}
\item $T$ is measurable: $E\in\mathcal{A}\Rightarrow T^{-1}E\in\mathcal{A}$;
\item $\mu$ is $T$-invariant: $\mu(T^{-1}E) = \mu(E)$ for all $E \in \mathcal{A}$.
\end{enumerate}
\end{definition}
A classic problem in ergodic theory is to find a suitable measure on $X$ which is preserved by $T$.\\
We now provide some examples of transformations and the corresponding preserved measures.
\paragraph{Rotations on a circle}\label{rotation}
Let $X=[0,1[$\ , let $\mathcal{B}$ be the Borel $\sigma$-algebra and $\lambda$ be the Lebesgue measure on $X$. Fix $\alpha \in \mathbb{R}$. Define $T_\alpha:X\rightarrow X$ by $T(x)=x+\alpha$ mod 1. $T_\alpha$ is called a circle rotation, because the map $R(x) = e^{2\pi ix}$ is an isomorphism between $T_\alpha$ and the rotation by the angle $2\pi\alpha$ on the unit circle $S^1$.\\
In fact, it is well-known that an alternative way to describe the circle $S^1$, that is often more convenient, is to cut and open the circle to obtain an interval. Let $I/\sim$ denote the unit interval $[0,1]$ with the endpoints identified: the symbol $\sim$ recalls that $0$ and $1$ are glued together. Then $I/\sim$ is equivalent to a circle.\\
More formally, consider $\mathbb{R}/\mathbb{Z}$, i.e. the space whose points are equivalence
classes $x+\mathbb{Z}$ of real numbers $x$ up to integers: two reals $x_1, x_2 \in \mathbb{R}$ are in the same equivalence class if and only if there exists $k \in\mathbb{Z}$ such that $x_1 = x_2 + k$. Then $\mathbb{R}/\mathbb{Z} = I/\sim$ since $[0, 1]$ contains exactly one representative for each equivalence class with the only exception of $0$ and $1$, which belong to the same equivalence class, but are identifyed.\\
All these spaces can be considered to describe a rotation on a circle.
\begin{lemma}
$T_\alpha$ is measure preserving on $X$ with respect to the Lebesgue measure $\lambda$.
\end{lemma}
\paragraph{Doubling map}
Consider now $X = [0, 1]$ equipped with the Lebesgue measure $\lambda$, and define $T_2:X\rightarrow X$ by $T_2(x)=2x$ mod 1. $T_2$ is  is called the doubling map. 
 It is an easy exercise to prove the following
\begin{lemma}
$T_2$ preserves the Lebesgue measure $\lambda$ on $[0,1]$.
\end{lemma}
\begin{remark}
The dynamical system associated to this map yields the binary expansion of points in $[0, 1[$ in the following way. We define the function $a_1$ by
\begin{equation*}
a_1(x)=\begin{cases}
0,\quad 0\leq x < 1/2\\
1, \quad 1/2\leq x < 1\ ,
\end{cases}
\end{equation*}
then $T_2x = 2x-a_1(x)$. Now, for $n \geq 1$ set $a_n(x) = a_1(T_2^{n-1}x)$. Fix $x \in X$,
rewriting we get $x = \frac{a_1(x)}{2} + \frac{T_2x}{2}$, where $T_2x = \frac{a_2(x)}{2} + \frac{T_2^2x}{2}$. Continuing in this manner, we see that for each $n \geq 1$,
\begin{equation*}
x=\frac{a_1(x)}{2} +\frac{a_2(x)}{2^2} +\dots+\frac{a_n(x)}{2^n} +\frac{T_2^nx}{2^n}.
\end{equation*}
Since $0 < T_2^nx < 1$, we get
\begin{equation*}
x-\sum_{i=1}^n \frac{a_i(x)}{2^i} =\frac{T_2^nx}{2^n} \to 0 \qquad {\rm as\ n}\to \infty.
\end{equation*}
Thus, $x = \sum_{i=1}^n \frac{a_i(x)}{2^i}$.
\end{remark}

We now present an example of a transformation which does not preserve the Lebesgue measure $\lambda$.

\paragraph{$\beta$-transformations}
Let $X = [0, 1[$ and $\beta >1$ be real. Define the transformation $T_\beta:X\rightarrow X$ by 
\begin{equation}
T_\beta x = \beta x\ {\rm mod\ } 1 = \begin{cases}
\beta x,\quad {\rm if\ } 0\leq x < \frac{1}{\beta}\\
\beta x-1, \quad {\rm if\ } \frac{1}{\beta}\leq x < 1\ .
\end{cases}
\end{equation}
Figure \ref{beta_map} illustrates the most well-known $\beta$-transformation, obtained with $\beta=\frac{\sqrt{5}+1}{2}$.
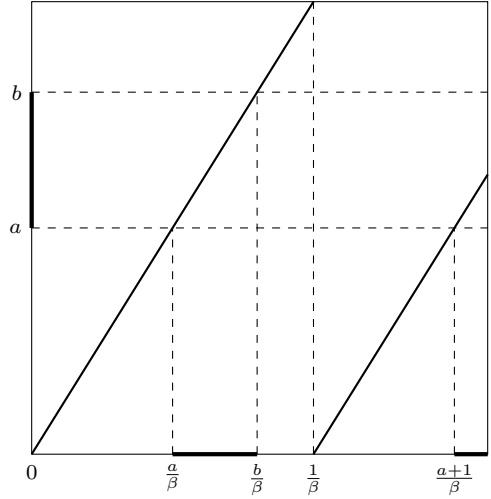
\begin{figure}[h!]
\begin{center}
\begin{tikzpicture}[scale=6]
\draw (0,0)  rectangle (1,1); 
\draw [thick] (0,0) node[below, black]{\scriptsize $0$} -- (0.618,1);
\draw [thick] (0.618,0) node[below, black]{\scriptsize $\frac{1}{\beta}$} -- (1, 0.618);
\draw [dashed] (0.618,0) -- (0.618,1);
\draw [dashed] (0, 0.5) node[left, black]{\scriptsize $a$} -- (1,0.5);
\draw [dashed] (0, 0.8) node[left, black]{\scriptsize $b$} -- (1,0.8);
\draw [dashed] (0.309,0) node[below, black]{\scriptsize $\frac{a}{\beta}$} -- (0.309,0.5);
\draw [dashed] (0.4944,0) node[below, black]{\scriptsize $\frac{b}{\beta}$} -- (0.4944,0.8);
\draw [dashed] (0.927,0) node[below, black]{\scriptsize $\frac{a+1}{\beta}$} -- (0.927,0.5);
\draw [ultra thick] (0.927,0) -- (1, 0);
\draw [ultra thick] (0, 0.5) -- (0, 0.8);
\draw [ultra thick] (0.309,0) -- (0.4944,0);
\end{tikzpicture}
\end{center}
\caption{Preimage of the interval $[a,b[$ under the $\beta$-transformation with $\beta=\frac{\sqrt{5}+1}{2}$}\label{beta_map}
\label{beta-transf}
\end{figure}

This family of transformations got a considerable interest since its introduction by R\'enyi \cite{Renyi} in 1957. He showed that for every $\beta>1$ there exists a unique normalized measure $\mu_\beta$, equivalent to the Lebesgue measure and invariant under $T_\beta$, such that for every element $E$ in the Borel $\sigma$-algebra of $[0,1[$
\begin{equation*}
\mu_\beta(E)=\int_{E}h_\beta(x)dx\ ,
\end{equation*}
whose density $h_\beta$ is a measurable function satisfying
\begin{equation*}
1-\frac{1}{\beta}\leq h_\beta(x)\leq \frac{1}{1-\frac{1}{\beta}}\ .
\end{equation*}
Few years later, Gelfond \cite{Gelfond} and Parry \cite{Parry} independently found the following explicit form for the density $h_\beta$ defining the measure $\mu_\beta$, where for convenience one defines $T_\beta^0(x)=x$ and inductively $T_\beta^n(1)=T_\beta^{n-1}(\{\beta\})$:
\begin{equation}\label{parry}
h_\beta(x) = \frac{1}{C(\beta)}\sum_{n=0}^\infty \frac{1}{\beta^n}\mathbf{1}_{[0, T_\beta^n1[}(x) \qquad {\rm for\ } x\in [0,1[\ ,
\end{equation}
 where $C(\beta)$ is the normalizing constant defined by
\begin{equation*}
C(\beta)=\int_0^1 \sum_{n=0}^\infty \frac{1}{\beta^n}\mathbf{1}_{[0, T_\beta^n1[}(x) dx\ .
\end{equation*}

\begin{remark}
By iterating $T_\beta$, one can show that every $x \in X$ has a series expansion of the form
\begin{equation*}
x= \sum_{i=1}^\infty \frac{d_i(x)}{\beta^i}\ ,
\end{equation*}
where the digits $d_i(x)$ are all elements of the set $\{0,1, \dots, \lfloor \beta - 1\rfloor\}$. In fact, if we define the function $d_1$ by
\begin{equation*}
d_1(x)=\begin{cases}
0,\quad 0\leq x < 1/\beta\\
1, \quad 1/\beta\leq x < 1\ ,
\end{cases}
\end{equation*}
then $T_\beta x = \beta x-d_1(x)$. For $n \geq 1$, set $d_n(x) = d_1(T_\beta^{n-1}x)$. Fix $x \in X$,
rewriting we get $x = \frac{d_1(x)}{\beta} + \frac{T_\beta x}{\beta}$, where $T_\beta x = \frac{d_2(x)}{\beta} + \frac{T_\beta^2x}{\beta}$. Continuing in this manner, we see that for each $n \geq 1$,
\begin{equation*}
x=\frac{d_1(x)}{\beta} +\frac{d_2(x)}{\beta^2} +\dots+\frac{d_n(x)}{\beta^n} +\frac{T_\beta^nx}{\beta^n}.
\end{equation*}
Since $0 < T_\beta^nx < 1$, we get
\begin{equation*}
x-\sum_{i=1}^n \frac{d_i(x)}{\beta^i} =\frac{T_\beta^nx}{\beta^n} \to 0 \qquad {\rm as\ n}\to \infty.
\end{equation*}
Thus, $x = \sum_{i=1}^n \frac{d_i(x)}{\beta^i}$.\\

R\'enyi \cite{Renyi} was the first to define the representation of real numbers in base $\beta$, generalising expansions in integer bases.\\
New generalizations, such as $(\alpha, \beta)$-expansions \cite{DK} and $(-\beta)$-expansion \cite{Liao_Steiner} have  been recently introduced. We will talk more about $\beta$-expansions of real numbers in the next subsection in relation to low-discrepancy sequences.
\end{remark}
\begin{remark}
A reasonable question about the density defining the $T_\beta$-invariant measure concerns the sum involved in the definition. More precisely one can ask under which conditions the sum is finite. The answer to this question has been given by Parry \cite{Parry} and is related to the $\beta$-expansion of $\beta$, i.e. if $\beta$ has a recurrent tail in its $\beta$-expansion, then $h_\beta$ is a step function with a finite number of steps.
\end{remark}

\begin{definition}
Let  $T:X\rightarrow X$ be a measure preserving transformation on a probability space $(X,\mathcal{A},\mu)$. The map $T$ is said to be ergodic if  for every $A\in\mathcal{A}$ satisfying 
$T^{-1}(A) = A$, we have $\mu(A)=0$ or $ \mu(A)=1$.
\end{definition}
We will use later an equivalent property: $T$ is ergodic if and only if for any measurable set $B$ with $\mu(B)>0$, the set $B^T=\bigcup_{i=- \infty}^{+ \infty}T^i(B)$ has measure $1$.\\
A useful characterization of ergodic transformations is given by the following lemma.
\begin{lemma}\label{lemma_ergodic}
Let $(X, \mathcal{A}, \mu)$ be a probability space and let $T : X\rightarrow X$ be a measure-preserving
transformation. Then $T$ is ergodic if and only if the only $T$-invariant measurable functions $f$, i.e. satisfying $f \circ T = f$ (up to sets of measure $0$), are constant almost everywhere.
\end{lemma}
For the straightforward proof, we notice that if the condition in the lemma holds
and $A$ is an invariant set, then $\mathbf{1}_A \circ T = \mathbf{1}_A$ almost everywhere, so that $\mathbf{1}_A$ is
an a.e. constant function and so $A$ or $X\setminus A$ is of measure $0$. Conversely, if $f$ is an
invariant function, we see that $\{x: f(x) < \alpha\}$ is an invariant set for each $\alpha$ and
hence of measure $0$ or $1$. It follows that $f$ is constant almost everywhere.\\

A basic problem in ergodic theory is to determine whether two measure preserving transformations are measure theoretically isomorphic. \\To answer to this question it seems natural to introduce the unitary operator $U_T$ associated to $T$ on $L_\mu^2(X)$. \\
Given a a measure-preserving map $T$ on a function space the associated induced operator $U_T :L_\mu^2\rightarrow L_\mu^2$ is defined by
\begin{equation*}
U_T(f) = f \circ T\ .
\end{equation*}
Since $L_\mu^2$ is a Hilbert space, then for any two functions $f_1, f_2 \in L_\mu^2$
\begin{eqnarray*}
\langle U_T f_1, U_T f_2 \rangle &=& \int f_1 \circ T \cdot \overline{f_2 \circ T} d\mu\\
&=& \int f_1f_2 d\mu \\
&=& \langle f_1, f_2 \rangle\ .
\end{eqnarray*}
where the second equality holds since $\mu$ is $T$-invariant.\\
Thus $U_T$ is an isometry mapping $L_\mu^2$ into $L_\mu^2$ whenever $(X,\mathcal{B}, \mu, T)$ is a measure-preserving system.\\
If $U : \mathcal{H}_1\rightarrow \mathcal{H}_2$ is a continuous linear operator between Hilbert spaces, then the relation
\begin{equation*}
 \langle Uf, g \rangle= \langle f, U^*g \rangle
\end{equation*}
defines an associated operator $U^{*} : \mathcal{H}_2\rightarrow \mathcal{H}_1$ called the adjoint of $U$. The operator $U$ is an isometry, i.e. $|| Uh||_{\mathcal{H}_2}= ||h||$ for all $h \in \mathcal{H}_1$, if and only if
\begin{equation*}
U^{*}U={\rm Id}_{\mathcal{H}_1}\ ,
\end{equation*}
where ${\rm Id}_{\mathcal{H}_1}$ is the identity operator on $\mathcal{H}_1$ and
\begin{equation*}
UU^*=\pi_{{\rm Im} U}\ ,
\end{equation*}
where $\pi_{{\rm Im}U}$ is the projection operator onto ${\rm Im}\ U$.\\ Finally, an invertible linear operator $U$ is called unitary if $U^{-1} = U^∗$, or equivalently if $U$ is invertible and
\begin{equation*}
|\langle Uh_1,Uh_2 \rangle|= |\langle h_1,h_2 \rangle|
\end{equation*}
for all $h_1, h_2 \in \mathcal{H}_1$. If $U : \mathcal{H}_1\rightarrow \mathcal{H}_2$ satisfies this last relation then $U$ is an isometry (even
if it is not invertible).\\ Thus for any measure-preserving transformation $T$,
the associated operator $U_T$ is an isometry, and if $T$ is invertible then the
associated operator $U_T$ is a unitary operator, called the associated unitary
operator of T or Koopman operator of T.\\
Properties of $T$ that are preserved under unitary equivalence of the associated unitary operators are called
spectral properties. $T_1$ and $T_2$ are called \emph{spectrally isomorphic} if their associated
unitary operators are unitarily equivalent. \\
The following lemma shows that ergodicity is a spectral property and it is a direct corollary of Lemma \ref{lemma_ergodic}.
\begin{lemma}
A measure-preserving transformation $T$ is ergodic if and only
if $1$ is a simple eigenvalue of the associated operator $U_T$.
\end{lemma}
For a concise treatment of the spectral theory relevant to ergodic theory, we refer to Parry's book \cite{Parry_book}.\\
All examples of measure preserving transformations considered above are also examples of ergodic transformations. In particular, it is an easy application of Lemma \ref{lemma_ergodic} to show that $T_\alpha:X\rightarrow X$ defined by $T(x)=x+\alpha$ mod 1 and $T_2:X\rightarrow X$ defined by $T_2(x)=2x$ mod 1 are ergodic. 
The proof of the ergodicity of $T_\beta$ is a bit more involved. A proof can be found in Renyi \cite{Renyi} and it makes use of the following
\begin{lemma}[Knopp's Lemma]
 Let $B$ be a Lebesgue set and $\mathcal{C}$ be a class of subintervals of $[0,1[$ such that
\begin{enumerate}
\item every open subinterval of $[0,1[$ is at most a countable union of disjoint elements from $\mathcal{C}$
\item $\forall A\in \mathcal{C}$ $\lambda(A\cap B)\geq \gamma \lambda(A)$ with $\gamma>0$ independent of $A$.
\end{enumerate}
Then $\lambda(B)=1$.
\end{lemma}
In general, checking that a given measure-preserving transformation is ergodic is a non-trivial task. We will discuss other stronger properties that a measure-preserving transformation may enjoy, and that in some cases are easier to check.\\
If $T$ is ergodic with respect to $\mu$, then the following result due to Birkhoff \cite{Birkhoff} holds.
\begin{theorem}[Birkhoff's Theorem]\label{B}
Let $(X,\mathcal{A},\mu, T)$ be a measure theoretical dynamical system. Then, for every $f \in L_\mu^1(X)$, the limit \begin{equation*}
\lim_{N\to\infty}\frac{1}{N}\sum_{n=0}^{N-1}f(T^{n}x)
\end{equation*}
exists for $\mu$-almost every $x\in X$ (here $T^0x=x$).
If $T:X\rightarrow X$ is ergodic, then for every $f\in \mathcal{L}^{1}(X)$ we have 
\begin{equation}\label{birkhoff}
\lim_{N\to\infty}\frac{1}{N}\sum_{n=0}^{N-1}f(T^{n}x)=\int_Xf(x)d\mu(x)\ ,
\end{equation}
for $\mu$-almost every $x\in X$.
\end{theorem}
In particular, an immediate consequence of the previous theorem is that the orbit $(T^nx)_{n\in\mathbb{N}}$ of $x$ under $T$ is a u.d. sequence for almost every $x \in X$ whenever T is ergodic.\\

The following lemma follows easily from Birkhoff's Theorem.
\begin{lemma}
Let $T$ be a measure-preserving transformation on the probability space $(X, \mathcal{A}, \mu)$. Then $T$ is ergodic if and only if for all $A,B \in\mathcal{A}$ we have
\begin{equation}\label{lemma14}
\lim_{N\to\infty}\frac{1}{N}\sum_{n=0}^{N-1}\mu(T^{-n}A\cap B)=\mu(A)\mu(B)\ .
\end{equation}
\end{lemma}
Recall from abstract probability theory that two events $A, B$ are independent if $\mu(A \cap B) = \mu(A)\mu(B)$. Also recall that a sequence $c_n$ is said to Ces\`{a}ro converge to $a$ if
\begin{equation*}
\lim_{N\to\infty}\frac{1}{N}\sum_{j=0}^{N-1}c_n=a\ .
\end{equation*}
Thus $T$ is ergodic if and only if the Ces\`{a}ro averages of the sequence $\mu(T^{-n}A\cap B)$ converge to $\mu(A)\mu(B)$. That is, given two sets $A,B \in \mathcal{A}$, the sets $T^{-n}A, B$ approach independence as $n$ tends to infinity in some appropriate sense.
￼
\subsection{Classical constructions of dynamical systems}
In this section we discuss several standard methods for creating new measure preserving transformations from old ones. These constructions appear quite frequently in applications.\\
Before giving explicit examples of constructions of dynamical systems, we provid some useful relations between measure-preserving transformations.\\
For instance, one interesting notion is that of isomorphism for measure-preserving transformations, i.e. when two measure-preserving transformations can be considered to be the same or equivalent.
\begin{definition}
Let $(X_i,\mathcal{A}_i, \mu_i, T_i)$ for $(i = 1,2)$ be two measure-preser-\linebreak ving systems over a probability space. We say that $T_1$ is isomorphic to $T_2$ if there exist $A\in\mathcal{A}_1$ and $B\in\mathcal{A}_2$ with $\mu_1(A)=1$, $\mu_2(B)=1$ such that
\begin{enumerate}
\item $T_1A\subset A$, $T_2B\subset B$
\item there exists an invertible measure-preserving transformation
\begin{equation*}
\phi : A\rightarrow B  
\end{equation*}
with
\begin{equation*}
\phi T_1(x)=T_2(\phi(x))
\end{equation*}
for all $x\in A$, where $A$ and $B$ are assumed to be equipped with the $\sigma$-algebras $A\cap\mathcal{A}_1$ and $B\cap\mathcal{A}_2$, respectively, and the restrictions of the measures $\mu_i$ to these $\sigma$-algebras.
\end{enumerate}
In this case we write $T_1\simeq T_2$.
\end{definition}
We have already seen that the map $\phi(x) = e^{2\pi ix}$ is an isomorphism between $T_\alpha$ and the rotation by the angle $2\pi\alpha$ on the unit circle $S^1$ and that the map $\phi(x) = e^{2\pi ix}$ is an isomorphism between the doubling map $T_2$ and the map $e^{i\theta} \rightarrow e^{2i\theta}$ on $S^1$.\\

This definition will be particularly useful in Chapter 4 when we will construct new ergodic systems with a prescribed property via an isomorphism with a system that has the desired property. In particular, we will refer to the following example that can be found in \cite[Example 2.4]{GHL}.
\begin{example}\label{Zb}
Let $\mathbb{Z}_b$ be the compact group of $b$-adic
integers and $\tau:\mathbb{Z}_b\longrightarrow \mathbb{Z}_b$ the addition-by-one
map (called odometer). Our goal is to find an isomorphism $\phi_b$ between $\tau$ and a transformation $T$ on $[0,1[$, such that $T=\phi_b\circ\tau\circ \phi_b^{-1}$.\\ For an integer
$b\geq 2$, every $z \in \mathbb{Z}_b$ has a unique expansion of the form
\begin{equation*}
 z=\sum_{j\geq 0}z_jb^j
\end{equation*}
with digits $z_j\in \{0,1,\dots, b-1\}$. For $z \in \mathbb{Z}_b$ we define the
$b$-adic Monna map $\phi_b:\mathbb{Z}_b\longrightarrow [0,1[$ by
\begin{equation*}
 \phi_b\left(\sum_{j\geq 0}z_jb^j  \right)= \sum_{j\geq 0}z_jb^{-j-1}.
\end{equation*}
The restriction of $\phi_b$ to $\mathbb{N}_0$ is the radical-inverse
function in base $b$ (see Definition \ref{radical_inverse}) that gives rise to the van der Corput sequence in base $b$.\\

The Monna map is continuous and surjective but not injective. In order to make
it an isomorphism we only consider the so-called regular representations, i.e.\
representations with infinitely 
many digits $z_j$ different from $b-1$. The Monna map restricted to these
regular representations admits an inverse (called pseudo-inverse)
$\phi_b^{-1}:[0,1)\longrightarrow \mathbb{Z}_b$, defined by
\begin{equation*}
 \phi_b^{-1}\left(\sum_{j\geq 0}z_jb^{-j-1}  \right)= \sum_{j\geq 0}z_jb^j\ ,
\end{equation*}
where $\sum_{j\geq 0}z_jb^{-j-1}$ is a $b$-adic rational in $[0,1)$. \\

Moreover $\phi_b$ is measure preserving from $\mathbb{Z}_b$ onto $[0,1[$ and $T=\phi_b\circ\tau\circ \phi_b^{-1}:[0,1[\rightarrow [0,1[$. Hence $\phi_b$ and $T$ are isomorphic.
\end{example}
\begin{remark}
For an isomorphism between two measure-preserving transformations, the following properties hold.
\begin{enumerate}
\item Isomorphism is an equivalence relation;
\item If $T_1\simeq T_2$, then $T_1^n\simeq T_2^n$ for every $n>0$\ .
\end{enumerate}
\end{remark}
Isomorphism is actually a relation which is in many cases too strong. A weaker and useful condition is conjugacy.
\begin{definition}
Let $(X,\mathcal{A}, \mu)$ be a probability space. Define an equivalence relation on $\mathcal{A}$ by saying that $A$ and $B$ are equivalent ($A\sim B$) if and only if $\mu(A\bigtriangleup B)=0$. Let $\tilde{\mathcal{A}}$ denote the collection of equivalence classes. Then $\tilde{\mathcal{A}}$ is a Boolean $\sigma$-algebra under the operations of complementation, union and intersection inherited from $\mathcal{A}$. The measure $\mu$ induces a measure $\tilde{\mu}$ on $\tilde{\mathcal{A}}$ by $\tilde{\mu}(\tilde{B})=\mu(B)$. The pair $(\tilde{\mathcal{A}}, \tilde{\mu})$ is called a measure algebra.	\\
A map $\Phi : (\tilde{\mathcal{A}_2}, \tilde{\mu_2})\rightarrow (\tilde{\mathcal{A}_1}, \tilde{\mu_1})$ is called an isomorphism of measure-algebras if it is a bijection that preserves complements, countable unions and satisfies $\tilde{\mu_1}(\Phi(\tilde{B}))=\tilde{\mu_2}(\tilde{B})$ for every $\tilde{B}\in\tilde{\mathcal{A}_2}$.
\end{definition}
Let $(X_i,\mathcal{A}_i, \mu_i)$ for $(i = 1,2)$ be two probability spaces with corresponding measure algebras $(\tilde{\mathcal{A}_i}, \tilde{\mu_i})$. If $\phi :X_1\rightarrow X_2$ is measure-preserving, then we have a map $\tilde{\phi}^{-1}:(\tilde{\mathcal{A}_2}, \tilde{\mu_2})\rightarrow (\tilde{\mathcal{A}_1}, \tilde{\mu_1})$ defined by $\tilde{\phi}^{-1}(\tilde{B})=\widetilde{\phi^{-1}(B)}$. This map is well-defined since $\phi$ is measure-preserving. The map $\tilde{\phi}^{-1}$ preserves complements and countable unions (and hence countable intersections). Also $\tilde{\mu}_1(\tilde{\phi}^{-1}(\tilde{B}))=\tilde{\mu}_2(\tilde{B})$ for every $\tilde{B}\in\tilde{\mathcal{A}_2}$. Therefore $\tilde{\phi}^{-1}$ can be considered a homomorphism of measure algebras. Note that $\tilde{\phi}^{-1}$ is injective.
\begin{definition}
Let $T_i$ be a measure-preserving transformation on the probability space $(X_i,\mathcal{A}_i, \mu_i)$, $i=1,2$. We say that $T_1$ is conjugate to $T_2$ if there exists a measure-algebra isomorphism $\Phi : (\tilde{\mathcal{A}_2}, \tilde{\mu_2})\rightarrow (\tilde{\mathcal{A}_1}, \tilde{\mu_1})$ such that $\Phi\tilde{T_2}^{-1}=\tilde{T_1}^{-1}\Phi$.
\end{definition}
Conjugacy is also an equivalence relation and all isomorphic measure-preserving transformations are conjugate, as stated by the following
\begin{theorem}
Let $T_i$ be a measure-preserving transformation on the probability space $(X_i,\mathcal{A}_i, \mu_i)$, $i=1,2$. If $T_1$ is isomorphic to $T_2$, then $T_1$ is conjugate to $T_2$.
\end{theorem}
In some cases, conjugacy can also imply isomorphism.

We have already seen that another way to compare two measure-preser-\linebreak ving transformations is to consider the associated unitary operator. We recall briefly that $T_1$ and $T_2$ are called spectrally isomorphic if their associated
unitary operators are unitarily equivalent.\\
We briefly remind what is an eigenvalue of a measure-preserving transformation.
\begin{definition}
Let $T$ be a measure-preserving transformation $T$ on a probability space $(X, \mathcal{A}, \mu)$, and let $U_T$ be the induced linear isometry of $L^2_\mu$. The eigenvalues and eigenfunctions of $U_T$ are called the eigenvalues and eigenfunctions of $T$. So a complex number $\lambda$ is called an eigenvalue of $T$ if there exists a non-zero function $f\in L_\mu^2$, satisfying $U_Tf=\lambda f$. The function $f$ is called an eigenfunction of $T$ corresponding to the eigenvalue $\lambda$.
\end{definition}

\begin{definition}
An ergodic measure-preserving transformation $T$ on a probability space $(X, \mathcal{A}, \mu)$  is said to have discrete spectrum (or pure point spectrum) if there exists an orthonormal basis for $L^2_\mu$ consisting of eigenfunctions of $T$.
\end{definition}

The following result shows that spectral isomorphism is weaker than conjugacy.
\begin{theorem}
Let $T_i$ be a measure-preserving transformation on the probability space $(X_i,\mathcal{A}_i, \mu_i)$, $i=1,2$. If $T_1$ and $T_2$ are conjugate, then they are spectrally isomorphic.
\end{theorem}
There are instances when spectral isomorphism implies conjugacy.\\
Then, summarising, we have the following definition.
\begin{definition}
A property $P$ of a measure-preserving transformation is an isomorphism, or
conjugacy or spectral invariant if the following holds: Given $T_1$ has $P$ and $T_2$ is isomorphic,
or conjugate or spectrally isomorphic, to $T_1$ then $T_2$ has property $P$.
\end{definition}
Now, since isomorphism implies conjugacy and conjugacy implies spectral isomorphism, a spectral invariant is a conjugacy invariant and a conjugacy invariant is an isomorphism invariant.\\

As we have already pointed out, any two measure-preserving transformations that are
conjugate are also spectrally isomorphic. Now, if two spectrally isomorphic measure-preserving transformations have discrete spectrum, then they are conjugate. Thus, the property of discrete spectrum is very important and depends upon the eigenvalues of the measure-preserving transformation. Note the if two measure-preserving transformations are spectrally isomorphic then they have the same eigenvalues.\\

The following theorem proved by Halmos and
von Neumann in 1942 shows that the eigenvalues determine completely whether two transformations with discrete spectrum are conjugate or not. 
\begin{theorem}[Discrete Spectrum Theorem]\label{Discr_spectr_thm}
Let $T_1$ and $T_2$ be ergodic measure-preserving transformations with discrete spectrum of the probability spaces $(X_i,\mathcal{A}_i, \mu_i)$ for $(i = 1,2)$. Then
the following are equivalent:
\begin{enumerate}
\item $T_1$ and $T_2$ are spectrally isomorphic;
\item $T_1$ and $T_2$ have the same eigenvalues;
\item $T_1$ and $T_2$ are conjugate.
\end{enumerate}
\end{theorem}
\begin{corollary}
If $T$ is an invertible ergodic measure-preserving transformation with discrete spectrum, then $T$ and $T^{-1}$ are conjugate.
\end{corollary}
\begin{remark}
If the spaces $(X_i,\mathcal{A}_i, \mu_i)$ for $(i = 1,2)$ are both complete separable spaces, then the statements of Theorem \ref{Discr_spectr_thm} are equivalent to $T_1$ being isomorphic to $T_2$.
\end{remark}

Let us now turn our attention to a collection of ergodic measure-preserving transformations which have discrete spectrum (see \cite[\S 3.3]{Walters} for details).
\begin{example}
Let $S^1$ be the complex unit circle and suppose $T : S^1 \rightarrow S^1$ is defined by $T(z) = az$ where a
is not a root of unity. We know that $T$ is ergodic and is a rotation of a compact group. Consider
the sequence of functions $f_n : S^1 \rightarrow \mathbb{C}$ defined by $f_n(z) = z^n$. Then $f_n$ is an eigenfunction of $T$ corresponding to the eigenvalue $a^n$. Since $(f_n)$ forms a basis for $L^2_\mu(S^1)$, we see that $T$ has discrete spectrum.
\end{example}
The following two theorems  completely solve the conjugacy problem for ergodic rotations with discrete spectrum.\\
They can be found in \cite[Theorem 3.6, Theorem 3.7]{Walters} and we provide a proof for the second one since it does not appear in the reference and it seemed to be of some interest for the reader.
\begin{theorem}[Representation Theorem]
An ergodic measure-preserving transformation $T$ with discrete spectrum on a probability space $(X,\mathcal{A}, \mu)$ is conjugate to an ergodic rotation
on some compact abelian group. The group is metrisable if and only if $(X,\mathcal{A}, \mu)$ has a countable basis.
\end{theorem}
\begin{theorem}[Existence Theorem]
Every subgroup $K$ of $S^1$ is the group of eigenvalues of an ergodic measure-preserving transformation with discrete spectrum.
\end{theorem}
\begin{proof}
Let $K$ be a subgroup of $S^1$ and consider the following family indexed by $K$
\begin{equation*}
G_t=\overline{\{nt | n\in\mathbb{Z}\}}\qquad t\in K\ .
\end{equation*}
Then consider the infinite product
\begin{equation*}
G=\prod_{t\in K} G_t\ .
\end{equation*}
This is a compact group and the Haar measure $\mu$ is defined on it. The following functions $g_t:G\rightarrow \mathbb{C}$ defined by
\begin{equation*}
g_t((x_s)_{s\in K})=e^{2\pi ix_t}
\end{equation*}
are characters of $G$.\\
We want to show that there exists an ergodic measure-preserving transformation $S$ with pure discrete spectrum $K$, such that these characters are the eigenfunctions of $S$ associated to the eigenvalues in $K$.\\
Define $S:G\rightarrow G$ to be the map $(x_t)_{t\in K}\mapsto (x_t+t)_{t\in K}$. Then it follows that
\begin{eqnarray*}
g_t(S(x_s)_{s\in K})&=&g_t((x_s+s)_{s\in K})= e^{2\pi i (x_t+t)}\\
&=& e^{2\pi i t}g_t((x_s)_{s\in K})\ .
\end{eqnarray*}
Hence the characters $g_t$ are eigenfunctions of $S$. Moreover, since $G$ is compact, the characters $g_t$ form an orthonormal basis for $L^2(\mu)$. Finally, $S$ is ergodic since for every non-empty open subset $U$ of $G$ we have $\cup_{n=-\infty}^{\infty}S^n U=G$. 
\end{proof}

\subsubsection{Products}
The product of two measure spaces $(X_i,\mathcal{A}_i, \mu_i)$ for $(i = 1,2)$ is the measure space $(X_1 \times X_2,\mathcal{A}_1 \otimes \mathcal{A}_2, \mu_1 \times \mu_2)$ where $\mathcal{A}_1 \times \mathcal{A}_2$ is the smallest $\sigma$–algebra which contains all set of the form $A_1 \times A_2$ where $A_i \in \mathcal{A}_i$ , and $\mu_1 \times \mu_2$ is the unique measure such that $(\mu_1 \times \mu_2)(A_1 \times A_2) = \mu_1(A_1)\mu_2(A_2)$.
This construction captures the idea of independence from probability theory: if $(X_i,\mathcal{A}_i, \mu_i)$ are the probability models of two random experiments, and these experiments are \lq\lq independent\rq\rq, then $(X_1 \times X_2, \mathcal{A}_1 \otimes \mathcal{A}_2, \mu_1 \times \mu_2)$ is the probability model of the pair of experiments.
\begin{definition}
The product of two measure preserving systems $(X_i,\mathcal{A}_i,$\linebreak $\mu_i, T_i)$ for $(i = 1,2)$ is the measure preserving system $(X_1 \times X_2,\mathcal{A}_1 \otimes \mathcal{A}_2, \mu_1 \times \mu_2, T_1\times T_2)$, where $(T_1 \times T_2)(x_1,x_2) = (T_1x_1,T_2x_2)$.
\end{definition}
We now show that the product of two ergodic measure preserving transformations is not always ergodic.
\begin{theorem}
Let $(X_i,\mathcal{A}_i, \mu_i, T_i)$ for $(i = 1,2)$ be two ergodic systems. Then
$T_1 \times T_2 $ is ergodic if and only if $U_{T_1}$ and $U_{T_2}$ have no common eigenvalues other than 1.
\end{theorem}
\subsubsection{Induced transformations}
A central problem in ergodic theory is that of recurrence, concerning how points in measurable dynamical systems return close to themselves under iteration. The first and most important result is due to Poincar\'e \cite{Poincare} in 1890 who proved it in the context of a natural invariant measure in the “three-body” problem of planetary orbits, before the creation of abstract measure theory. Poincar\'e recurrence is the pigeon-hole principle for ergodic theory; indeed on a finite measure space it is exactly the pigeon-hole principle.
\begin{theorem}
Let $T : X \rightarrow X$ be a measure-preserving transformation on a probability space $(X,\mathcal{A}, \mu)$, and let $E \subset X$ be a measurable set. Then almost every point $x \in E$ returns to $E$ infinitely often. That is, there exists a measurable set $F \subset E$ with $\mu(F) = \mu(E)$ with the property that for every $x \in F$ there exist positive integers $n_1 < n_2 < \dots$ with $T^{n_i} x \in E$ for all $i\geq1$.
\end{theorem}
In the following example we show that the Poincar\'e recurrence does not necessarily hold if the measure space is not
of finite measure.
\begin{example}
The map $T :\mathbb{R}\rightarrow \mathbb{R}$ defined by $T(x)=x+1$ preserves the Lebesgue measure $\lambda$ on $\mathbb{R}$. 
For any bounded set $E \subset \mathbb{R}$ and any $x \in E$, the set
\begin{equation*}
\{n\geq 1 : T^nx\in E\}
\end{equation*}
is finite. Thus the map $T$ exhibits no recurrence.
\end{example}
Now let $A$ be a measurable set with $\mu(A) > 0$. By the Poincar\'e recurrence, the first return time to $A$, defined by
\begin{equation*}
n_A(x) = \inf_{n\geq 1} \{n : T^n(x) \in A\}
\end{equation*}
exists, i.e. is finite, almost everywhere.
\begin{definition}
 The map $T_A:A\rightarrow A$ defined (almost everywhere) by
\begin{equation*}
T_A(x)=T^{n_A(x)}(x)
\end{equation*} 
is called the transformation induced by $T$ on the set $A$. 
\end{definition}
Observe that both $n_A : X\rightarrow \mathbb{N}$ and $T_A : A\rightarrow A$ are measurable, since for every $n \geq 1$, we can write $A_n = \{x \in A : n_a(x) = n\}$. Then the sets
\begin{eqnarray*}
A_1 &=& A\cap T^{-1}A\ ,\\
A_2 &=&A\cap T^{-2}A\setminus A_1\ ,\\
& \vdots &\\
A_n &=&A\cap T^{-n}A\setminus \bigcup_{i<n}A_i
\end{eqnarray*}
are all measurable, as it is
\begin{equation*}
T^nA_n =A\cap T^nA\setminus (TA\cup T^2A\cup\dots \cup T^{n-1}A)\ ,
\end{equation*}
since $T$ is invertible by assumption.
\begin{proposition}
The induced transformation $T_A$ is a measure-preser-\linebreak ving transformation on the space $(A,\mathcal{B}(A),\mu_A,T_A)$, where $\mathcal{B}(A)=\{E\cap A : E\in \mathcal{B}\}$, $\mu_A$ is the measure $\mu_A(E)= \mu(E|A) = \mu(E \cap A)/\mu(A)$. If T is ergodic with ￼respect to $\mu$ then $T_A$ is ergodic with respect to $\mu_A$.
\end{proposition}
As pointed out in \cite{einsiedler_ward}, the effect of $T_A$ can be seen in the situation described by Figure \ref{Kakutani_skycraper}, called the Kakutani skyscraper. 
\begin{figure}[h!]
\begin{center}
\begin{tikzpicture}[scale=1.2]
\draw [ultra thick] (0, 0) -- (7, 0);
\draw[ultra thick, dashed] (7,0) -- (9,0) node[right, black]{$A$};
\draw [ultra thick] (1, 1) node[left, black]{$T(A)\setminus A$}-- (7, 1);
\draw[ultra thick, dashed] (7,1) -- (9,1);
\draw [ultra thick] (2, 2) node[left, black]{$T^2(A)\setminus (A\cup T(A))$} -- (7,2);
\draw[ultra thick, dashed] (7,2) -- (9,2);
\draw [ultra thick] (3, 3) -- (7, 3);
\node at (0.8,2.7) {$\vdots$};
\node at (5.7,3.6) {$\vdots$};
\draw[ultra thick, dashed] (7,3) -- (9,3);
\draw [thick,->] (1.5, 0.1) -- (1.5, 0.8);
\draw [thick,->] (2.5, 0.1) -- (2.5, 0.8);
\draw [thick,->] (3.5, 0.1) -- (3.5, 0.8);
\draw [thick,->] (5.5, 0.1) -- (5.5, 0.8);
\draw [thick,->] (2.5, 1.1) -- (2.5, 1.8);
\draw [thick,->] (3.5, 1.1) -- (3.5, 1.8);
\draw [thick,->] (5.5, 1.1) -- (5.5, 1.8);
\draw [thick,->] (5.5, 2.1) -- (5.5, 2.8);
\node at (1.3,0.4) {$T$};
\node at (2.3,0.4) {$T$};
\node at (3.3,0.4) {$T$};
\node at (2.3,1.4) {$T$};
\node at (3.3,1.4) {$T$};
\node at (5.3,0.4) {$T$};
\node at (5.3,1.4) {$T$};
\node at (5.3,2.4) {$T$};
\draw [ultra thick] (1, -0.1) -- (1, 0.1);
\draw [ultra thick] (2, -0.1) -- (2, 0.1);
\draw [ultra thick] (3, -0.1) -- (3, 0.1);
\draw [ultra thick] (4, -0.1) -- (4, 0.1);

\draw [ultra thick] (2, 0.9) -- (2, 1.1);
\draw [ultra thick] (3, 0.9) -- (3, 1.1);
\draw [ultra thick] (4, 0.9) -- (4, 1.1);
\draw [ultra thick] (3, 1.9) -- (3, 2.1);
\draw [ultra thick] (4, 1.9) -- (4, 2.1);
\draw [->] (0.2,0.3) arc[x radius=0.5cm, y radius =.5cm, start angle=30, end angle=340];
\node at (-1,0.1) {$T$};
\node at (0.5,-0.3) {$A_1$};
\node at (1.5,-0.3) {$A_2$};
\node at (2.5,-0.3) {$A_3$};
\node at (4.5,-0.4) {$\dots$};
\end{tikzpicture}
\caption{The induced transformation $T_A$}\label{Kakutani_skycraper}
\end{center}
\end{figure}
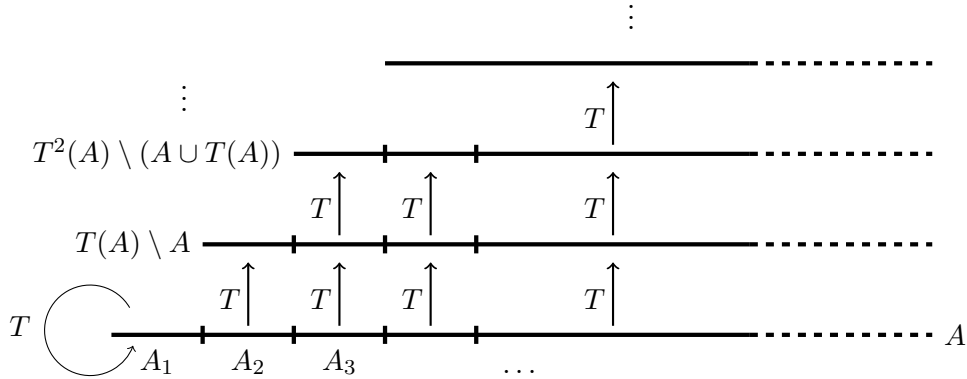

The original transformation $T$ sends any point with a floor above it to the point immediately above on the next floor, and any point on a top floor is moved somewhere to the base floor $A$. The induced transformation $T_A$ is the map defined almost everywhere on the bottom floor by sending each point to the point obtained by going through all the floors above it and returning to $A$.\\

The Poincar\'e recurrence says that for any measure-preserving system $(X, \mathcal{B}, \mu, T)$ and any set $A$ of positive measure, almost every point on the ground floor of the associated Kakutani skyscraper returns to the ground floor at some point. Ergodicity strengthens this statement to say that almost every point of the entire space $X$ lies on some floor of the skyscraper. However, the Poincar\'e recurrence does not tell us how long we should have to wait for
this to happen. One would expect that return times to sets of large measure are small, and that return times to sets of small measure are large. This is indeed the case, and forms the content of Kac's Lemma \cite{Kac}
\begin{theorem}[Kac's Lemma]
Let $(X, \mathcal{B}, \mu, T)$ be an ergodic measure-preserving probability system and let $A \in \mathcal{B}$ have strictly positive measure. Then the expected return time to $A$ is $\frac{1}{\mu(A)}$; equivalently
\begin{equation*}
\int_A n_A d\mu =1\ .
\end{equation*} 
\end{theorem}
As it will appear clear in Chapter 3, Kakutani skyscrapers are a powerful tool in ergodic theory. From the Kakutani skycraper construction we can deduce a very useful lemma of Kakutani \cite{Kakutani2} and Rokhlin \cite{Rokhlin} often called Rokhlin's lemma.
\begin{lemma}
Let $T : X \rightarrow X$ be an ergodic measure preserving transformation on a non-atomic probability space $(X, \mathcal{B}, \mu)$. Then for any $n\geq 1$ and $\epsilon >0$ there exists a measurable set $B \subset X$ such
that $B, TB,\dots , T^{n-1}B$ are pairwise disjoint and $\mu\left(\cup_{i=0}^{n-1} T^i(B)\right) > 1- \epsilon$. The
collection $\{B, TB,\dots ,$\linebreak $T^{n-1}B\}$ is referred to as a Rokhlin tower of height $n$ for the transformation $T$.
\end{lemma}
\subsubsection{Interval exchange}
The class of interval exchange transformations was introduced by Sinai \cite{Sinai}. An interval exchange transformation
is the map obtained by cutting the interval into a finite number of pieces and permuting them in such a way that the resulting map is invertible, and restricted to each interval is an order-preserving isometry. More formally, we have the following definition.

\begin{definition}
Let $d \geq 2$ be a natural number and let $\pi$ be an irreducible permutation
of $\{1, \dots, d\}$, that is, $\pi\{1, \dots , k\} \neq \{1, \dots , k\}$, for any $1 \leq k < d$. Moreover, let $\Lambda_d$ be the set of vectors $(\lambda_1,\dots,\lambda_d)$ in $\mathbb{R}^d$ such that $0\leq \lambda_i\leq 1$ for all $i$ and $\sum_{i=1}^d\lambda_i=1$.\\
An interval exchange on $[0,1[$ is a map $T_{\lambda, \pi}:[0,1[ \ \rightarrow [0,1[$ such that it is the piecewise translation defined by partitioning the interval $[0, 1[$ into $d$ sub-intervals of lengths $\lambda_1,\lambda_2,\dots , \lambda_r$ and rearranging them according to the permutation $\pi$;  formally
\begin{equation*}
T_{\lambda, \pi}(x)=x+\sum_{j<i}\lambda_{\pi(j)}-\sum_{j<i}\lambda_j\ ,
\end{equation*}
when $x$ is in the interval 
\begin{equation*}
I_i=\left[\sum_{j<i}\lambda_j, \sum_{j\leq i}\lambda_j\right[\ .
\end{equation*}
\end{definition}
It is easy to see that an interval exchange is a map of $[0,1[$ into itself which is one-to-one, preserves the Lebesgue measure $\lambda$ and is continuous $\lambda$-almost everywhere.
Masur \cite{Masur} and Veech \cite{Veech} independently showed that for almost all values of the sequence of lengths $\lambda_i$, $1\leq i\leq d$, the interval exchange transformation is ergodic. In fact they proved unique ergodicity, which we will discuss in the last part of this section.

\subsubsection{Cutting-stacking}
The cutting-stacking method is a useful tool to construct interval exchanges. Its first fomulation is due to von Neumann and Kakutani and then generalised by Friedman \cite{Friedman}.\\
We refer to the description made in \cite{GHL} where it is used to build sequences in the unit interval with good discrepancy.\\
Before starting with the description of the method, we need to fix the notation. \\
We will call columns and denote them by $C = (I_1,\dots , I_h)$ (also called
towers) a set of disjoint subintervals $I_j = [c_j , d_j[$ of $[0, 1[$ having the same length. The length of the interval $I_j$ is called the width of $C$ and denoted by $l(C)$. The interval $I_1$ is called the bottom of $C$, the interval $I_h$ is called the top of $C$, the union supp$(C)=\cup_{i=j}^h I_j$ is the support of $C$ and the integer $h$ its height. With the column $C$ is associated a translation
map 
\begin{equation*}
T_C : {\rm supp}(C) \setminus I_h \rightarrow {\rm supp}(C) \setminus I_1
\end{equation*}
defined by 
\begin{equation*}
T_C(x) = x + (c_{j+1}- c_j)
\end{equation*}
if $x \in I_j$ , $1 \leq j < h$. We represent a column $C = (I_1,\dots , I_h)$ by drawing each interval $I_{j+1}$, $1 \leq j < d$ above the interval $I_j$. \\
Consider now a given finite set of columns $S = \{C_1,\dots ,C_s\}$ with disjoint supports. We associate to $S$ the map $T_S$ which coincides with $T_{C_i}$ for $1 \leq i \leq s$. By extending the above notation, we have supp$(S)= \cup_{i=1}^s {\rm supp}(C_i)$ is the support of $S$ and $w(S)=\sum_{i=1}^s l(C_i)$ is the
width of $S$. In the sequel, we usually assume that the columns $C_i$ of $S$ are indexed according to the order of their bottoms, the one induced by the natural order of $[0, 1[$.\\ 
A cutting of a column $C = (I_1,\dots , I_h)$ in $t$ columns is the set of columns $C_i= \{I_{i,1},\dots , I_{i,h}\}$ such that $\cup_{i=1}^t I_{i,1}=I_1$ and each map $T_{|C_i}$ is
the restriction of $T_C$ on $C_i$. More generally, a cutting of a set $S$ of columns is obtained by collecting all columns resulting by cutting part or all columns from $S$
and then producing a new set of columns $S' = \{C'_1 , \dots ,C'_{s'}\}$.\\
Now, a stacking of a column $C'= (I'_1 , \dots , I'_{h'})$ above a column $C = (I_1, \dots ,$\linebreak $I_h)$ having same width and disjoint support is by definition the column $C\ast C' = (I_1, \dots, I_h, I'_1,\dots , I'_{h'})$.\\
The map $T_{C\ast C'}$ extends both $T_C$ and $T_C'$ and $T_{C\ast C'}$ translates $I_h$ onto $I'_1$.\\
One can also introduce the empty column $( )$ of height $0$ to set by definition $C \ast ( ) = ( ) \ast C = C$ for any column $C$.\\ A sequence $\Sigma = (S_m)_{m\geq 0}$ of sets $S_m$ of columns is said to be complete if supp$(S_0) = [0, 1[$, $\lim_{m\to \infty} w(S_m) = 0$ and for each $m \geq 1$, $S_{m+1}$ is built from $S_m$ by performing cutting and stacking but a finite number of times. By construction $T_{S_{m+1}}$ extends $T_{S_m}$. We denote by top$(S_m)$
(resp. bot$(S_m)$) the union of top (resp. bottom) intervals of columns in $S_m$.\\
Clearly top$(S_{m+1}) \subset {\rm top}(S_m)$, bot$(S_{m+1}) \subset {\rm bot}(S_m)$ and the intersections top$(\Sigma) = \cap_{m\geq 0}{\rm top}(S_m)$, bot$(\Sigma) = \cap_{m\geq 0} {\rm bot}(S_m)$ are at most countable and finite if the numbers of columns in infinitely many $S_n$ are bounded. Clearly, the map $T_{S_m}$ is not defined on top$(\Sigma)$ but it is easy to prove that for a complete sequence $\Sigma$ there is a unique transformation $T : [0, 1[ \setminus {\rm top}(\Sigma)\rightarrow [0, 1[$ which extends all the $T_{S_m}$'s.\\ Moreover, $T$ is a measure-preserving map of $([0, 1[, \lambda)$, well defined on $[0, 1[ \setminus$\linebreak ${\rm top}(\Sigma)$ and invertible on $[0, 1[ \setminus {\rm bot}(\Sigma)$.\\

The transformations obtained in this manner are called staircase transformations. A transformation created by a cutting and stacking with a single column resulting from each iteration is a rank-one transformation. We refer to \cite{Ferenczi} and \cite{Friedman2} for more details. Rank-one transformations are measurable and measure-preserving under Lebesgue measure.\\

We now present the simplest example of rank-one transformation obtained by means of the cutting-stacking technique.
\paragraph{The von Neumann-Kakutani odometer}
 Let us consider the unit interval as a column and let us denote it by $S_0=\{[0,1[\}$. Let us split $S_0$ into two subintervals, $[0, 1/2[$ and $[1/2, 1[$ and stack them one on the other in order to form the column 
$S_1=\{[0, 1/2[, [1/2, 1[\}.$
The {\it heigh} of $S_1$ is $h(S_1)=2$ and its {\it width} is $w(S_1)=1/2$. We define the map $T_{|_{[0,1/2[}}(x)=x+\frac{1}{2}$ as the translation of the first subinterval of $S_1$ onto the second one. At the second step we cut each interval of $S_1$ into equal parts, take the right half and stack it onto the top of the left half. So we get the column
$S_2=\{[0, 1/4[,[1/2, 3/4[, [1/4,1/2[,[3/4,1[\}$,
with $h(S_2)=4$ and $w(S_2)=1/4$. To this second step we associate the map $T_{|_{[1/2, 3/4[}}(x)=x+\frac{1}{4}$.
We can visualize the above steps of the cutting-stacking procedure in Figure 2.

The procedure goes on this way, splitting each interval of the column $S_{n-1}$ in half and stacking the right column of intervals onto the left column. So we obtain the sequence of columns
\begin{equation*}
S_n=\{[\phi_2(0),\phi_2(0)+2^{-n}[,[\phi_2(1),\phi_2(1)+2^{-n}[,\ldots,[\phi_2(n-1),1[\}\ ,
\end{equation*} 
with $h(S_n)=2^{n-1}$ and $w(S_n)=\frac{1}{2^{n-1}}$. 
It is worthwhile to note that the left endpoints of the intervals are exactly $\phi_2(n)$, the radical inverse function of $n$ in base $2$.

 Let us recall that the associated sequence $(\phi_2(n))_{n\geq 0}$ is the already described van der Corput sequence.

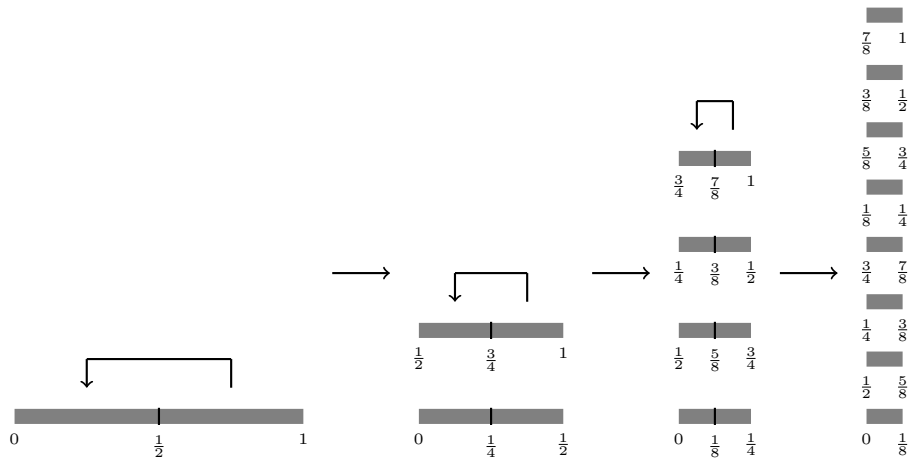
\begin{figure}[h!]
\begin{tikzpicture}[scale=3.8]
\draw [gray, line width=0.2cm] (0,0) node[below, black]{\tiny $0$} -- (1,0) node[below, black]{\tiny $1$};
\draw [thick] (0.5,-0.03) node[below, black]{\tiny $\frac{1}{2}$} -- (0.5, 0.03);
\draw [thick] (0.75, 0.1) -- (0.75, 0.2);
\draw [thick] (0.75, 0.2) -- (0.25, 0.2);
\draw [->, thick]  (0.25, 0.2) -- (0.25, 0.1);

\draw [->, thick] (1.1,0.5) -- (1.3,0.5);

\draw [gray, line width=0.2cm] (1.4,0) node[below, black]{\tiny $0$} -- (1.9,0) node[below, black]{\tiny $\frac{1}{2}$};
\draw [gray, line width=0.2cm] (1.4,0.3) node[below, black]{\tiny $\frac{1}{2}$} -- (1.9,0.3) node[below, black]{\tiny $1$};
\draw [thick] (1.65,-0.03) node[below, black]{\tiny $\frac{1}{4}$} -- (1.65, 0.03);
\draw [thick] (1.65,0.27) node[below, black]{\tiny $\frac{3}{4}$} -- (1.65, 0.33);
\draw [thick] (1.775, 0.4) -- (1.775, 0.5);
\draw [thick] (1.775, 0.5) -- (1.525, 0.5);
\draw [->, thick]  (1.525, 0.5) -- (1.525, 0.4);

\draw [->, thick] (2,0.5) -- (2.2,0.5);

\draw [gray, line width=0.2cm] (2.3,0) node[below, black]{\tiny $0$} -- (2.55,0) node[below, black]{\tiny $\frac{1}{4}$};
\draw [gray, line width=0.2cm] (2.3,0.3) node[below, black]{\tiny $\frac{1}{2}$} -- (2.55,0.3) node[below, black]{\tiny $\frac{3}{4}$};
\draw [gray, line width=0.2cm] (2.3,0.6) node[below, black]{\tiny $\frac{1}{4}$} -- (2.55,0.6) node[below, black]{\tiny $\frac{1}{2}$};
\draw [gray, line width=0.2cm] (2.3,0.9) node[below, black]{\tiny $\frac{3}{4}$} -- (2.55,0.9) node[below, black]{\tiny $1$};
\draw [thick] (2.425, -0.03) node[below, black]{\tiny $\frac{1}{8}$} -- (2.425, 0.03);
\draw [thick] (2.425, 0.27) node[below, black]{\tiny $\frac{5}{8}$} -- (2.425, 0.33);
\draw [thick] (2.425, 0.57) node[below, black]{\tiny $\frac{3}{8}$} -- (2.425, 0.63);
\draw [thick] (2.425, 0.87) node[below, black]{\tiny $\frac{7}{8}$} -- (2.425, 0.93);
\draw [thick] (2.4875, 1) -- (2.4875,1.1);
\draw [thick] (2.4875, 1.1) -- (2.3625,1.1);
\draw [->, thick] (2.3625, 1.1) -- (2.3625,1);

\draw [->, thick] (2.65,0.5) -- (2.85,0.5);

\draw [gray, line width=0.2cm] (2.95,0) node[below, black]{\tiny $0$} -- (3.075,0) node[below, black]{\tiny $\frac{1}{8}$};
\draw [gray, line width=0.2cm] (2.95,0.2) node[below, black]{\tiny $\frac{1}{2}$} -- (3.075,0.2) node[below, black]{\tiny $\frac{5}{8}$};
\draw [gray, line width=0.2cm] (2.95,0.4) node[below, black]{\tiny $\frac{1}{4}$} -- (3.075,0.4) node[below, black]{\tiny $\frac{3}{8}$};
\draw [gray, line width=0.2cm] (2.95,0.6) node[below, black]{\tiny $\frac{3}{4}$} -- (3.075,0.6) node[below, black]{\tiny $\frac{7}{8}$};
\draw [gray, line width=0.2cm] (2.95,0.8) node[below, black]{\tiny $\frac{1}{8}$} -- (3.075,0.8) node[below, black]{\tiny $\frac{1}{4}$};
\draw [gray, line width=0.2cm] (2.95,1) node[below, black]{\tiny $\frac{5}{8}$} -- (3.075,1) node[below, black]{\tiny $\frac{3}{4}$};
\draw [gray, line width=0.2cm] (2.95,1.2) node[below, black]{\tiny $\frac{3}{8}$} -- (3.075,1.2) node[below, black]{\tiny $\frac{1}{2}$};
\draw [gray, line width=0.2cm] (2.95,1.4) node[below, black]{\tiny $\frac{7}{8}$} -- (3.075,1.4) node[below, black]{\tiny $1$};
\end{tikzpicture}

\caption{Partial graph of the dyadic odometer}
\end{figure}

 Inductively, the transformation $T$ associated to this construction and called {\it Kakutani-von Neumann odometer} is defined on the countable sequence of intervals considered above by 
$$T_k(x)=x-1+2^{-k}+2^{-k-1}\qquad {\rm for}\ x\in[1-2^{-k},1-2^{-k-1}[\ {\rm and}\ k\geq 0\ .$$

Figure 3 shows the transformation $T$ on $[0,1/2[\cup [1/2,3/4[\cup [3/4,7/8[$.

\begin{figure}[h!]

\begin{center}
\begin{tikzpicture}[scale=6]
\draw (0,0) node[below, black]{\scriptsize 0} -- (1,0) node[below, black]{\scriptsize 1} -- (1,1) -- (0,1)node[left]{\scriptsize1} -- (0,0);
\draw [thick] (0,0.5) node[left, black]{\scriptsize $\frac{1}{2}$} -- (0.5,1);
\draw [thick] (0.5,0.25) -- (0.75,0.5);
\draw [thick] (0.75,0.125) -- (0.875,0.25);
\draw [dashed] (0.5,0) node[below, black]{\scriptsize $\frac{1}{2}$} -- (0.5,1); 
\draw [dashed] (0.75,0) node[below, black]{\scriptsize $\frac{3}{4}$} -- (0.75,1);
\draw [dashed] (0.875,0) node[below, black]{\scriptsize $\frac{7}{8}$} -- (0.875,1);
\draw [dashed] (0,0.5)  -- (1,0.5);
\draw [dashed] (0,0.25) node[left, black]{\scriptsize $\frac{1}{4}$}  -- (1,0.25);
\draw [dashed] (0,0.125) node[left, black]{\scriptsize $\frac{1}{8}$}  -- (1,0.125);
\end{tikzpicture}
\end{center}

\caption{Partial graph of the von Neumann-Kakutani transformation}
\end{figure}
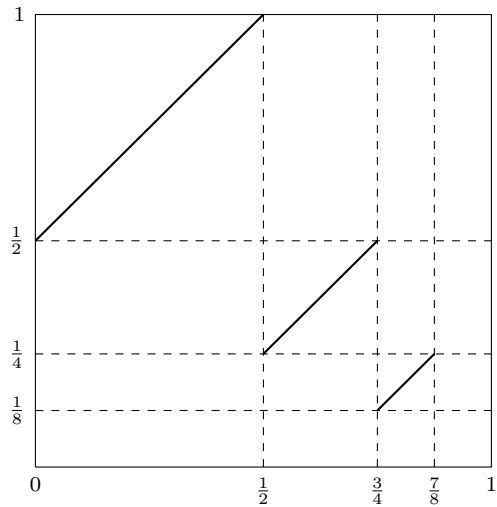
Lambert \cite{Lambert} proved that the sequence obtained as the orbit of $0$ under the von Neumann-Kakutani transformation $T$ is the van der Corput sequence, $(\phi_2(n))_{n\geq 0}$.

The von Neumann-Kakutani odometer is one of the simplest examples of a  rank one transformation  and, consequently, it is ergodic, even if ergodicity can be proved directly.
We now consider a slightly different variation of the cutting-stacking technique according to a substitution $\sigma$.\\

Let $A$ be a non empty set, called alphabet, of $s$ elements, called
letters. Usually we take $A = \{1,\dots , s\}$. A word $w$ of length $|w| = n$ on $A$ is an ordered string $w_1\dots w_n$ of $n$ letters $w_j$ in $A$. A word of length $0$ is called empty word and denoted by $\wedge$. For any letter $a$, the number of occurrences of $a$ in $w$ is denoted $|w|_a$. Hence $|w| =\sum_{a\in A}|w|_a$.\\ We denote by $A^*$ the set of words over the alphabet $A$, equipped with the concatenation law $(v_1\dots v_m) (w_1\dots w_n) =( v_1\dots v_mw_1 \dots w_n)$. $A^*$ is the free monoid generated by $A$, where the empty word is the neutral element.
\begin{definition}
A monoid endomorphism $\sigma : A^* \rightarrow A^*$ is called a substitution if $|\sigma(a)| \geq 1$ for all letters $a \in A$. If $\sigma(a) =\wedge$ for at least one letter, we say that
$\sigma$ is a pseudo-substitution.
\end{definition}
If $A = \{1,\dots , s\}$, then the following matrix $M(\sigma)$
\begin{equation*}
M(\sigma) =
 \begin{pmatrix}
  |\sigma(1)|_1 &  \cdots & |\sigma(s)|_1 \\
  \vdots  & \ddots & \vdots  \\
  |\sigma(1)|_s & \cdots & |\sigma(s)|_s
 \end{pmatrix}
\end{equation*}
is called the companion matrix of $\sigma$. It will play a fundamental role. Let $\mathcal{P}_s$ be the set of positive column vectors $l \in \mathbb{R}^s$, i.e. such that all entries $l_i$ of $l$ are
positive. For any couple $(l, l')$ of vectors in $\mathcal{P}_s$, we say that $l'$ derives from $l$ by
$\sigma$, and we write $l \sigma \rightarrow l'$, if the relation
\begin{equation*}
 l=M(\sigma)l'
\end{equation*}
holds.\\

Now let us consider the cutting-stacking process introduced in \cite{GHL}, where starting from a set of columns $S= \{C_1,\dots ,C_s\}$ we can build another set of columns $ S' = \{C'_1,\dots ,C'_s\}$, according to $\sigma$.\\
To do so, let $l$ be the column vector in $\mathcal{P}_s$ with entries $l_i = l(C_i)$ and assume that there exists $l' \in \mathcal{P}_s$ such that $l'$ derives from $l$ by
$\sigma$.  Now cut each column $C_j$ in order to create a set of $|\sigma|_j = \sum_{k=1}^s |\sigma(k)|_j$ (sum of
entries of the $j$-th line of $M(\sigma)$) sub-columns $S_j = \{C_{j,1},\dots ,C_{j,|\sigma|_j}\}$ such that $|\sigma(k)|_j$ of them have width $l'_k$. Then, for each $k$, build the column $C'_k$ by stacking $\sum_{k=1}^s|\sigma(k)|_j$
sub-columns such that
\begin{itemize}
\item $|\sigma(k)|_j$ sub-columns come from the sub-columns of width $l'_k$ in $S_j$;
\item from the bottom to the top the column $C'_k$ is built according to the word $\sigma(k) = \sigma_{k,1}\dots \sigma_{k,|\sigma(k)|}$
\begin{equation*}
C'_k=T_{\sigma_{k,1}}^{(1)}\ast T_{\sigma_{k,2}}^{(2)}\ast\dots T_{\sigma_{k,|\sigma(k)|}}^{(|\sigma(k)|)}\ ,
\end{equation*}
where $T_{\sigma_{k,j}}^{(j)}$ is a column from $S_{\sigma_{k,j}}$ not used yet. In the standard construction, we select the successive $T_{\sigma_{k,j}}^{(j)}$ in $C'_k$ from left to right.\\
This construction is not unique but at least a standard one exists due to
$l=M(\sigma)l'$. When we use a standard construction we say that the couple
$(S', l')$ derives from $(S, l)$ by taking into account the derivation $l=M(\sigma)l'$.
\end{itemize}
At this point we want to define the transformation $T_\sigma$ obtained by the iteration of the above derivation process. To make it possible we need to consider a particular class of substitutions on $A$, namely adapted substitutions and introduce some definitions.
\begin{definition}
A letter $a\in A$ is said to be expansive for the substitution $\sigma$ if the increasing sequence $n\mapsto |\sigma^n(a)|$ is unbounded.\\
We denote by $E(\sigma)$ the set of all expansive letters for $\sigma$ in $A$.
\end{definition}
\begin{definition}
A substitution $\sigma$ is called adapted if $E(\sigma) \neq \emptyset$ and for all expansive letters $a$ and all letters $x\in A$, there exists an integer
$k\geq 1$ such that $|\sigma^k(a)|_x \geq 1$.
\end{definition}
\begin{definition}
Let $\sigma$ be an adapted substitution. The period $h$ of $\sigma$ is the
period of the companion matrix of $\sigma_{E(\sigma)}$. Therefore, $h$ is given from any expansive letter $a$ by
\end{definition}
\begin{equation*}
h = {\rm gcd}\{k \geq 1 ; |\sigma^k_E(a)|_a \geq 1\}\ .
\end{equation*}
Now we can state a useful theorem that can be seen as a generalization of the Perron-Frobenius Theorem, already considered at the beginning of this section.
\begin{theorem}
Let $\sigma$ be an adapted substitution with companion matrix $M = M(\sigma)$. Then
\begin{enumerate}
\item $M$ has an eigenvalue $\theta > 1$ and $\theta \geq |\lambda|$ for all eigenvalues $\lambda$ of $M$,
\item $\theta$ has an eigenvector with positive entries,
\item $\theta$ is simple.
\end{enumerate}
\end{theorem}
The eigenvalue $\theta$ in the theorem above is called the dominant eigenvalue of
$M(\sigma)$ or of $\sigma$ and the unique eigenvector associated with $\theta$ such that the sum of its entries is equal to $1$ is called the unitary dominant eigenvector of $M(\sigma)$.
If $\sigma$ has p expansive letters we may assume that they form the set
$E = \{1, \dots , p\}$ so that the matrix $M(\sigma)$ takes the form
\begin{equation*}
M =
 \begin{pmatrix}
 M(\sigma_E) &  0 \\
 B & C
 \end{pmatrix}\ .
\end{equation*}
The dominant eigenvalue $\theta$ of $M$ is also the dominant eigenvalue of $M(\sigma_E)$ and
the first $p$ entries of the dominant eigenvector $l$ of $M$, after normalisation is the
dominant eigenvector of $M(\sigma_E)$.\\

So with this notation at hand we can define the interval exchange $T_\sigma : [0, 1[ \rightarrow [0, 1[$ by a cutting-stacking process associated to an adapted substitution $\sigma$ with dominating eigenvalue $\theta$.\\
The positive vector $l\in \mathcal{P}_s$ will be exactly the unitary dominant eigenvector of $M(\sigma)$.\\
Then the required transformation will be 
\begin{equation*}
T_\sigma : [0,1[\setminus {\rm top}(\Sigma)\rightarrow [0,1[\ ,
\end{equation*}
where $\Sigma= (S_n)_{n\geq 0}$.\\

We now analyze the example given in \cite[\S 3.7.4]{GHL}.
\begin{example}[The Fibonacci transformation]
The Fibonacci substitution $\sigma_{Fib}$ on the alphabet $A=\{0,1\}$ is defined as
\begin{equation*}
\sigma_{Fib}(0)=01 \qquad \sigma_{Fib}(1)=0\ .
\end{equation*}
Since the iteration of each letter gives rise to an infinite sequence, the substitution is adapted and the associated companion matrix is
\begin{equation*}
M(\sigma_{Fib}) =
 \begin{pmatrix}
 1 &  1 \\
 1 & 0 
 \end{pmatrix}\ .
\end{equation*}
The dominant eigenvalue is $\theta=\frac{1+\sqrt{5}}{2}$, the golden ratio, and the corresponding normalized dominant eigenvector is $l=\binom{\alpha}{1-\alpha}$, where $\alpha=\theta^{-1}=\theta -1=$\linebreak $\frac{\sqrt{5}-1}{2}$.\\
Moreover, one can notice that the following relation
\begin{equation*}
1 = F_{n+1}\alpha^{n+1} + F_n\alpha^{n+2}  
\end{equation*}
where $(F_n)_{n\geq 0}$ is the usual Fibonacci sequence, holds. So, by induction we also have that
\begin{equation*}
\alpha^{n+1} = (−1)^n(F_n\alpha- F_{n-1}) = ||F_n\alpha ||\ .
\end{equation*}
Now we want to get a transformation $T_{Fib}$, called the Fibonacci transformation, by means of the cutting-stacking derivations
\begin{equation*}
(S_n, l^{(n)})\xrightarrow{\sigma_{Fib}} (S_{n+1}, l^{(n+1)})\ .
\end{equation*}
\end{example}
According to the general procedure defined above, we have that our first set of columns is 
\begin{equation*}
S_0=\left\{C_1^{(0)}, C_2^{(0)}\right\}= \left\{[1-\alpha, 1[, [0,1-\alpha[\right\}\ ,
\end{equation*}
with $C_1^{(0)}, C_2^{(0)}$ of width $\alpha$ and $1-\alpha$, respectively.\\
At this point, applying the cutting-stacking derivation to $l=M(\sigma)l^{(1)}$ we get $l^{(1)}= \binom{\alpha^2}{\alpha^3}$. Then each new column $C_1^{(0)'}=C_1^{(1)}, C_2^{(0)'}=C_2^{(1)}$ will be obtained by stacking respectively 
\begin{equation*}
|\sigma_{Fib}(1)|_1+|\sigma_{Fib}(2)|_1=2
\end{equation*}
and
\begin{equation*}
|\sigma_{Fib}(1)|_2+|\sigma_{Fib}(2)|_2=1
\end{equation*}
sub-columns such that $w\left(C_1^{(1)}\right)=l_1^{(1)}=\alpha^2$ and $w\left(C_2^{(1)}\right)=l_2^{(1)}=\alpha^3$. Therefore
\begin{equation*}
S_1=\left\{C_1^{(1)}, C_2^{(1)}\right\}=
\left\{\begin{matrix}
  [0,1-\alpha[ &  \\
  [1-\alpha, 2-2\alpha[ & [2-2\alpha,1[
 \end{matrix}\right\}\ .
\end{equation*}
Let us produce one more step before giving the general description of the cutting-stacking process.\\
Applying again the standard derivation we get $l^{(2)}= \binom{\alpha^3}{\alpha^4}$. Then each new column $C_1^{(1)'}=C_1^{(2)}, C_2^{(1)'}=C_2^{(2)}$ will be obtained by stacking respectively 
\begin{equation*}
|\sigma_{Fib}(1)|_1+|\sigma_{Fib}(2)|_1=2
\end{equation*}
and
\begin{equation*}
|\sigma_{Fib}(1)|_2+|\sigma_{Fib}(2)|_2=1
\end{equation*}
sub-columns such that $w\left(C_1^{(2)}\right)=l_1^{(2)}=\alpha^3$ and $w\left(C_2^{(2)}\right)=l_2^{(2)}=\alpha^4$. Therefore
\begin{equation*}
S_1=\left\{C_1^{(1)}, C_2^{(1)}\right\}=
\left\{\begin{matrix}
  [2-2\alpha,1[ &  \\
  [2-3\alpha, 1-\alpha[ & [0,2-3\alpha[\\
  [3-4\alpha, 2-2\alpha[ & [1-\alpha, 3-4\alpha[
 \end{matrix}\right\}\ .
\end{equation*}
In general, the standard derivation depends on the parity of $n$ and at every step we have that $S_n$ is composed of two columns $C_1^{(n)}$ and $C_2^{(n)}$ of height respectively $F_{n+1}$ and $F_n$ and width $w\left(C_1^{(n)}\right)=\alpha^{n+1}$ and $w\left(C_2^{(n)}\right)=\alpha^{n+2}$, respectively.\\
Since the position of the intervals in the columns depends on the parity of $n$, we have to describe top$(C_1^{(n)})$ and top$C_1^{(n)}$ for $n=2m$ and $n=2m+1$. If we denote by $\{x\}$ the fractional part of $x$, we get that
\begin{equation*}
{\rm top}\left(C_1^{(2m)}\right)= \left[\{-F_{2m}\alpha\}, 1 \right[\qquad {\rm top}\left(C_2^{(2m)}\right)= \left[0, \{-F_{2m+1}\alpha\}\right[
\end{equation*}
and
\begin{equation*}
{\rm top}\left(C_1^{(2m+1)}\right)= \left[0, \{-F_{2m+1}\alpha\} \right[\qquad {\rm top}\left(C_2^{(2m+1)}\right)= \left[\{-F_{2m+2}\alpha\},1\right[\ .
\end{equation*}
The transformation $T_{Fib}$ obtained in this way is not defined at the point $0$. Moreover it is not difficult to see that its explicit expression is defined by two families of transformations
\begin{equation*}
T_{{Fib}_{2k}}: {\rm top}\left(C_1^{(2k)}\right)\longrightarrow {\rm bot}\left(C_2^{(2k)}\right)
\end{equation*}
\begin{equation*}
T_{{Fib}_{2k}}(x)=x+\alpha
\end{equation*}
and
\begin{equation*}
T_{{Fib}_{2k+1}}: {\rm top}\left(C_2^{(2k+1)}\right)\longrightarrow {\rm bot}\left(C_1^{(2k+1)}\right)
\end{equation*}
\begin{equation*}
T_{{Fib}_{2k+1}}(x)=x-\alpha^2\ .
\end{equation*}
Extending it by continuity to $0$, the graph of $T_{Fib}$ is shown in the following picture.
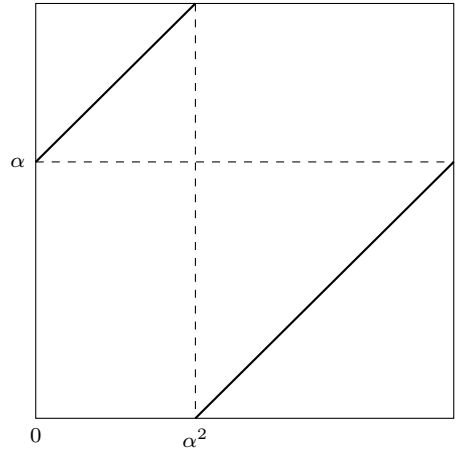
\begin{figure}[h!]
\begin{center}
\begin{tikzpicture}[scale=5.5]
\draw (0,0) node[below, black]{\scriptsize $0$} rectangle (1,1); 
\draw [thick] (0.38196,0) node[below, black]{\scriptsize $\alpha^2$} -- (1,0.618);
\draw [thick] (0,0.618)node[left, black]{\scriptsize $\alpha$} -- (0.38196,1);
\draw [dashed] (0,0.61803)  -- (1 ,0.61803); 
\draw [dashed] (0.38196,1)  -- (0.38196,0); 
\end{tikzpicture}
\end{center}
\caption{Fibonacci transformation}
\end{figure}
As we have already pointed out, this construction is not the only one possible, starting from the Fibonacci substitution, but is the standard one in the sense that it is obtained via the standard derivation.\\
It is worth considering that this transformation is nothing else but the $T_\alpha$, for $\alpha=\frac{\sqrt{5}-1}{2}$ defined after Definition \ref{rotation}.
\subsection*{Topological dynamics}
So far, we have been studying a measurable map $T$ defined on a probability space $(X,\mathcal{A}, \mu)$.
We have asked whether the given measure $\mu$ is invariant or ergodic. Now we focus on the space $\mathcal{M}^T(X)$ of all probability measures on the compact metric space $X$, which are invariant under a continuous transformation $T : X \rightarrow X$.\\
Any continuous map $T : X \rightarrow X$ induces a continuous map
\begin{equation*}
 T_* : \mathcal{M}(X) \rightarrow \mathcal{M}(X)
\end{equation*}
defined by $T_∗(\mu)(A) = \mu(T^{-1}A)$ for any Borel set $A \subset X$. Each point $x \in X$ defines a measure $\delta_x$ (called the Dirac measure) by
\begin{equation*}
\delta_x(A)=\begin{cases}
             1 & \ {\rm if\ } x\in A\\
             0 & \ {\rm if\ } x\notin A\ .
            \end{cases}
\end{equation*}
We claim that $T_∗(\delta_x) = \delta_{T(x)}$ for any $x \in X$. To see this, let $A \subset X$ be any measurable set, and notice that
\begin{equation*}
(T_∗\delta_x)(A) = \delta_x(T^{-1}A) = \delta_{T(x)}(A)\ .
\end{equation*}
This suggests that we should think of the space of measures $\mathcal{M}(X)$ as generalized points, and the transformation $T_∗ : \mathcal{M}(X) \rightarrow \mathcal{M}(X)$ as a natural extension of the map $T$ from the copy $\{\delta_x | x \in X\}$ of $X$ to the larger set $\mathcal{M}(X)$. For $f \in \mathcal{C}(X)$ and $\mu \in \mathcal{M}(X)$,
\begin{equation*}
\int_X f d(T_*\mu)=\int_X f\circ Td\mu\ ,
\end{equation*}
and this property characterizes $T_∗$.\\
The map $T_∗$ is continuous and affine, so the set $\mathcal{M}^T(X)$ of $T$-invariant measures is a closed convex subset of $\mathcal{M}(X)$.\\
Given a continuous mapping $T : X \rightarrow X$ on a compact metric space, a natural question, is whether invariant measures necessarily exist, i.e. whether $\mathcal{M}^T(X)\neq \emptyset$. The next result and its corollary show that this is indeed the case.
\begin{theorem}
Let $T : X \rightarrow X$ be a continuous map of a compact metric
space, and let $(\nu_n)_{n\in \mathbb{N}}$ be any sequence in $\mathcal{M}(X)$. Then any weak*-limit of the sequence $(\mu_n)_{n\in\mathbb{N}}$ defined by
\begin{equation*}
\mu_n= \frac{1}{n}\sum_{j=0}^{n-1}T_*^j\nu_n
\end{equation*}
is a member of $\mathcal{M}^T(X)$.
\end{theorem}
\begin{corollary}[Krylov-Bogolioubov]
Under the hypotheses of the previous Theorem, $\mathcal{M}^T(X)$ is non-empty.
\end{corollary}
Thus $\mathcal{M}^T(X)$ is a non-empty compact convex set, since convex combinations of elements of $\mathcal{M}^T(X)$ belong to $\mathcal{M}^T(X)$. It follows that $\mathcal{M}^T(X)$ is an
infinite set unless it is a singleton.\\
In general, it is difficult to identify measures with specific properties, but the ergodic measures are readily characterized in terms of the geometry of the space of invariant measures.\\
We denote by $\mathcal{E}(X, T) \subset \mathcal{M}^T(X)$ the set of ergodic $T$-invariant probability measures on $X$.\\
The next result will allow us to show that ergodic measures for continuous transformations on compact metric spaces always exist.
\begin{theorem}\label{uniq_erg_thm}
Let $T$ be a continuous transformation of a compact metric space $X$ equipped with the Borel $\sigma$-algebra $\mathcal{B}$. The following are equivalent
\begin{enumerate}
\item a $T$-invariant probability measure $\mu$ is ergodic;
\item $\mu$ is an extremal point of $\mathcal{M}^T(X)$.
\end{enumerate}
\end{theorem}
A natural distinguished class of transformations is formed by those for which there is
only one invariant Borel measure. This measure is automatically ergodic, and
the uniqueness of this measure has several powerful consequences.
\begin{definition}

Let $X$ be a compact metric space and let $T:X\longrightarrow X$ be a
continuous map. Then $T$ is said to be uniquely ergodic if $\mathcal{M}^T(X)=\{\mu\}$. The dynamical system $(X,\mathcal{A},\mu,T)$ is said to be uniquely ergodic.
\end{definition}
\begin{theorem}
For a continuous map $T:X\longrightarrow X$ on a compact metric
space, the following properties are equivalent
\begin{enumerate}
\item $T$ is uniquely ergodic.
\item $|\mathcal{E}(X,T)|=1$.
\item For every continuous function $f: X\longrightarrow X$,
\begin{equation}\label{unique_erg}
\lim_{N\to\infty}\frac{1}{N}\sum_{n=0}^{N-1}f(T^nx)= C_f\ ,
\end{equation}
where $C_f$ is a constant independent of $x$.
\item For every continuous function $f: X\longrightarrow X$, the convergence \eqref{unique_erg} is uniform across $X$.
\item The convergence \eqref{unique_erg} holds for every $f$ in a dense subset of $\mathcal{C}(X)$.
\end{enumerate}
Under any of these assumptions, the constant $C_f$ in \eqref{unique_erg} is $\int_X f d\mu$, where $\mu$ is the unique invariant measure.
\end{theorem}

An immediate consequence of the previous theorem is that the orbit $(T^nx)_{n\in\mathbb{N}}$ of $x$ under $T$ is uniformly distributed for every $x\in X$.\\

Let $X =S^1$ and $T_\alpha :X\longrightarrow X$ be the irrational rotation on the unit circle. We have already proved that the Lebesgue measure $\lambda$ is an ergodic $T_\alpha$-invariant measure. Furthermore the following result says that the Lebesgue measure is the only invariant measure, i.e. $T_\alpha$ is uniquely ergodic.
\begin{lemma}
An irrational rotation of a circle $T_\alpha$ is uniquely ergodic and the only $T_\alpha$-invariant measure is the Lebesgue measure.
\end{lemma}

One important question concerns the unique ergodicity of the cartesian product of uniquely ergodic systems, that is under which conditions the product of uniquely ergodic systems is uniquely ergodic. As for the product of ergodic systems, it turns out that the spectral analysis of the associated operators is very useful. More precisely, again we need to require that the spectra of the associated operators intersect only at 1, as explained by the following 
\begin{theorem}\label{product}
 Let $\mathcal{T}_i=(X_i,T_i), ~i=1,\ldots,s,$ be uniquely ergodic dynamical
systems. Then the dynamical system $\mathcal{T}_1 \times \ldots \times
\mathcal{T}_s$ is uniquely ergodic if and only if for all $i,j \in \{1, \ldots,
s\}, ~i \neq j,$ the discrete parts of the spectra of $T_i$ and $T_j$ intersect
only at 1.
\end{theorem}
\subsection{Systems of numeration}
We have already considered some examples of numeration systems, such as the Ostrowski expansion or the $b$-adic and $\beta$-adic expansions. Roughly speaking, a numeration system is a coding of the elements of a set with a (finite or infinite) sequence of digits. The result of the coding, the sequence of digits, is a representation of the element.\\
In this paragraph we give a formal and general definition. For a complete survey on numeration systems from a dynamical viewpoint we refer to \cite{BBLT}.
\begin{definition}
A numeration system (resp. a finite numeration system) is a triple $(X, I,\varphi)$, where $X$ is a set, $I$ a finite or countable set, and $\varphi$ an injective map $\varphi: X \rightarrow I^{\mathbb{N}^*}$, defined by
\begin{equation*}
\varphi(x)=(\epsilon_n(x))_{n\geq 0}
\end{equation*}
The map $\varphi$ is the representation map, and $\varphi(x)$ is the representation of $x \in X$.
\end{definition}
Given a numeration system $(X, I,\varphi)$, one can define an expansion as a map $\psi : I^{\mathbb{N}^*}\rightarrow X$ such that $\psi \circ \varphi(x) = x$ for all $x \in X$. An
expansion of an element $x \in X$ is an equality $x =\psi(y)$; it is a proper expansion if $y = \varphi(x)$, and an improper expansion otherwise.\\
However this definition is still unsatisfactory since for instance it is not clear how the digits $\epsilon_i$ are produced and how the representation map is constructed. \\
The dynamical point of view on numeration systems lies precisely in these requirements. So we need to find a map whose iteration gives the sequences of digits of the representation. In order to define this map in a precise way we need to introduce the concept of fibred system introduced by Schweiger \cite{Schweiger}.
\begin{definition}
A fibred system is a set $X$ and a transformation $T : X \rightarrow X$
for which there exist a finite or countable set $I$ and a partition $X = \bigsqcup_{i\in I} X_i$ of $X$ such that the restriction $T_i$ of $T$ on $X_i$ is injective, for any $i \in I$. This yields a well defined map $\epsilon: X \rightarrow I$ that associates the index $i$ with $x \in X$ such that $x \in X_i$.
\end{definition}
Now it is easy to see that the definition of numeration system given above can have a dynamical description by means of a fibred system. The new definition of numeration system will be called fibred numeration system.\\
Briefly, assume that $(X,T)$ is a fibred system as above. Let
$\varphi: X \rightarrow I^{\mathbb{N}^*}$
be defined by $\varphi(x) = (\epsilon(T^nx))_{n\geq 0}$. We will write $\epsilon_n = \epsilon\circ T^{n-1}$ for short. Let $\sigma$ be the (right-sided) shift operator on $I^{\mathbb{N}^*}$. Then we get the following commutative diagram
\begin{center}
\begin{tikzpicture}[scale=2.5]
\node (A) at (0,1) {$X$};
\node (B) at (1,1) {$X$};
\node (C) at (0,0) {$I^{\mathbb{N}^*}$};
\node (D) at (1,0) {$I^{\mathbb{N}^*}$};
\path[->,font=\scriptsize,>=angle 90]
(A) edge node[above]{$T$} (B)
(A) edge node[left]{$\varphi$} (C)
(B) edge node[right]{$\varphi$} (D)
(C) edge node[below]{$\sigma$} (D);
\end{tikzpicture}
\end{center}
Formally, we have the following
\begin{definition}
Let $(X,T)$ be a fibred system and $\varphi: X \rightarrow I^{\mathbb{N}^*}$ be defined
by $\varphi(x) = (\epsilon(T^nx))_{n\geq 0}$. If the function $\varphi$ is injective (i.e., if $(X, I,\varphi)$ is a numeration system), we call the quadruple $\mathcal{N} = (X,T, I,\varphi)$ a fibred numeration system. Then $I$ is the set of digits of the numeration system;
the map $\varphi$ is the representation map and $\varphi(x)$ the $\mathcal{N}$-representation of $x$. In general, the representation map is not surjective. The set of prefixes of $\mathcal{N}$-representations is called the language $\mathcal{L} = \mathcal{L}(\mathcal{N})$ of the fibred numeration system, and its elements are said to be admissible. The admissible sequences
are defined as the elements $y \in I^{\mathbb{N}^*}$ for which $y = \varphi(x)$ for some $x \in X$.
\end{definition}
The representation map transports cylinders from the product space $I^{\mathbb{N}^*}$ to $X$, and for $(i_0, i_1,\dots , i_{n-1}) \in I^n$, one may define the cylinder 
\begin{equation*}
X \supset C(i_0, i_1, \dots, i_{n-1})= \bigcap_{j=0}^{n-1}T^{-j}(X_{i_j}) = \varphi^{-1}[i_0, i_1,\dots , i_{n-1}]\ .
\end{equation*}
Moreover, the assumption that the restriction of $T$ to
$X_i$ is injective says that the application $x \mapsto (\epsilon(x),T(x))$ is itself injective. It is a necessary condition for $\varphi$ to be injective, and $\mathcal{N}$ is a fibred numeration system if and only if
\begin{equation*}
\forall x\in X \qquad \bigcap_{n\geq 0}C(\epsilon_1(x),\epsilon_2(x),\dots, \epsilon_n(x))=\{x\}
\end{equation*}
In the case when $X$ is a metric space, this last condition is satisfied if for any admissible sequence $(i_1, i_2,\dots , i_n,\dots)$, the diameter of the cylinders $[i_1, i_2,\dots, i_n]$ tends to zero when $n$ tends to infinity. In this case, every closed subset $F$ of $X$ can be expressed as
\begin{equation*}
F=\bigcap_{n=1}^\infty \bigcup_{\substack{
   (i_1, \dots, i_n)\in\mathcal{L}\\
   [i_1, \dots, i_n]\cap F\neq \emptyset}} C(i_1, \dots, i_n)\ ,
\end{equation*}
which proves that the $\sigma$-algebra $\mathcal{A}$ generated by the cylinders is the Borel algebra. In general, $T$ is $\mathcal{A}$-measurable.\\
We can also define a topological structure on the fibred numeration system in the following way.
\begin{definition}
For a fibred numeration system $\mathcal{N} = (X,T, I,\varphi)$, with a Hausdorff topological space $I$ as digit set, the associated $\mathcal{N}$-compactification $X_{\mathcal{N}}$ is defined as the closure of $\varphi(X)$ in the product space $I^{\mathbb{N}^*}$.
\end{definition}
Now we can see that some of the examples of numeration systems already mentioned in the section can be obtained with a particular choice of $X, I, T$.
\begin{example}[$b$-adic representation]
Let $X = \mathbb{N}$, $I = \{0, 1,\dots , b-1\}$, $X_i = i+q\mathbb{N}$. Then, according to the definition of fibred system we have
\begin{equation*}
\epsilon(n)\equiv n({\rm mod}\ q)\ .
\end{equation*}
Let $T : X\rightarrow X$ be defined by 
\begin{equation*}
T(n) = \frac{n-\epsilon(n)}{q}\ .
\end{equation*}
So we have a fibred numeration system with language, set of representations and compactification given by
\begin{equation*}
\mathcal{L}_b=\bigcup_{n\geq 0}\{0,1,\dots, b-1\}^n\ ,
\end{equation*}
\begin{equation*}
\varphi(X) = I^{(\mathbb{N})} =\{(i_0,\dots , i_{n-1}, 0, 0,\dots); n\in \mathbb{N}, i_j \in \{0, 1, \dots , b-1\}\}\ , 
\end{equation*}
\begin{equation*}
X_{\mathcal{N}} = \{0, 1,\dots , b-1\}^{\mathbb{N}}\ .
\end{equation*}
On the same set $I$ of digits we can also consider the space $X=[0,1[$ with partition given by
\begin{equation*}
[0,1[=\bigsqcup_{i\in I} X_i=\bigsqcup_{i\in I} \left[\frac{i}{b},\frac{i+1}{b} \right[
\end{equation*}
and with transformation
\begin{equation*}
T(x) = qx-\lfloor qx\rfloor\ .
\end{equation*}
The language is again $\mathcal{L}_b$ and the compactification $\mathbb{Z}_b$, the compact group of $b$-adic integers. The set of representations is the whole product space without the sequences ultimately equal to $b-1$.
\end{example}
In the last part of this section we want to describe a popular extension of this method. It consists in changing the base $b$ at any step: this method is called Cantor expansion. \\In this way, the $b$-adic expansion can also be obtained beginning with the most significant digit, using the so called greedy algorithm.\\
This procedure still gives rise to a numeration system, but not to a fibred one. Moreover, this way of producing expansions of nonnegative integers yields a more general concept than the Cantor expansion, the $G$-scale, which is the most general possible way of representing nonnegative integers based on the greedy algorithm. As we have seen, given a fibred numeration system $\mathcal{N}$, we can consider its associated $\mathcal{N}$-compactification. For a $G$-scale, a compactification can be also built, but it is not possible in general to extend the addition from $\mathbb{N}$ to it in a reasonable way. Nevertheless, the addition by $1$ on $\mathbb{N}$ extends naturally and gives a dynamical system called odometer.

\begin{example}[$\beta$-expansions and $\beta$-adic van der Corput sequences]\label{ex_beta_adic}
The basic idea to obtain $\beta$-expansions is to replace $b$ by any real number $\beta > 1$. Then let $X = [0, 1[$ and $T_\beta:[0,1[\rightarrow [0,1[$ be the $\beta$-transformation defined by $T_\beta(x)=\beta x$ mod $1$. Moreover, let $I =\{0, 1,\dots, \lceil\beta\rceil - 1\}$. The we have the following fibred system $([0,1[, T_\beta)$
\begin{center}
\begin{tikzpicture}[scale=2.5]
\node (A) at (0,1) {$[0,1[$};
\node (B) at (1,1) {$[0,1[$};
\node (C) at (0,0) {$I^{\mathbb{N}^*}$};
\node (D) at (1,0) {$I^{\mathbb{N}^*}$};
\path[->,font=\scriptsize,>=angle 90]
(A) edge node[above]{$T_\beta$} (B)
(A) edge node[left]{$\varphi$} (C)
(B) edge node[right]{$\varphi$} (D)
(C) edge node[below]{$\sigma$} (D);
\end{tikzpicture}
\end{center}
where the representation map $\varphi$ of the fibred system is defined by
\begin{equation*}
x=\sum_{n=0}^\infty \frac{\epsilon_n}{\beta^n} \Leftrightarrow \varphi(x)=(\epsilon_0,\epsilon_1,\dots)\in I^{(\mathbb{N})}\ .
\end{equation*}

Parry \cite{Parry} proved that the set of admissible sequences $\varphi(X)$ is characterised in terms of one particular $\beta$-expansion. For $x \in [0, 1[$, set $d_\beta(x) = \varphi(X)$. In particular, let $d_\beta(1) = (t_n)_{n\geq 1}$. We then set $d^*_\beta(1) = d_\beta(1)$, if $d_\beta(1)$ is infinite, and
\begin{equation*}
d_\beta^*(1)=(t_1\dots t_{m-1},(t_m-1))^\omega\ ,
\end{equation*}
if $d_\beta(1)=t_1\dots t_{m-1}t_m0^\omega$ is finite, with $t_m\neq 0$. The set $\varphi(X)$ of $\beta$-expansions of real numbers in $[0, 1[$ is exactly 
\begin{equation*}
\varphi(X)=\{y\in I^{(\mathbb{N})} : \forall k\geq 1, \sigma^k y\prec d_\beta^*(1)\ .
\end{equation*}
where $\omega \prec \psi$ in the sense of the lexicographical order.\\
The set $X_{\mathcal{N}}=\overline{\varphi([0,1[)}$ is called one-sided $\beta$-shift. \\
Numbers $\beta$ such that $d_\beta(1)$ is ultimately periodic are called Parry numbers and those such that $d_\beta(1)$ is finite are called simple Parry numbers.\\
Then simple Parry numbers are those which produce improper expansions. To any element in $X_{\mathcal{N}}$, we can associate the number
\begin{equation*}
\psi_\beta(x)=\sum_{n=0}^\infty \epsilon_n\beta^{-n-1}\ .
\end{equation*}
Then $\psi_\beta(x)\in [0, 1[$ and numbers with two expansions are exactly those with finite expansion
\begin{equation*}
\psi_\beta(\epsilon_1 \dots \epsilon_{k-1}\epsilon_k 0^\omega) = \psi_\beta((\epsilon_1\dots \epsilon_{k-1}(\epsilon_k -1)d_\beta^*(1))\ .
\end{equation*}
If $\beta$ is a Pisot number, that is it is a real algebraic integer greater than 1 such that all its Galois conjugates are less than 1 in absolute value, then every element of $\mathbb{Q}(\beta)\cap [0, 1[$ admits
a ultimately periodic expansion (see e.g. \cite{K_Schmidt}). Hence $\beta$ is either a Parry number or a simple Parry number.\\

Now we can see how to construct $\beta$-adic van der Corput sequences. These sequences, denoted by $N_\beta$, were introduced for the first time by Ninomiya \cite{Ninomiya} in 1997. Their construction is based on the $\beta$-adic transformation and the fibred numeration system just defined.\\
In order to define them we need to consider the following sets:
\begin{itemize}
\item $X_\beta(n)$ is the set of admissible sequences of length $n$,
\item $Y_\beta(n)=\bigcup_{i=0}^nX_\beta(i)$,
\item $F_\beta(n)=\#X_\beta(n)$,
\item $G_\beta(n)=\sum_{i=0}^nF_\beta(i)=\#Y_\beta(n)$.
\end{itemize}
Now for an arbitrary positive integer $n$, define $l_n$ to satisfy $G_\beta(l_n)\leq n < G_\beta(l_n+1)$. Then, reverse the lexicographical order on the digits of $X_\beta(l_n+1)$ for every $n$. Then the sequence $N_\beta$ is defined as follows:
\begin{equation*}
N_\beta=\left\{ \psi_\beta(\omega^{l_n+1}_{n-G_\beta(l_n)+1})\right\}_{n=1}^\infty\ ,
\end{equation*}
where $\omega^{l_n+1}\in X_\beta(l_n+1)$.\\
In \cite{Ninomiya}, the following example with $\beta=\frac{\sqrt{5}+1}{2}$ is given explictly. Let us derive the first elements of the sequence $N_\beta$ by means of this method.\\
When $\beta$ is the golden mean $\frac{\sqrt{5}+1}{2}$, we have that $I=\{0,1\}$ and $d_\beta(1)=110^\omega$. Hence the set $\varphi(X)$ of admissible sequences is given by all sequences on the alphabet $I$ such that no two consecutive digits of the form $\epsilon_i\epsilon_{i+1}=11$ can exist, for every $i\geq 0$.\\
Now fix $n=4$. By writing down the sequence of admissible digits in the right-to-left order
\begin{eqnarray*}
& & 1\\
& & 01\\
& & 001\\
& & 101\\
& & 0001\\
& & 1001\\
& & 0101\\
& & 00001\\
& & \vdots\ 
\end{eqnarray*}
we can see that $G_\beta(3)\leq 4 < G_\beta(4)$. Then the first elements of $N_\beta$ with respect to $n=4$ are
\begin{eqnarray*}
& & N_\beta(1)=\beta^{-1}\\
& & N_\beta(2)=\beta^{-2}\\
& & N_\beta(3)=\beta^{-3}\\
& & N_\beta(4)=\beta^{-1}+\beta^{-3}\\
& & N_\beta(5)=\beta^{-4}\\
& & N_\beta(6)=\beta^{-1}+\beta^{-4}\\
& & N_\beta(7)=\beta^{-2}+\beta^{-4}
\end{eqnarray*}
Ninomiya \cite[Theorem 3.1]{Ninomiya} also proved that when $\beta$ is a Pisot number and all its conjugates belong to $\{z\in\mathbb{C} : |z|<1\}$, then the $\beta$-adic van der Corput sequence $N_\beta$ is low-discrepancy.\\
Steiner \cite{Steiner} gave an explicit formula for the discrepancy function of these sequences. In particular he showed that for Pisot numbers $\beta$ with irreducible polynomial, the discrepancy function $D(N,[0,y[)$ is bounded if and only if the $\beta$-expansion of $y$ is finite or its tail is the same as that of the expansion of 1. Moreover, if $\beta$ is a Parry number, then he showed that the discrepancy function is unbounded for all intervals of length $y\notin \mathbb{Q}(\beta)$. 
\end{example}

\begin{definition}
 Let $G=(G_n)_{n \geq 0}$ be an increasing sequence of positive integers with $G_0=1$. Then every positive integer can be expanded in the following way
\begin{equation*}
  n=\sum_{k=0}^{\infty}\varepsilon_k(n) G_k1 \ ,
\end{equation*}
where $\varepsilon_k(n) \in \{0, \ldots, \lceil G_{k+1}/G_k \rceil -1 \}$ and $\lceil x \rceil$ denotes the smallest integer not less than $x \in \mathbb{R}$. This expansion (called $G$-expansion) is uniquely determined and finite, provided that for every $K$
\begin{equation}\label{eq1}
\sum_{k=0}^{K-1}\varepsilon_k(n) G_k < G_K .
\end{equation}
For short we write $\varepsilon_k$ for the $k$-th digit of the $G$-expansion.
\end{definition}
$G = (G_n)_{n \geq 0}$ is called numeration system and the digits $\varepsilon_k$ can be computed by the following greedy algorithm (for details see for instance \cite{Fraenkel}).\\

Given a positive integer $n$, its greedy $G$-expansion is obtained in the following way: let $i$ be the largest index such that $G_i \leq n < G_{i+1}$. Set 
\begin{equation*}
\epsilon_i= \left\lfloor \frac{n}{G_i}\right\rfloor\ .
\end{equation*}
Then $n$ can be written as 
\begin{equation*}
n = \epsilon_i Gi + n_i\ ,
\end{equation*}
with $0 \leq n_i < G_i$. By iterating this procedure with $n_i$ the expansion we obtain that the greedy representation for $n$ is the string $\epsilon_i\epsilon_{i-1} \dots \epsilon_1\epsilon_0$ and we say $n$ is greedily representable.\\

It is easy to see that every non-negative integer is greedily representable if and only if $G_0 = 1$.
If $G_0 \neq 1$, it is possible for a number to be representable, but not greedily representable. For example, consider expressing $12$ in the numeration system $G$ where $(G_n)_{n \geq 0}$ is the sequence of prime numbers $(G_0,G_1,G_2,\dots) = (2, 3, 5, \dots)$. In this case we have that the largest index such that $G_i \leq 12 < G_{i+1}$ is 4, with corresponding $G_4=11$, but the procedure cannot go on. \\
Note that if $G_0 = 1$, then the greedy representation is in fact the lexicographically greatest representation. A desirable property of any numeration system is that the mapping that sends an integer $n$ to its representation is order-preserving. For more properties enjoyed by digits and numeration systems we refer to \cite{Allouche_Shallit, Shallit}.\\

We denote by $\mathcal{K}_G$ the subset of sequences that satisfy (\ref{eq1}), i.e. $\mathcal{K}_G$ is the set of sequences $\epsilon = \epsilon_0\epsilon_1\dots$ belonging to the infinite product
\begin{equation*}
\Pi(G)=\prod_{m=0}^\infty\{0,1,\dots ,\lfloor G_{k+1}/G_k \rfloor -1  \}\ ,
\end{equation*}
satisfying (\ref{eq1})\ . The elements of $\mathcal{K}_G$ are called $G$-admissible. \\
Although the $G$-scale is not fibred, one may consider its compactification, in the sense of the closure of the language in the product space $\Pi(G)$. The set of nonnegative integers $\mathbb{N}$ is embedded in $\mathcal{K}_G$ by the canonical injection $n\mapsto \epsilon_0(n)\epsilon_1(n)\dots \epsilon_L(n)0^\infty$.\\
Obviously $\mathcal{K}_G$ is compact and it will be called the $G$-compactification of $\mathbb{N}$.
In order to extend the addition-by-one map defined on $\mathbb{N}$ to $\mathcal{K}_G$ we introduce the set $\mathcal{K}_G^0 \subseteq \mathcal{K}_G$
\begin{equation}
 \mathcal{K}_G^0=\left\lbrace x\in \mathcal{K}_G\ : \exists M_x, \forall j\geq M_x, \quad \sum_{k=0}^{j}\varepsilon_k G_k\ < G_{j+1}-1\right\rbrace\ .
\end{equation}
Put $x(j)=\sum_{k=0}^{j}\varepsilon_k G_k$, and set
\begin{equation}\label{eq2}
 \tau(x)=(\varepsilon_0(x(j)+1)\dots \varepsilon_j(x(j)+1))\varepsilon_{j+1}(x)\varepsilon_{j+2}(x)\dots \ ,
\end{equation}
for every $x\in \mathcal{K}_G^0$ and $j\geq M_x$. This definition does not depend on the choice of $j\geq M_x$. In fact, let $l$ be the greatest integer such that $x(l-1)+1 = G_l$ provided that such an $l$ exists; otherwise there is no carry and we just add one to the first digit. Then for all $j \geq l$ we have
\begin{eqnarray*}
 (x(j)+1)&=& (\epsilon_0(x(l)+1)\dots \epsilon_l(x(l)+1))x_{l+1}\dots x_j\\
 &=& 0^l(x_l+1)x_{l+1}\dots x_j\ .
\end{eqnarray*}
So we can extended the definition of $\tau$ to sequences $x$ in $\mathcal{K}_G\setminus \mathcal{K}_G^0$ by $\tau(x)=0=(0^{\infty})$. In this way the transformation $\tau$ is defined on $\mathcal{K}_G$ and it is called $G$-odometer. We refer to \cite{glt} for a complete survey on odometers related to general numeration systems.\\
In particular, the authors in \cite{glt} focus their attention to a particular family of $G$-odometers, namely $G$-odometers where the base sequence is a
linear recurrence. In this case they can show that the map $\tau$ is continuous and $(\mathcal{K}_G, \tau)$ is uniquely ergodic.\\
In the sequel we summarise the important steps of this result, that will be necessary in the proofs of Chapter 4.\\

In the sequel we will consider sequences $(G_n)_{n\in\mathbb{N}}$ generated by a finite linear recurrence of order $d+1$, where $d+1$ is the period length:
\begin{equation}\label{rec}
 G_{n+d+1}=a_0G_{n+d} + a_{1}G_{n+d-1}+\dots + (a_d+1)G_n, \quad n\geq 0\ .
\end{equation}
and the initial values are given by $G_0=1$ and 
\begin{equation*}
G_{n+1}=\sum_{k=0}^n a_{n-k}G_k+1\ .
\end{equation*}

In \cite[Theorem 5]{glt}, the authors show that the odometer on an admissible
numeration system $G$ is uniquely ergodic, providing an explicit formula for the unique
invariant measure $\mu$.
\begin{theorem}
The odometer $\tau$ is a uniquely ergodic transformation, i.e.
there is a unique invariant measure $\mu$ given by
{ \small   
\begin{align}
 &\mu(Z) \label{mu} =\\ 
 &\frac{F_{K,0} \beta^{d-1} + (F_{K,1} - a_0 F_{K,0}) \beta^{d-2} + \ldots +
(F_{K,d-1} - a_0 F_{K,d-2} - \ldots - a_{d-2} F_{K,0})}{\beta^K (\beta^{d-1} +
\beta^{d-2} + \ldots + 1)},\notag
\end{align}}
where $F_{k,r} := \# \{ n < G_{k+r}: (\varepsilon_0(n), \varepsilon_1(n),
\ldots) \in Z \}$ and $Z$ is a cylinder with length $k$.
\end{theorem}
Note that the formula
in \cite[Theorem 5]{glt} included a misprint and was stated in corrected form in
\cite{Barat_Grabner}.\\

In \cite{glt} the author also prove that under a certain hypothesis the odometer has purely discrete spectrum. An analogous result was proved by Solo-\linebreak myak \cite{Solomyak} for linear recurrences with decreasing coefficients.
\begin{hypothesis}[Grabner, Tichy and Liardet \cite{glt}]\label{hypA}
There exists an integer $b > 0$ such that for all $k$ and 
\begin{equation*}
 N = \sum_{i = 0}^k \epsilon_i G_i + \sum_{j = k + b + 2}^\infty \epsilon_j G_j,
\end{equation*}
the addition of $G_m$ to $N$, where $m \geq k + b + 2$, does not change the
digits $\epsilon_0, \ldots, \epsilon_k,$ in the greedy representation i.e.\
\begin{equation*}
N + G_m = \sum_{i = 0}^k \epsilon_i G_i + \sum_{j = k + 1}^\infty \epsilon'_j
G_j.
\end{equation*}
\end{hypothesis}

\begin{theorem}\label{disc_spec}
$\mathcal{K}_G$ is (measure-theoretically) isomorphic to a group rotation with purely discrete spectrum given by the countable group
\begin{equation}\label{Gamma}
 \Gamma := \{z \in \mathbb{C}: \lim_{n \rightarrow \infty} z^{G_n} = 1 \}.
\end{equation}
provided that the above Hypothesis holds.
\end{theorem}
This analysis of the spectrum will be very useful in Chapter 4 when we will consider the product of systems of the form $(\mathcal{K}_{G_i}, \tau)$. We will prove that the product of these uniquely ergodic systems is uniquely ergodic, too.\\

There are only a few results concerning this hypothesis. In \cite{glt} the authors
remark that the Multinacci sequence, i.e.\ $a_0 = \ldots = a_{d-1} = 1$,
fulfills Hypothesis \ref{hypA}. Furthermore Bruin et al.\ \cite{bks} show that
the numerations systems with coefficients $(a_0, a_1, a_2) = (1,0,1)$ fulfills
Hypothesis 1.\\ 
Sometimes it seems easier to consider the following hypothesis introduced by Frougny and Solomyak \cite{frougny}.
\begin{hypothesis}[Frougny and Solomyak \cite{frougny}]\label{hypB}
 The characteristic root $\beta$ of the numeration system $G$ is a Pisot number
such that all elements of the set $\mathbb{Z}[\beta^{-1}]$ have finite
$\beta$-expansions.
\end{hypothesis}

\begin{remark}\label{solo}
Several researchers worked on algebraic characterizations of Pisot numbers
$\beta$ which satisfy Hypothesis \ref{hypB}. Solomyak \cite[Main Theorem]{Solomyak} showed that numeration systems defined by linear recurrences with decreasing coefficients fullfill Hypothesis \ref{hypB}, proving an analogous of Theorem \ref{disc_spec} for this case.
\end{remark}
Furthermore Hollander \cite{holl} states another sufficient condition for
Hypothesis \ref{hypB} and Akiyama \cite{aki} characterizes all Pisot units of
degree three satisfying Hypothesis \ref{hypB}.\\
However there exists no complete algebraic characterization for Pisot
numbers satisfying Hypothesis \ref{hypB} of degree greater than two. Note that
both hypotheses can be satisfied by the same numeration system but, to the best
of the author of this thesis knowledge, it is unknown if the two hypotheses are equivalent,
see \cite{glt}.\\

\chapter{$LS$ sequences}
 \thispagestyle{empty}
In the first chapter we mentioned that the concept of uniform distribution and discrepancy applies also to sequences of sets and partitions.\\
We presented the Kakutani splitting procedure and its generalization, the $\rho$-splitting, which produce uniformly distributed sequences of partitions.\\
In this chapter we want to introduce the constrction of $LS$-sequences of partitions. This is a countable family corresponding to pair of natural numbers $L$ and $S$ such that $S\geq 0$ and $L+S\geq 2$. For $S=0$ and $L\geq 2$ the procedure reduces to the well-known sequences of partitions into $b^n$ equal parts, with $b=L$. These sequences of partitions have been introduced by Carbone \cite{carbone}. She proved that these sequences of partitions are low-discrepancy if and only if $L\geq S$ and computed explicit asymptotic behaviour of the discrepancy of the sequences for $L<S$.\\
What is more interesting is that she presented an algorithm (improved in \cite{Carbone2}) which associates to any $LS$-sequence of partitions a sequence (called $LS$-sequence) of points and that for $L\geq S$ they are low-discrepancy. For $S=0$ and $L\geq 2$ the algorithm provides the well-known van der Corput sequence of base $b=L$.
\section{$LS$-sequences in the unit interval}
In this section we first give a description of the $LS$-sequences of partitions, obtained as subsequent $\rho$-refinements of the trivial partition $\omega$. Then we describe the $LS$-sequence of points associated to the sequences of partitions.
\begin{definition}
Let us fix two positive integers $L$ and $S$, with $L+S\geq 2$ and $S\geq 0$ and let $0 <\alpha < 1$ be the real number such that $L\alpha + S\alpha^2 = 1$. Denote by $\rho_{L,S}$ the partition defined by $L$ “long” intervals having length $\alpha$ and by $S$ “short” intervals having length $\alpha^2$. By $(\rho_{L,S}^n)_{n\in\mathbb{N}}$
we denote the sequence of successive $\rho_{L,S}$-refinements of the trivial partition $\omega$. This will be called the $LS$-sequence of partitions.
\end{definition}
\begin{example}
Let us consider $L=S=1$ and the corresponding $1,1$-sequence of partitions $(\rho_{1,1}^{n})_{n\in\mathbb{N}}$ whose characteristic polynomial is $\alpha+\alpha^{2}=1$. Its solution in $[0,1[$ is $\alpha=\frac{\sqrt{5}-1}{2}$, the inverse of the golden ratio. The first partition consists only of two intervals of length $\alpha$ and $\alpha^{2}$, respectively, i.e.
\begin{equation*}
\rho_{1,1}^{1}=\{[0,\alpha[,[\alpha,1[\}\ .
\end{equation*}
In order to construct the second partition, we divide the interval of maximal length proportionally to $\alpha$ and $\alpha^{2}$. Therefore
\begin{equation*}
\rho_{1,1}^{2}=\{[0,\alpha^{2}[,[\alpha^{2},\alpha[,[\alpha ,1[\}\ .
\end{equation*}
At this point we have two intervals of maximal length, the first and the third, and we proceed by splitting them. Then we get 
\begin{equation*}
\rho_{1,1}^{3}=\{[0,\alpha^{3}[, [\alpha^{3},\alpha^{2}[,[\alpha^{2},\alpha[,[\alpha , \alpha +\alpha^{3}[,[\alpha+\alpha^{3},1[\}\ ,
\end{equation*}
and so on.
\begin{figure}[h!]
\begin{center}
\begin{tikzpicture}[scale=12.3]
\draw (0, 0) -- (1,0);
\draw (0,-0.01) node[below, black]{0} -- (0,0.01);
\draw (1,-0.01) node[below, black]{1} -- (1,0.01);
\draw (0.61803,-0.01) node[below, black]{\begin{scriptsize} $\alpha$\end{scriptsize}}-- (0.61803,0.01);
\end{tikzpicture}
\end{center}
\end{figure}
\begin{figure}[h!]
\begin{center}
\begin{tikzpicture}[scale=12.3]
\draw (0, 0) -- (1,0);
\draw (0,-0.01) node[below, black]{0} -- (0,0.01);
\draw (1,-0.01) node[below, black]{1} -- (1,0.01);
\draw (0.61803,-0.01) node[below, black]{\begin{scriptsize} $\alpha$\end{scriptsize}}-- (0.61803,0.01) ;
\draw (0.38196,-0.01) node[below, black]{\begin{scriptsize}$\alpha^{2}$\end{scriptsize}}-- (0.38196,0.01);
\end{tikzpicture}
\end{center}
\end{figure}

\begin{figure}[h!]
\begin{center}
\begin{tikzpicture}[scale=12.3]
\draw (0, 0) -- (1,0);
\draw (0,-0.01) node[below, black]{0} -- (0,0.01);
\draw (1,-0.01) node[below, black]{1} -- (1,0.01);
\draw (0.61803,-0.01) node[below, black]{\begin{scriptsize} $\alpha$\end{scriptsize}}-- (0.61803,0.01);
\draw (0.38196,-0.01) node[below, black]{\begin{scriptsize}$\alpha^{2}$\end{scriptsize}}-- (0.38196,0.01);
\draw (0.23606,-0.01) node[below, black]{\begin{scriptsize}$\alpha^{3}$\end{scriptsize}}-- (0.23606,0.01);
\draw (0.85410,-0.01) node[below, black]{\begin{scriptsize}$\alpha +\alpha^{3}$\end{scriptsize}}-- (0.85410,0.01);
\end{tikzpicture}
\end{center}
\end{figure}
\end{example}
\begin{example}
Let us consider the sequence of partitions $(\rho_{2,1}^{n})_{n\in\mathbb{N}}$ corresponding to $2\alpha+\alpha^{2}=1$. At the first step we have a subdivision of the unit interval into two long intervals and a short one, and we get the following partition
\begin{equation*}
\rho_{2,1}^{1}=\{[0,\alpha[,[\alpha, 2\alpha[,[2\alpha,1[\}\ .
\end{equation*}
In order to obtain the second partition we split the first two intervals, as follows
\begin{eqnarray*}
\rho_{2,1}^{2}&=& \{[0,\alpha^{2}[,[\alpha^{2},2\alpha^{2}[,[2\alpha^{2},\alpha[,[\alpha,\alpha+\alpha^{2}[, \\
& & [\alpha+\alpha^{2}, \alpha+2\alpha^{2}[,[\alpha+2\alpha^{2},2\alpha[,[2\alpha,1[\}\ ,
\end{eqnarray*}
and so on.
\begin{figure}[h!]
\begin{center}
\begin{tikzpicture}[scale=12.3]
\draw (0, 0) -- (1,0);
\draw (0,-0.01) node[below, black]{0} -- (0,0.01) ;
\draw (1,-0.01) node[below, black]{1} -- (1,0.01);
\draw (0.414214,-0.01) node[below, black]{\begin{scriptsize} $\alpha$\end{scriptsize}}-- (0.414214,0.01);
\draw (0.828428,-0.01) node[below, black]{\begin{scriptsize}$2\alpha$\end{scriptsize}}-- (0.828428,0.01);
\end{tikzpicture}
\end{center}
\end{figure}

\begin{figure}[h!]
\begin{center}
\begin{tikzpicture}[scale=12.3]
\draw (0, 0) -- (1,0);
\draw (0,-0.01) node[below, black]{0} -- (0,0.01);
\draw (1,-0.01) node[below, black]{1} -- (1,0.01);
\draw (0.414214,-0.01) node[below, black]{\begin{scriptsize} $\alpha$\end{scriptsize}}-- (0.414214,0.01) ;
\draw (0.828428,-0.01) node[below, black]{\begin{scriptsize}$2\alpha$\end{scriptsize}}-- (0.828428,0.01);
\draw (0.171573,-0.01) node[below, black]{\begin{scriptsize}$\alpha^{2}$\end{scriptsize}}-- (0.171573,0.01);
\draw (0.343146,-0.01) node[below, black]{\begin{scriptsize}$2\alpha^2 $\end{scriptsize}}-- (0.343146,0.01) ;
\draw (0.585787,-0.01) node[below, black]{\begin{scriptsize}$\alpha+\alpha^{2}$\end{scriptsize}}-- (0.585787,0.01);
\draw (0.75736,-0.01) node[below, black]{\begin{scriptsize}$\alpha +2\alpha^{2}$\end{scriptsize}}-- (0.75736,0.01);
\end{tikzpicture}
\end{center}
\end{figure}
\end{example}
Let us note that every partition $\rho_{L,S}^{n}$ produces only intervals of length $\alpha^{n}$ (long) and of length $\alpha^{n+1}$ (short), respectively. If we denote by $t_n$ the total number of intervals of $\rho_{L,S}^{n}$, then the following linear recurrence holds 
\begin{equation}\label{LSrecurr}
t_n=Lt_{n-1}+St_{n-2} \ , 
\end{equation}
with initial conditions $t_0=1$ and $t_1=L+S$. \\ In fact, if we denote by $l_n$ the total number of long intervals of $\rho_{L,S}^{n}$ and by $s_n$ the total number of short intervals, then the following relations hold:
\begin{equation*}
t_n=l_n+s_n\ , \qquad l_n=Ll_{n-1}+s_{n-1}\ , \qquad s_n=Sl_{n-1} \ .
\end{equation*}
This is beacause the $n$-th partition is obtained splitting the long intervals of the $(n-1)$-th partition into $L$ long and $S$ short intervals. The short intervals of the $(n-1)$-th partition remain untouched and hence become long intervals of the $n$-th partition. Formally:
\begin{eqnarray*}
t_n&=&l_n+s_n=Ll_{n-1}+s_{n-1}+Sl_{n-1}= \\
&=& Ll_{n-1}+Sl_{n-2}+SLl_{n-2}+Ss_{n-2}=\\
&=&L(l_{n-1}+s_{n-1})+S(l_{n-2}+s_{n-2})=\\
&=&Lt_{n-1}+St_{n-2}.
\end{eqnarray*}

From \eqref{LSrecurr} immediately follows that the sequence $(t_n)_{n\in\mathbb{N}}$, in the case $L=S=1$, is the Fibonacci sequence. On the other hand we have seen that in this case $(\rho_{1,1}^{n})_{n\in\mathbb{N}}$  is a Kakutani sequence. For these reasons, the corresponding $1,1$-sequence of partitions has been called in \cite{carbone} the Kakutani-Fibonacci sequence of partitions.\\

The following theorem shows that the $LS$-sequences of partitions are low-discrepancy if $L\geq S$.
\begin{theorem}[Theorem 2.2, \cite{carbone}]\begin{itemize}
\item[i)] If $S \leq L$ there exist $c_1, c'_1 > 0$ such that 
\begin{equation*}
c'_1 \leq t_n D(\rho_{L,S}^n) \leq c_1 \quad {\rm for\ any\ }n\in\mathbb{N}\ .
\end{equation*}
\item[ii)] If $S = L+1$ there exist $c_2, c'_2 > 0$ such that 
\begin{equation*}
c'_2\log t_n \leq t_nD(\rho_{L,S}^n) \leq c_2\log t_n \quad {\rm for\ any\ } n\in\mathbb{N}.
\end{equation*}
\item[iii)] If $S \geq L+2$ there exist $c_3, c'_3 > 0$ such that
\begin{equation*}
c'_3t_n^{1-\tau} \leq t_nD(\rho_{L,S}^n) \leq c_3t_n^{1-\tau} \quad {\rm for\ any\ } n\in\mathbb{N}\ ,
\end{equation*}
where $\tau =1+\frac{\log(S\alpha)}{\log \alpha} >0$.
\end{itemize}
\end{theorem}
Observe that the same estimates for the discrepancy of the $LS$-sequences of partitions have been reobtained by Drmota and Infusino \cite{Drmota_Infusino}, by applying Theorems \ref{DM1} and \ref{DM2} to this particular class of $\rho$-refinements.\\ 
Let us now explain how to get the associated sequences of points of the corresponding $LS$-sequence of partitions, also introduced in \cite{carbone}.
\begin{definition}
Given the sequence of partition $\left(\rho_{L,S}^n\right)_{n\in\mathbb{N}}$, we define the $LS$-sequence of points $(\xi_{L,S}^n)_{n\in\mathbb{N}}$
as follows. \\The first $t_1$ points are just the left endpoints of the intervals of $\rho_{L,S}^1$ ordered by magnitude. This set of points will be denoted by $\Lambda_{L,S}^1$ and its points will be denoted by $\xi_1^{(1)},\dots, \xi_{t_1}^{(1)}$. For $n > 1$ and if $\Lambda_{L,S}^n=\left(\xi_1^{(n)},\dots,\xi_{t_n}^{(n)}\right)$ is the set of the $t_n$ points (written in their order)
of $\rho_{L,S}^n$, then the $t_{n+1}$ points of $\rho_{L,S}^{n+1}$
are recursively ordered as follows:

\begin{equation*}
\begin{split}
& \Lambda_{L,S}^{n+1} = \Big(\xi_1^{(n)},\ldots, \xi_{t_n}^{(n)},\\
& \varphi_1^{(n+1)}(\xi_1^{(n)}),\ldots , \varphi_1^{(n+1)}(\xi_{l_n}^{(n)}),\ldots , \varphi_L^{(n+1)}(\xi_1^{(n)}), \ldots , \varphi_L^{(n+1)}(\xi_{l_n}^{(n)}), \\
& \varphi_{L,1}^{(n+1)}(\xi_1^{(n)}), \ldots, \varphi_{L,1}^{(n+1)}(\xi_{l_n}^{(n)}),\ldots , \varphi_{L,S-1}^{(n+1)}(\xi_1^{(n)}), \ldots , \varphi_{L,S-1}^{(n+1)}(\xi_{l_n}^{(n)})\Big)\ .
\end{split}
\end{equation*}
Here, $l_n$ is the number of long intervals of $\rho_{L,S}^n$, and the two families of functions are
\begin{equation*}
\varphi_i^{(n+1)}(x) = x +i\alpha^{n+1}\quad and \quad \varphi_{L,j}^{(n+1)}(x) = x + L\alpha^{n+1} + j\beta^{n+2}
\end{equation*}
for $1 \leq i \leq L$ and $1 \leq j \leq S -1$\ .
\end{definition}
In \cite{carbone} it has been shown that, if $L\geq S$ $D(\xi_{L,S}^{1}, \ldots, \xi_{L,S}^{N})=\mathcal{O}\left(\frac{\log N}{N}\right)$ for $n\to\infty$ and that to sequences of partitions $(\rho_{L,S}^{n})_{n\in\mathbb{N}}$ with low-discrepancy correspond sequences of points $(\xi_{L,S}^{n})_{n\in\mathbb{N}}$ with low-discrepancy.
\begin{theorem}[Theorem 3.9, \cite{carbone}]
i) If $S\leq L$ there exist $k_1, k_1'>0$ such that for every $N\in\mathbb{N}$
\begin{equation*}
k_1'\frac{1}{N}\leq D_N(\xi_{L,S}^{1}, \xi_{L,S}^{2},\ldots, \xi_{L,S}^{N})\leq k_1\frac{\log N}{N}\ .
\end{equation*}
ii) If $S=L+1$ there exist $k_2, k_2'>0$ such that for every $N\in\mathbb{N}$
\begin{equation*}
k_2'\frac{\log N}{N}\leq D_N(\xi_{L,S}^{1}, \xi_{L,S}^{2},\ldots, \xi_{L,S}^{N})\leq k_2\frac{\log^{2} N}{N}\ .
\end{equation*}
iii) If $S>L+1$ there exist $k_3,k_3'>0$ such that for every $N\in\mathbb{N}$
\begin{equation*}
k_3'\frac{1}{N^{\tau -1}}\leq D_N(\xi_{L,S}^{1}, \xi_{L,S}^{2},\ldots, \xi_{L,S}^{N})\leq k_3\frac{\log N}{N^{\tau -1}}\ ,
\end{equation*}
where $\tau=1+\frac{\log S\alpha}{\log\alpha}< 1$.
\end{theorem}
As an example we show how to get the sequence of points associated to the sequences of partitions considered in the above examples.
\begin{example}\label{example8}
We want to consider the Kakutani-Fibonacci sequence $(\xi_{1,1}^{n})_{n\in\mathbb{N}}$. The set of points corresponding to the sequence of partitions $\rho_{1,1}^{1}$ contains exactly the left endpoints of the two intervals $[0,\alpha[, [\alpha,1[$, i.e.
\begin{equation*}
\Lambda_{1,1}^{1} = \Big(0,\alpha\Big)=\Big(\xi_{1,1}^{1}, \xi_{1,1}^{2}\Big)\ .
\end{equation*}
We can observe that since we have only one short interval, then $j=0$ and therefore there are no functions of the type $\varphi_{L,j}^{n+1}$, but only the function $\varphi_1^{(n+1)}$. Since $l_1=1$, applying the function $\varphi_1^{(2)}$ to the point zero we have
\begin{equation*}
\Lambda_{1,1}^{2} = \Big(0,\alpha,\varphi_1^{(2)}(0)\Big)=\Big(0,\alpha,\alpha^{2}\Big)=\Big(\xi_{1,1}^{1}, \xi_{1,1}^{2}, \xi_{1,1}^{3}\Big)\ .
\end{equation*}and also
\begin{eqnarray*}
\Lambda_{1,1}^{3} &=& \Big(0,\alpha, \alpha^{2},\varphi_1^{(3)}(0),\varphi_1^{(3)}(\alpha)\Big)\\
&=&\Big(0,\alpha,\alpha^{2},\alpha^{3},\alpha+\alpha^{3}\Big)=\Big(\xi_{1,1}^{1}, \xi_{1,1}^{2}, \xi_{1,1}^{3}, \xi_{1,1}^{4}, \xi_{1,1}^{5}\Big)\ .
\end{eqnarray*}
Going on in this way and applying the function $\varphi_1^{(4)}$ to the first three points (since $l_3=3$), we get $\Lambda_{1,1}^{4}$ as represented in the following figure.\\
\begin{figure}[h!]

\begin{center}
\begin{tikzpicture}[scale=12.3]
\draw (0, 0) -- (1,0);
\draw (0,-0.01) node[below, black]{0} -- (0,0.01) node[above, black]{$\xi_1$};
\draw (1,-0.01) node[below, black]{1} -- (1,0.01);
\draw (0.61803,-0.01) node[below, black]{\begin{scriptsize} $\alpha$\end{scriptsize}}-- (0.61803,0.01) node[above, black]{$\xi_2$};
\draw (0.38196,-0.01) node[below, black]{\begin{scriptsize}$\alpha^{2}$\end{scriptsize}}-- (0.38196,0.01) node[above, black]{$\xi_3$};
\draw (0.23606,-0.01) node[below, black]{\begin{scriptsize}$\alpha^{3}$\end{scriptsize}}-- (0.23606,0.01) node[above, black]{$\xi_4$};
\draw (0.85410,-0.01) node[below, black]{\begin{scriptsize}$\alpha +\alpha^{3}$\end{scriptsize}}-- (0.85410,0.01) node[above, black]{$\xi_5$};
\draw (0.14589,-0.01) node[below, black]{\begin{scriptsize}$\alpha^{4}$\end{scriptsize}}-- (0.14589,0.01) node[above, black]{$\xi_6$};
\draw (0.76393,-0.01) node[below, black]{\begin{scriptsize}$\alpha +\alpha^{4}$\end{scriptsize}}-- (0.76393,0.01) node[above, black]{$\xi_7$};
\draw (0.52786,-0.01) node[below, black]{\begin{scriptsize}$\alpha^2 +\alpha^{4}$\end{scriptsize}}-- (0.52786,0.01) node[above, black]{$\xi_8$};
\end{tikzpicture}
\end{center}
\caption{First 8 points of $(\xi_{1,1}^n)_{n \ge 1}$}
\end{figure}
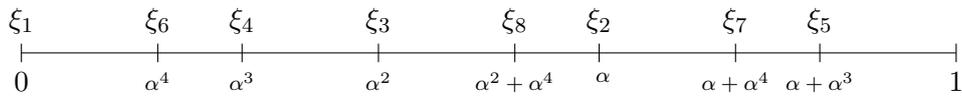
\\ By iterating this procedure we get the sequence of points $(\xi_{1,1}^{n})_{n\in\mathbb{N}}$ associated to $(\rho_{1,1}^{n})_{n\in\mathbb{N}}$. \\ This sequence has been called Kakutani-Fibonacci in \cite{carbone} and \cite{CIV} and it will be of particular interest in the next chapter, where we will consider a method to get this sequence as orbit of an ergodic transformation.
\end{example}
\begin{example}
Let us now consider the sequence of points $(\xi_{2,1}^{n})_{n\in\mathbb{N}}$. Also in this case we do not have functions of the type $\varphi_{L,j}^{(n+1)}$, but only of type $\varphi_i^{(n+1)}$, with $i=1, 2$, to apply to the first $l_n$ points of the set $\Lambda_{2,1}^{n}$. The first set of points is
\begin{equation*}
\Lambda_{2,1}^{1} = \Big(0,\alpha,2\alpha\Big)=\Big(\xi_{2,1}^{1}, \xi_{2,1}^{2}, \xi_{2,1}^{3}\Big)\ .
\end{equation*}
Then, since $l_1=2$, if we apply the two functions $\varphi_1^{(2)}$ and $\varphi_2^{(2)}$ to the first two points of $\Lambda_{2,1}^{1}$ we have
\begin{eqnarray*}
\Lambda_{2,1}^{2} &=& \Big(0,\alpha,\varphi_1^{(2)}(0),\varphi_1^{(2)}(\alpha),\varphi_2^{(2)}(0),\varphi_2^{(2)}(\alpha)\Big)\\
&=&\Big(0,\alpha,2\alpha,\alpha^{2},\alpha+\alpha^{2},2\alpha^{2},\alpha+2\alpha^{2}\Big)\\
&=&\Big(\xi_{2,1}^{1}, \xi_{2,1}^{2}, \xi_{2,1}^{3}, \xi_{2,1}^{4}, \xi_{2,1}^{5}, \xi_{2,1}^{6}, \xi_{2,1}^{7}\Big)\ .
\end{eqnarray*}
Let us now compute the set of points of $\rho_{2,1}^{3}$. \\
Applying the functions $\varphi_1^{(3)}$ and $\varphi_2^{(3)}$ to the first five points of $\Lambda_{2,1}^{2}$ (since $l_2=5$), we have
\begin{eqnarray*}
\Lambda_{2,1}^{3} &=&\Big(0,\alpha,2\alpha,\alpha^{2},\alpha+\alpha^{2},2\alpha^{2},\alpha+2\alpha^{2},\\
& & \varphi_1^{(3)}(0),\varphi_1^{(3)}(\alpha),\varphi_1^{(3)}(2\alpha),\varphi_1^{(3)}(\alpha^{2}),\varphi_1^{(3)}(\alpha+\alpha^{2}),\\
& & \varphi_2^{(3)}(0),\varphi_2^{(3)}(\alpha),\varphi_2^{(3)}(2\alpha),\varphi_2^{(3)}(\alpha^{2}),\varphi_2^{(3)}(\alpha+\alpha^{2})\Big)\\
&=& \Big(0,\alpha, 2\alpha, \alpha^{2},\alpha+\alpha^{2}, 2\alpha^{2},\alpha+2\alpha^{2}, \\
& & \alpha^{3},\alpha +\alpha^{3}, 2\alpha+\alpha^{3}, \alpha^{2}+\alpha^{3}, \alpha+\alpha^{2}+\alpha^{3}, \\
& & 2\alpha^{3},\alpha+2\alpha^{3},2\alpha+2\alpha^{3},\alpha^{2}+2\alpha^{3},\alpha+\alpha^{2}+2\alpha^{3}\Big)\\
&=&\Big(\xi_{2,1}^{1}, \xi_{2,1}^{2}, \xi_{2,1}^{3}, \xi_{2,1}^{4}, \xi_{2,1}^{5}, \xi_{2,1}^{6}, \xi_{2,1}^{7}, \xi_{2,1}^{8},\\
&=& \xi_{2,1}^{9}, \xi_{2,1}^{10}, \xi_{2,1}^{11}, \xi_{2,1}^{12}, \xi_{2,1}^{13}, \xi_{2,1}^{14}, \xi_{2,1}^{15}, \xi_{2,1}^{16},\xi_{2,1}^{17}\Big) \ .
\end{eqnarray*}
The first points of this point set are shown in the following picture
\begin{figure}[h!]

\begin{center}
\begin{tikzpicture}[scale=12.3]
\draw (0, 0) -- (1,0);
\draw (0,-0.01) node[below, black]{0} -- (0,0.01) node[above, black]{$\xi_1$};
\draw (1,-0.01) node[below, black]{1} -- (1,0.01);
\draw (0.414214,-0.01) node[below, black]{\begin{scriptsize} $\alpha$\end{scriptsize}}-- (0.414214,0.01) node[above, black]{$\xi_2$};
\draw (0.828428,-0.01) node[below, black]{\begin{scriptsize}$2\alpha$\end{scriptsize}}-- (0.828428,0.01) node[above, black]{$\xi_3$};
\draw (0.171573,-0.01) node[below, black]{\begin{scriptsize}$\alpha^{2}$\end{scriptsize}}-- (0.171573,0.01) node[above, black]{$\xi_4$};
\draw (0.585787,-0.01) node[below, black]{\begin{scriptsize}$\alpha +\alpha^{2}$\end{scriptsize}}-- (0.585787,0.01) node[above, black]{$\xi_5$};
\draw (0.343146,-0.01) node[below, black]{\begin{scriptsize}$2\alpha^{2}$\end{scriptsize}}-- (0.343146,0.01) node[above, black]{$\xi_6$};
\draw (0.75736,-0.01) node[below, black]{\begin{scriptsize}$\alpha +2\alpha^{2}$\end{scriptsize}}-- (0.75736,0.01) node[above, black]{$\xi_7$};
\draw (0.071068,-0.01) node[below, black]{\begin{scriptsize}$\alpha^3$\end{scriptsize}}-- (0.071068,0.01) node[above, black]{$\xi_8$};
\draw (0.485282,-0.01) node[below, black]{\begin{scriptsize}$\alpha +\alpha^3$\end{scriptsize}}-- (0.485282,0.01) node[above, black]{$\xi_9$};
\end{tikzpicture}
\end{center}
\caption{First 9 points of $(\xi_{2,1}^n)_{n \ge 1}$}
\end{figure}
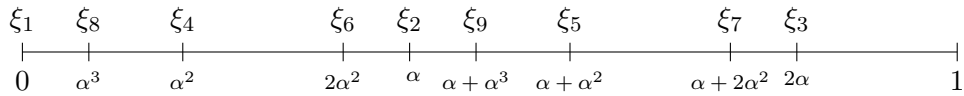
\end{example}
More efficient ways of constructing the $LS$-sequences of points have been recently introduced in \cite{AHZ} and \cite{Carbone2}.
\section{$LS$-sequence in higher dimension}
We now consider two possible bidimensional extensions of the $LS$-sequences of points, namely the $LS$-point set \`a la van der Corput and the $LS$-sequences \`a la Halton. These extensions have been proposed to me by I. Carbone and A. Vol\v{c}i\v{c} as the main object of my master thesis \cite{iaco} and are contained in the unpublished paper \cite{civ2}.
\begin{definition}
An $LS$-point set \`a la van der Corput of order $N\in\mathbb{N}$ in $[0,1[^2$ is defined as
\begin{equation}
\textbf{x}_n=\left(\frac{n}{N}, \xi_{L,S}^{n}\right)\qquad n=1,\ldots, N\ , 
\end{equation}
where $(\xi_{L,S}^{n})_{n\in\mathbb{N}}$ is an $LS$-sequence of points.
\end{definition}
\begin{remark}
By Lemma \ref{discr_point_set}, it follows immediately that if we choose a low-discrepancy $LS$-sequence $(\xi_{L,S}^{n})$, then the corresponding $LS$-point set \`a la van der Corput is low-discrepancy in $[0,1[^2$.
\end{remark}
\begin{definition}
An $LS$-sequence of points \`a la Halton  in $[0,1[^2$ is the sequence $((\xi_{L_1S_1}^{n},\xi_{L_2S_2}^{n}))_{n\in\mathbb{N}}$, where $(\xi_{L_1S_1}^{n})_{n\in\mathbb{N}}$ and $(\xi_{L_2S_2}^{n})_{n\in\mathbb{N}}$ are two distinct $LS$-sequences of points.
\end{definition}
Both definitions can be extended to an $s$-dimensional sequence in two possible ways. The first one is to consider $s$-tuples of $LS$-sequences, while the second one is to consider $s$-tuples where the first coordinate is $\left( \frac{n}{N} \right)$, and the remaining $s-1$ are $LS$-sequences.\\
Of course, like for classical sequences already encountered in the first chapter of the thesis, such as the multidimensional Kronecker sequence and the Halton sequences, the main problem is to provide a condition (often algebraic) to assure the uniform distribution property in the multidimensional setting.\\ In fact, even pairing low-discrepancy $LS$-sequences, in general, does not guarantee uniform distribution in the unit square. The following two figures show how the behaviour of the bidimensional sequence can indeed be very erratic.
\begin{figure}[h!]
\begin{minipage}{.45\textwidth}
\centering
\includegraphics[scale=0.55, keepaspectratio]{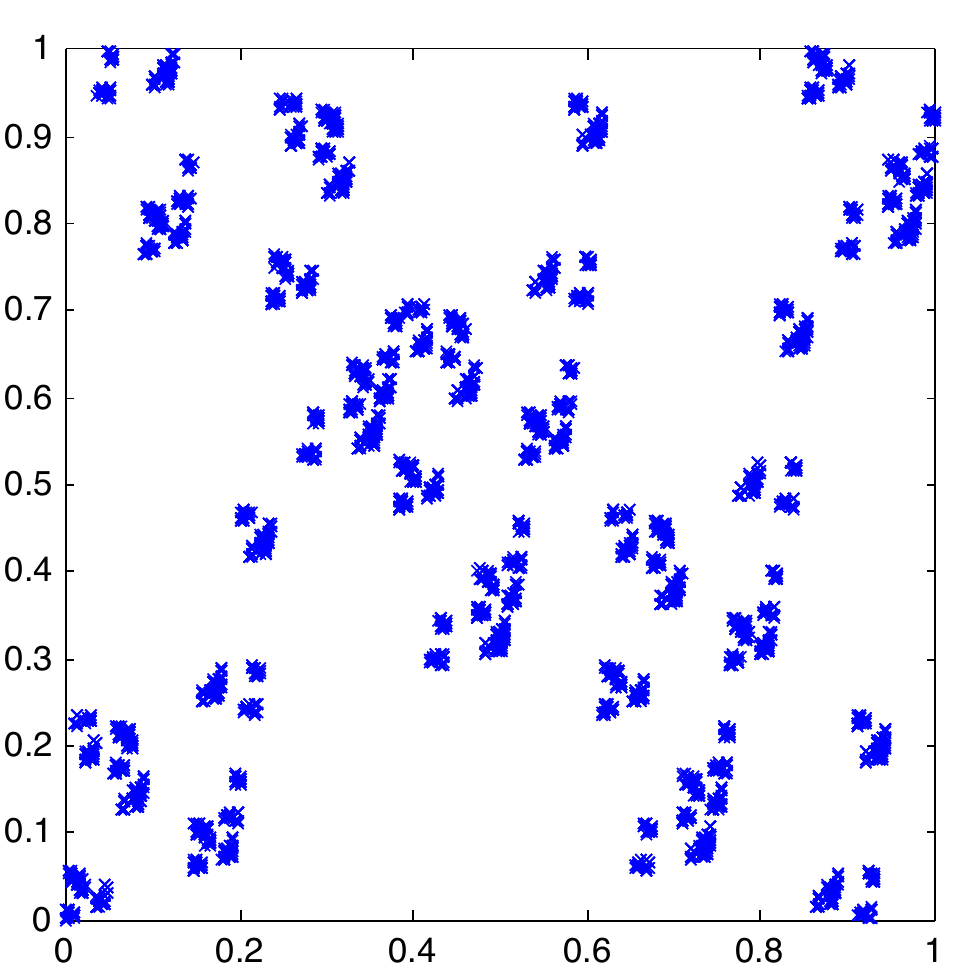}
\caption{$((\xi_{1,1}^{n},\xi_{5,1}^{n}))_{n\in\mathbb{N}}$ for $1\leq n\leq 5000$}\label{squirrel}
\end{minipage}
\hspace{7mm}
\begin{minipage}{.45\textwidth}
\centering
\includegraphics[scale=0.55, keepaspectratio]{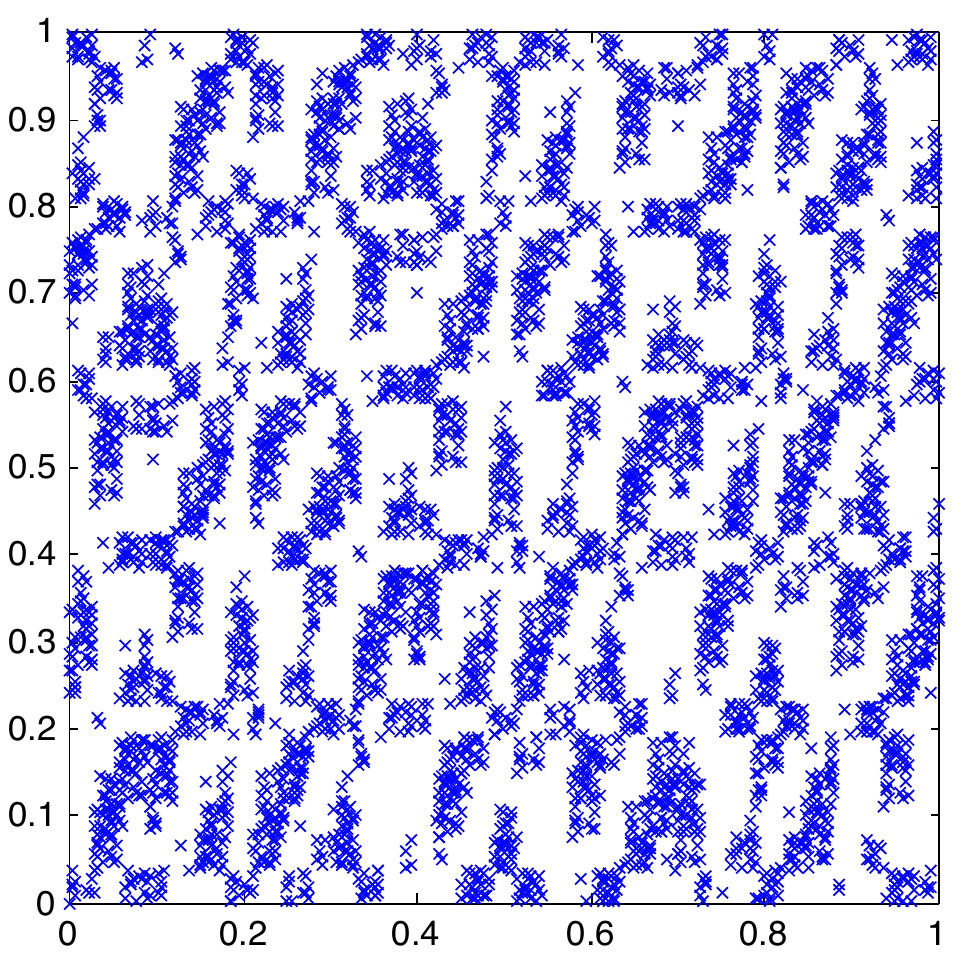}
\caption{$((\xi_{3,1}^{n},\xi_{4,1}^{n}))_{n\in\mathbb{N}}$ for $1\leq n\leq 5000$}
\end{minipage}
\end{figure}

A partial answer to this problem has been given by Aistleitner et al. in \cite{AHZ}, where the authors show that if the roots $\alpha_1$, $\alpha_2$ of the polynomials $S_1x^2+L_1x-1=0$ and $S_2x^2+L_2x-1=0$ do not fulfill a precise condition, then the sequence $((\xi_{L_1S_1}^{n},\xi_{L_2S_2}^{n}))_{n\in\mathbb{N}}$ is not dense in $[0,1[^2$. \\
In the sequel we will denote by $\mathcal B=((L_1,S_1),\dots,(L_s,S_s))$ an $s$-tuple of pairs, each of them defining an $LS$-sequence of points and by  $(\xi_{\bf{\mathcal{B}}}^n)_{n\in\mathbb{N}}$ the corresponding multidimensional $LS$-sequence, obtained by combining them componentwise. \\
More precisely, they proved the following
\begin{theorem}
Let $\mathcal B=((L_1,S_1),(L_2,S_2))$ with $L_i> S_i-1\geq 0$ and assume that there exist integers $m$ and $k$ such that $\frac{\alpha_1^{k+1}}{\alpha_2^{m+1}}\in\mathbb{Q}$. Then the two-dimensional $LS$-sequence $(\xi_{\bf\mathcal{B}}^n)_{n\in\mathbb{N}}$ is not uniformly distributed, and not even dense in $[0,1[^2$.
\end{theorem}
For instance, in the case $(L_1,S_1)=(1,1)$ and $(L_2,S_2)=(4,1)$ (see Figure \ref{squirrel}) we have $\alpha_2=\alpha_1^3$ and so the sequence is not even dense in $[0,1[^2$.\\

This theorem can also be applied to the multidimensional case, since for any multidimensional sequence of points, which is uniformly distributed, all lower-dimensional projections also have to be uniformly distributed. More precisely, one has the following
\begin{corollary} \label{co1}
Let $\mathcal B=((L_1,S_1),\dots,(L_s,S_s))$ with $L_i> S_i-1\geq 0$ and assume that there exist numbers $u,w \in \{1, \dots, s\}$ and integers $m$ and $k$ such that $\frac{\alpha_u^{k+1}}{\alpha_w^{m+1}}\in\mathbb{Q}$. Then the $s$-dimensional $LS$-sequence $(\xi_{\bf{\mathcal{B}}}^n)_{n\in\mathbb{N}}$ is not uniformly distributed, and not even dense in $[0,1]^s$.
\end{corollary}
However, no positive results, proving uniform distribution of a bidimensional $LS$-sequence for an appropriate choice of $L_1,S_1$ and $L_2,S_2$, are known so far. 
\chapter{A dynamical system approach to the Kakutani-Fibonacci sequence}
 \thispagestyle{empty}
\addtocounter{section}{1}
\setcounter{definition}{0}
In this chapter we show how to obtain the Kakutani-Fibonacci sequence of points, introduced in Chapter 2, as orbit of an ergodic transformation.
More precisely, we consider the Kakutani-Fibonacci sequence of partitions and associate to it an ergodic interval exchange, which will be called the Kakutani-Fibonacci transformation. The construction of the interval exchange makes use of the cutting-stacking technique, that we already considered in Chapter 1. Moreover, we will prove that the orbit of the origin under this map coincides with the low discrepancy Kakutani-Fibonacci sequence of points, defined in Example \ref{example8}.\\
The content of this chapter is presented in \cite{CIV}.

\begin{definition}[Cutting-stacking technique for the Kakutani-Fibonacci sequence]\label{8}
For every fixed $n$,  the intervals of the $n$-th Kakutani-Fibonacci partition $\alpha^{n}\omega$ are represented by two columns $C_n=\{L_n,S_n \}$ constructed as follows.

Let us start with two columns $C_1=\{L_1,S_1\}$, where $L_1=[0,\alpha[$, $S_1=[\alpha,1[$. 
 If we divide $L_1$ proportionally to $\alpha$ and $\alpha^2$, we can write $L_1=\{L_1^0, L_1^1\}$, where $L_1^0=[0,\alpha^2[$ and $L_1^1=[\alpha^2,\alpha[$.
Now we stack the interval $S_1$ over $L_1^0$ (and use the common notation $L_1^0 *S_1$) as they have the same  {\it width}, namely $w(L_1^0)=w(S_1)=\alpha^2$. So we get two new columns $C_2=\{L_2,S_2\}$, 
 where $L_2=L_1^0 *S_1$ and $S_2=L_1^1$. We denote by $b(L_2)=[0, \alpha^2[$ the {\it bottom} of $L_2$ and by $h(L_2)=2=l_2$ its {\it height}. \\
 If we continue this way, at the $n$-th step we get two columns denoted by $C_n=\{L_n,S_n\},$
with $h(L_n)=l_n$ and $h(S_n)=s_n$.  We divide $L_n$ into two columns, say $L_n^0$ and $L_n^1$, where $w(L_n^0)=\alpha^{n+1}$ and $w(L_n^1)=\alpha^{n+2}$, and stack $S_n$ over $L_n^0$, obtaining 
$$C_{n+1}=\{L_{n+1},S_{n+1}\},$$ where
\begin{equation*}L_{n+1}=L_n^0*S_n \qquad {\rm and}\qquad S_{n+1}=L_n^1\ ,\end{equation*} with
$$w(L_{n+1})=\alpha^{n+1}\ ,\quad h(L_{n+1})=l_{n+1}\ , \quad b(L_{n+1})=[0,\alpha^{n+1}[$$ and
$$w(S_{n+1})=\alpha^{n+2}\ , \quad h(S_{n+1})=s_{n+1}\ ,\quad b(C_{n+1})=[\alpha^{n+1},\alpha^n[\ .$$
\end{definition}
The first steps of the procedure are visualized in Figure \ref{cut_stack}.

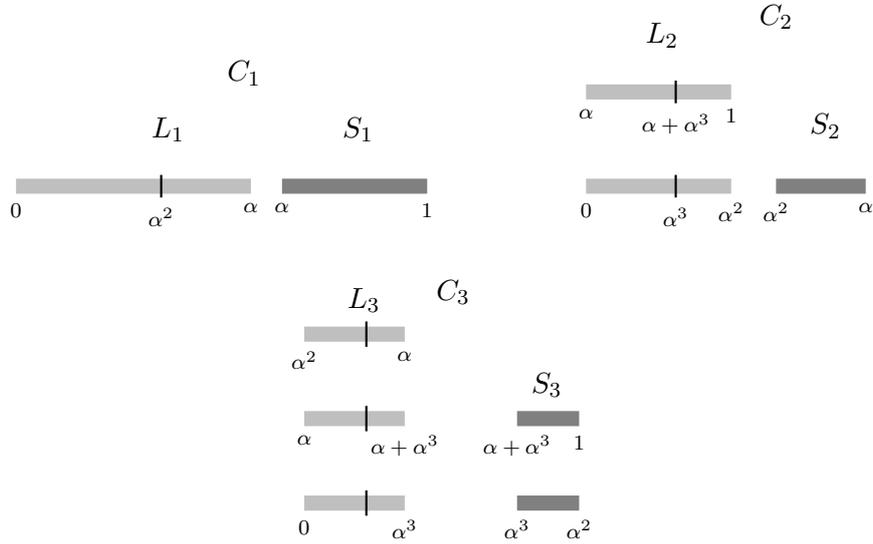
\begin{figure}[h!]
\begin{center}
\begin{tikzpicture}[scale=5]
\draw [lightgray, line width=0.2cm] (0,0) node[below, black]{\scriptsize $0$} -- (0.61803,0) node[below, black]{\scriptsize $\alpha$};
\node at (0.6,0.3) {$C_1$};
\node at (0.4,0.15) {$L_1$};
\node at (0.9,0.15) {$S_1$};
\draw [gray, line width=0.2cm] (0.7,0) node[below, black]{\scriptsize $\alpha$} -- (1.08196,0) node[below, black]{\scriptsize $1$};
\draw [thick] (0.38196,-0.03) node[below, black]{\scriptsize $\alpha^{2}$} -- (0.38196, 0.03);

\draw [lightgray, line width=0.2cm] (1.5,0) node[below, black]{\scriptsize $0$} -- (1.88196,0) node[below, black]{\scriptsize $\alpha^{2}$};
\draw [lightgray, line width=0.2cm] (1.5,0.25) node[below, black]{\scriptsize $\alpha$} -- (1.88196,0.25) node[below, black]{\scriptsize $1$};
\draw [gray, line width=0.2cm] (2,0) node[below, black]{\scriptsize $\alpha^{2}$} -- (2.23606,0) node[below, black]{\scriptsize $\alpha$};
\draw [thick] (1.73606,-0.03) node[below, black]{\scriptsize $\alpha^{3}$} -- (1.73606, 0.03);

\draw [thick,black] (1.73606, 0.22) node[below, black]{\scriptsize $\alpha+\alpha^{3}$} -- (1.73606, 0.28);

\node at (2,0.45) {$C_2$};
\node at (1.7,0.4) {$L_2$};
\node at (2.13,0.16) {$S_2$};

\end{tikzpicture}
\end{center}

\begin{center}
\begin{tikzpicture}[scale=5.6]

\draw [lightgray, line width=0.2cm] (0,0) node[below, black]{\scriptsize $0$} -- (0.23606,0) node[below, black]{\scriptsize $\alpha^{3}$};
\draw [thick,black] (0.14589, -0.03) -- (0.14589, 0.03);
\draw [lightgray, line width=0.2cm] (0,0.2) node[below, black]{\scriptsize $\alpha$} -- (0.23606,0.2) node[below, black]{\scriptsize $\alpha+\alpha^{3}$};
\draw [thick,black] (0.14589, 0.17) -- (0.14589, 0.23);
\draw [lightgray, line width=0.2cm] (0,0.4) node[below, black]{\scriptsize $\alpha^{2}$} -- (0.23606,0.4) node[below, black]{\scriptsize $\alpha$};
\draw [thick,black] (0.14589, 0.37) -- (0.14589, 0.43);

\draw [gray, line width=0.2cm] (0.5,0) node[below, black]{\scriptsize $\alpha^{3}$} -- (0.64589,0) node[below, black]{\scriptsize $\alpha^{2}$};
\draw [gray, line width=0.2cm] (0.5,0.2) node[below, black]{\scriptsize $\alpha +\alpha^{3}$} -- (0.64589,0.2) node[below, black]{\scriptsize $1$};

\node at (0.35,0.5) {$C_3$};
\node at (0.14,0.48) {$L_3$};
\node at (0.57, 0.28){$S_3$};

\end{tikzpicture}

\caption{Partial graph of the cutting-stacking method for $(\alpha^n\omega)_{n\ge 1}$}\label{cut_stack}
\end{center}
\end{figure}

To the above cutting-stacking construction it is naturally associated an interval exchange, whose explicit expression is given by the following result.

\begin{proposition}\label{9}
The interval exchange corresponding to the cutting-stacking procedure described in Definition \ref{8} is the map $T:[0,1[\rightarrow [0,1[$ whose restriction to $ I_{k}$ is  $T_k$, where
 $$T_1(x)=x+\alpha \quad {\rm if\ }\ x \in I_1=[0,\alpha^{2}[ \,$$
and, for every $k\geq 1$,
$$T_{2k}(x)=x+\alpha^{2k}-\sum_{j=0}^{k-1}\alpha^{2j+1} \quad {\rm if}\ x \in I_{2k}=\left[\sum_{j=0}^{k-1}\alpha^{2j+1}, \sum_{j=0}^{k}\alpha^{2j+1}\right[$$
and
$$T_{2k+1}(x)=x+\alpha^{2k+1}-\sum_{j=0}^{k-1}\alpha^{2(j+1)} \quad {\rm if}\ x \in I_{2k+1}= \left[\sum_{j=0}^{k-1}\alpha^{2(j+1)}, \sum_{j=0}^{k}\alpha^{2(j+1)}\right[.$$
\end{proposition}

\begin{proof} 
If we write $T(x)=x+c_k$ whenever  $x \in I_k$ for any $k \ge 1$, we simply observe that $I_k+c_k=[\alpha^k, \alpha^k+\alpha^{k+1} [ $ for all $k \ge 1$.  Therefore, $\lambda(\cup_{k\geq 0}(I_k+c_k))=1$ and $(I_h+c_h)\cap (I_k+c_k) =\emptyset $ whenever $h \neq k$, which proves that $T$ is an interval exchange.

Now we show how the map $T$ acts in the cutting-stacking procedure, proving that in each column $C_n=\{L_n, S_n \}$ the transformation $T$ maps each interval of $L_n$ (respectively, $S_n$) onto the interval above it and the {\it top} interval of $L_n^0$, denoted as usual by $top(L_n^0)$, onto the bottom interval of $S_n$. 

 Let us start with $C_1=\{L_1, S_1 \}$. When we divide the column $L_1$ into two sub-columns $L^0_1$ and $L_1^1$, made by one interval each, we notice that $top(L^0_1)=I_1$ and $b(S_1)=T(I_1)$. Therefore, we stack $S_1$ onto $L_1^0$ using $T_1$, which maps $I_1$ onto $b(S_1)$ and gives birth to the columns $C_2=\{ L_2, S_2\}$, where $L_2=L_1^0*S_1$ and $S_2=L_1^1$.

Now we consider $n=2k$ and $C_{2k}=\{ L_{2k}, S_{2k}\}$. When we divide $L_{2k}$ proportionally to $\alpha$ and $\alpha^2$, obtaining therefore $L_{2k}^0$ and $L_{2k}^1$, we notice that $b(S_{2k})=T_{2k}(I_{2k})=[\alpha^{2k}, \alpha^{2k+1}[$ and, consequently, $top(L_{2k}^0)=I_{2k}$ because $T_{2k}$ is a bijection. In other words, $T_{2k}$ stacks the bottom of $S_{2k}$ onto the top of $L_{2k}^0$ because $T_{2k}(top(L_{2k}^0))=b(S_{2k})=[\alpha^{2k}, \alpha^{2k-1}[$\ .

A simple calculation shows that if we consider the case $n={2k+1}$ we have  $T_{2k+1}(I_{2k+1})=b(S_{2k+1})=[\alpha^{2k+1}, \alpha^{2k}[$ and, therefore,  $top(L_{2k+1}^0)=I_{2k+1}$. 

This completes the proof of the proposition.
\end{proof}

\begin{figure}[h!]
\begin{center}
\begin{tikzpicture}[scale=7]
\draw (0,0) node[below]{\scriptsize 0} -- (1,0) -- (1,1) -- (0,1)node[left]{\scriptsize1} -- (0,0);
\draw [thick] (0,0.61803) node[left]{\scriptsize $\alpha$} -- (0.38196,1)
node[midway, sloped,above] {$T_1$};;
\draw [thick] (0.38196,0.23606) -- (0.52786,0.38196)
node[midway, sloped,above] {$T_3$};;

\draw [thick] (0.61803,0.38198) -- (0.85410,0.61803)
node[midway, sloped,above] {$T_2$};
\draw [thick] (0.85410,0.14589) -- (0.94427,0.23606)
node[midway, sloped,above] {$T_4$};;
\draw [dashed] (0.61803,0) node[below]{\scriptsize $\alpha^{}$} -- (0.61803,1); 
\draw [dashed] (0.38196,0) node[below left]{\scriptsize $\alpha^{2}$} -- (0.38196,1);
\draw [dashed] (0.52786,0) node[below]{\scriptsize $\alpha^{2}+\alpha^{4}$} -- (0.52786,1);
\draw [dashed] (0.85410,0) node[below left]{\scriptsize $\alpha+\alpha^{3}$} -- (0.85410,1);
\draw [dashed] (0.94427,0) node[right, below]{\scriptsize $\alpha+\alpha^{3}+\alpha^{5}$} -- (0.94427,1);
\draw [dashed] (0,0.61803) -- (1,0.61803);
\draw [dashed] (0,0.38196)  node[left]{\scriptsize $\alpha^{2}$} -- (1,0.38196);
\draw [dashed] (0,0.23606) node[left]{\scriptsize $\alpha^{3}$}  -- (1,0.23606);
\draw [dashed] (0,0.14589) node[left]{\scriptsize $\alpha^{4}$}  -- (1,0.14589);
\end{tikzpicture}
\end{center}
\caption{Partial graph of the Kakutani-Fibonacci transformation $T$}\label{kak_fib}
\end{figure}
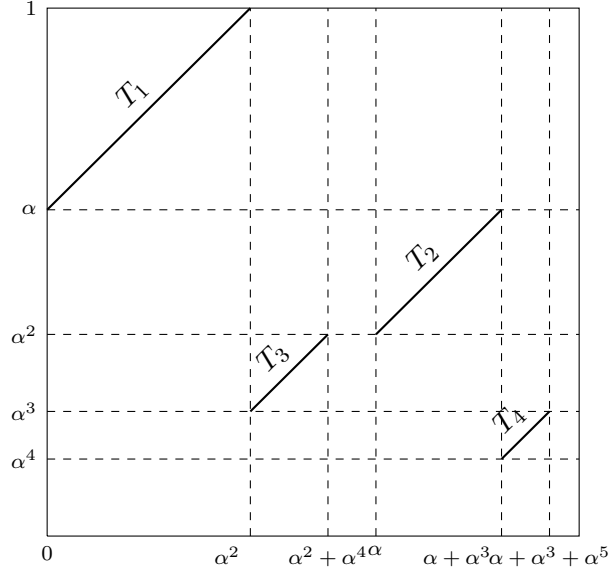

\begin{definition}\label{KF}
 The map $T:[0,1[\rightarrow [0,1[$, whose expression is given in Proposition \ref{9}, is called the \emph{Kakutani-Fibonacci transformation}. 
\end{definition}
Figure \ref{kak_fib} shows the graph of the maps $T_k$, with $1 \leq k\leq 4$.\\

It is worthwhile to compare the Fibonacci transformation $T_{Fib}$ introduced in \cite{GHL} (that we have already considered in the first chapter) and the Kakutani-Fibonacci transformation just defined. Both transformations are obtained by means of a cutting-stacking technique and at each step of the cutting-stacking procedure, the two columns obtained by means of the Fibonacci transformation have the same height and  width as the two columns obtained by means of the Kakutani-Fibonacci transformation $T$. Nevertheless, the bottoms are different as well as the order of the intervals in each column. Moreover, as we have already pointed out, the Fibonacci transformation is  not defined in $0$, while the Kakutani-Fibonacci transformation is defined in $0$ and, as we will show later, its iteration in $0$ gives rise to the Kakutani-Fibonacci sequence.\\
If we extend, as we did in Chapter 1, $T_{Fib}$ to $0$ by continuity letting $T_{Fib}(0)=\alpha$, we see that the extension of $T_{Fib}$ coincides with the rotation on $[0,1[$ mod 1 by $\frac{\sqrt{5}-1}{2}$.

\begin{theorem}\label{ergodicity}
The Kakutani-Fibonacci transformation $T$ is ergodic.
\end{theorem}
\begin{proof} The arguments we use are inspired by  \cite{Friedman}.

Let us denote by $C_m^*$ the {\it stack} made by  $L_m*S_m$ and let us define the mapping $\tau_m: [0,1[\rightarrow [0,1[$ as follows.
For $1 \le i \le l_m-1$ it is the translation of the interval $J_i$ of $L_m$ onto $J_{i+1}$, the interval just above $J_i$, and hence it coincides with $T$.
If $i=l_m$,  it is a contraction by parameter $\alpha$ of $top(L_m)$ onto $b(S_m)$.
For $1 \le i \le s_m-1$ it is the translation of the interval $\tilde J_i$ of $S_m$ onto $\tilde{J}_{i+1}$, the interval just above $\tilde J_i$, and hence it coincides with $T$ again.

Now we fix a measurable set $B$ in $[0,1[$ with $\lambda(B)>0$, such that $T^{-1}(B)=B$, and we prove that  $\lambda(B^T)=1$, where $B^T=\bigcup_{i=- \infty}^{+ \infty} T^i(B)$.

From Lebesgue density theorem, there exists a point $x_0$ having density $1$ for $B$. 
This implies that for every fixed $\epsilon>0$ there exists $\delta>0$ such that if $I$ is any interval containing $x_0$ with $\lambda(I)< \delta$, we have

\begin{equation}\label{2}\lambda(B \cap I)>(1-\epsilon) \lambda(I).
\end{equation}

\smallskip
Due to the fact that the diameter of the Kakutani-Fibonacci sequence of partition $(\alpha^m \omega)_{m \ge 1}$ tends to $0$ when $m \rightarrow \infty$, there exists $m >0$ such that $\alpha^m< \delta$  and, therefore, there exists an interval of $C^*_m$ for which (\ref{2}) holds. For sake of brevity, we will denote it by $I$.
 
As $C^*_m= L_m* S_m$, the interval $I$ could belong to the column $L_m$ or to the column $S_m$. If $I \in S_m$, it is equivalent to say that $I \in {L_{m+1}}$. For this reason, without any loss of generality we may suppose that $I \in L_m$.

Suppose now that $I$ is the $r$-th interval from below in the column $L_m$ and observe that 
\begin{equation*} \lambda \left (\bigcup_{i=-r+1}^{t_m-r} \tau_m^i(I) \right)=1.  
\end{equation*}

\smallskip
Taking the above identity and (\ref{2}) into account, we have the following inequalities:
\begin{eqnarray*}\lambda(B^T)&\ge&\lambda \left((B \cap I)^T\right)  \\
&\ge&\lambda \left(  \bigcup_{i=1-r}^{l_m-r} T^i(B \cap I) \right)+\lambda \left(  \bigcup_{j=0}^{s_m-1} \left( T^j\left (B \cap b(S_m)\right)\right) \right)\\
&=&\lambda \left(  \bigcup_{i=1-r}^{l_m-r}  \tau_m^i(B \cap I) \right)+\lambda \left(  \bigcup_{j=l_m+1-r}^{t_m-r}  \tau_m^j(B \cap I) \right)\\
&=& \sum_{i=1-r}^{l_m-r} \lambda \left( \tau_m^i(B \cap I)\right)+\sum_{j=l_m+1-r}^{t_m-r} \lambda \left( \tau_m^j(B \cap I)\right)\\
&>&(1-\epsilon)\ l_m\ \alpha^{m}+(1-\epsilon)\ s_m \ \alpha^{m+1}=1-\epsilon.
\end{eqnarray*}
As $\epsilon$ is arbitrary, we conclude that $\lambda(B)=1$. Therefore, $T$ is ergodic and the theorem is proved.
\end{proof}

A direct consequence of the above theorem and Birkhoff's Theorem (Theorem \ref{B}) is the following

\begin{theorem} \label{16} The sequence $(T^{n}(x))_{n\geq 0}$ is u.d. for almost every $x \in [0,1[$.
\end{theorem}

We now show how to get the points of the Kakutani-Fibonacci sequence by means of the  Kakutani-Fibonacci transformation.

\begin{theorem}\label{17}
The Kakutani-Fibonacci sequence of points  $(\xi_{1,1}^n)_{n \ge 1}$ coincides with $(T^{n}(0))_{n\geq 0}$.
\end{theorem}
\begin{proof}  We want to show  by induction on $n\ge 1$ that  

\begin{equation}\label{3} \big( \xi_1, \ \xi_2, \dots,\ \xi_{t_n})=\big(0, T(0), T^2(0), \dots, T^{t_{n}-1}(0)\big).
\end{equation}

\bigskip
 If $n=1$,  (\ref{3}) is obviously verified.

We suppose that (\ref{3}) is true for  $n$ and prove that
\begin{eqnarray*} &&\big(\ \xi_1, \ \xi_2, \dots,\ \xi_{t_n}, \phi_{n+1}(\xi_1),\dots ,\phi_{n+1}(\xi_{l_n}))\\
=&&\big(0, T(0), T^2(0), \dots, T^{t_{n}-1}(0), T^{t_n}(0), T^{t_n+1}(0), \dots, T^{t_{n}+l_n-1}(0)\big).
\end{eqnarray*}

Due to the definition of the set $\Lambda_{1,1}^n$ of points of the $n+1$-th Kakutani-Fibonacci sequence of partitions (see Example \ref{example8}) and the inductive assumption, it is sufficient to prove that
 \begin{eqnarray*}&&\big(\phi_{n+1}(0),\phi_{n+1}(T(0)), \dots ,\phi_{n+1}(T^{t_n-1}(0))\big)\\
= &&\big(T^{t_n}(0), T^{t_n+1}(0), \dots, T^{t_{n}+l_n-1}(0)\big).\end{eqnarray*}

As $\phi_{n+1}(x)=x+\alpha^{n+1}$, the above identity  is equivalent to  
 \begin{equation*}\big(\alpha^{(n+1)},T(0)+\alpha^{(n+1)}, \dots ,T^{t_n-1}(0)+\alpha^{(n+1)}\big)\end{equation*}
\begin{equation}\label{4}=\big(T^{t_n}(0), T^{t_n+1}(0), \dots, T^{t_{n}+l_n-1}(0)\big).\end{equation}

\bigskip
In order to prove (\ref{4}), we focus our attention on the intervals of the column $S_{n+1}$ and on their left endpoints.

We note that, due to the cutting-stacking procedure described in Definition \ref{8} and to the nature of $T$ described in the proof of Proposition \ref{9}, and specifically to the fact that  $b(S_{n+1})=T(top(L_{n+1}))=T(T^{l_{n}+s_n-1}(\ [0, \alpha^{n+1}[\ ) )$,  
the columns $C_{n+1}=\{L_{n+1}, S_{n+1} \}$ can be written as follows:
\begin{eqnarray*} L_{n+1}&&=L_n^0*S_n=\Big(b(L_n^0)* \cdots *top(L_n^0)\Big)*\Big(b(S_n)*\cdots*top(S_n)\Big)\\
&&=\Big([0, \alpha^{n+1}[ \ * \cdots *\ T^{l_{n}-1}(\ [0, \alpha^{n+1}[\ )\ \Big)*\\
  &&\ * \ \Big( T^{l_{n}}(\ [0, \alpha^{n+1}[\ )\ *\cdots\ *\ T^{l_{n}+s_n-1}(\ [0, \alpha^{n+1}[\ ) \Big)
 \end{eqnarray*}
and
 \begin{eqnarray*} S_{n+1}=L_n^1&&=b(L_n^1)* \cdots *top(L_n^1)\\
 &&=\Big([\alpha^{n+1}, \alpha^{n}[\ * \cdots *\ T^{l_{n}-1}\left (\  [\alpha^{n+1}, \alpha^{n}[\ \right  ) \Big)\\
&&= \Big( T^{t_{n}}(\ [0, \alpha^{n+1}[\ )\ *\cdots\ *\ T^{t_{n}+l_n-1} \left(\ [0, \alpha^{n+1}[\ \right) \Big).
 \end{eqnarray*}

 \bigskip
 Therefore, as $T$ is an interval exchange, the $t_n$ left endpoints of  $L_{n+1}$ are $0, T(0), \cdots, T^{t_n-1}(0) $ and the left endpoints of $S_{n+1}$, which are a right shift by the constant $\alpha^{n+1}$ of the left endpoints of $L_{n+1}$, are $T^{t_n}(0), T^{t_n+1}(0),$ $ \cdots, T^{t_n+s_n-1}(0)$, which proves (\ref{4}).
 
The theorem is now completely proved.
\end{proof}
\begin{remark}
As we have already pointed out, the Kakutani-Fibonacci sequence of points $(\xi_{1,1}^n)_{n \ge 1}$ is an example of uniformly distributed sequence (in fact low-discrepancy). Unfortunately, this property can not be detected by this ergodic approach, since Birkhoff's Theorem (Theorem \ref{B}) implies the uniform distribution of the orbit $(T^nx)_{n\in\mathbb{N}}$ only for almost every point $x\in[0,1[$ and not for every $x\in [0,1[$.
\end{remark}

We will see in the next chapter that the Kakutani-Fibonacci transformation is not only ergodic, but in fact uniquely ergodic and therefore we can prove  the uniform distribution not only of the Kakutani-Fibonacci sequence, that was already known to be so, but of all orbits $(T^nx)_{n\in\mathbb{N}}$, for every $x\in [0,1[$.

\chapter{Ergodic properties of ${\boldsymbol \beta}$-adic Halton sequences}
\addtocounter{section}{1}
\setcounter{definition}{0}
 \thispagestyle{empty}
In this chapter we will construct a parametric extension of the classical $s$-dimensional Halton
sequence, where the bases are special Pisot numbers. We use methods from ergodic
theory in order to investigate the distribution behavior of such sequences constructed as orbits of uniquely ergodic transformations. As a consequence of the main result, we show that the
Kakutani-Fibonacci transformation, introduced in the previous chapter, is uniquely ergodic.\\
The content of this chapter is presented in \cite{HIT}.\\

Our first goal is to prove that the product of systems of the type $(\mathcal{K}_G, \tau)$ considered in Chapter 1  is uniquely ergodic.
\begin{theorem}\label{main4.1}
 Let $G^1, \ldots, G^s$ be numeration systems given by \eqref{rec} and let the coefficients of the linear recurrences be given as $a_j^i = b_i, ~j = 0, \ldots, (d_i - 1), ~i = 1, \ldots, s,$ with pairwise coprime positive integers $b_i, ~i = 1, \ldots, s$. Furthermore let $\frac{\beta_i^{k}}{\beta_j^l} \notin \mathbb{Q}$, for all $l, k \in \mathbb{N}$, where $\beta_1, \ldots, \beta_s$ denote the characteristic roots of the numerations systems. Then the dynamical system which is constructed as the $s$-dimensional Cartesian product of the odometers, i.e.\ $((\mathcal{K}_{G^1}, \tau_1) \times \ldots \times (\mathcal{K}_{G^s}, \tau_s))$, is uniquely ergodic.
\end{theorem}
\begin{proof}
 It follows by Remark \ref{solo} that each numeration system $G^j, ~j=1, \ldots, s$ fulfills Hypothesis \ref{hypB}, thus the components of the $s$-dimensional dynamical system are uniquely ergodic. By Theorem \ref{product}, we derive that the Cartesian product is uniquely ergodic if and only if $\Gamma_j \cap \Gamma_k = \{1\}$ for all $1 \leq j < k \leq s$, where we denote by $\Gamma_i$ the spectrum of $\tau_i$. As noted in \cite{glt}, we have
\begin{equation}\label{approx}
 \lim_{n \rightarrow \infty} \frac{G^j_n}{\beta_j^n} = C_j,
\end{equation}
where the constant $C_j$ can be computed by residue calculus. Using the standard notation $\sim$ for asymptotic equality (if $n \rightarrow \infty$) we obtain for fixed $l \in \mathbb{N}$
\begin{align*}
 \exp \left( 2 \pi i \frac{G^j_n}{\beta_j^l} \right) &\sim \exp \left( 2 \pi i C_j \beta_i^{n - l} \right)\\
 &\sim \exp \left( 2 \pi i G^j_{n - l} \right),
\end{align*}
and thus
\begin{equation*}
 \lim_{n \rightarrow \infty} \exp \left( 2 \pi i \frac{G^j_n}{\beta_j^l} \right) = \lim_{n \rightarrow \infty} \exp \left( 2 \pi i G^j_{n - l} \right) = 1,
\end{equation*}
where $C_j$ is given by \eqref{approx}. Furthermore, it is easy to see that for every $k \in \mathbb{N}$ there exists an $n_0$ with $b_j^k \mid G^j_{n}$ for all $n \geq n_0$ and there exist no $b', n_0' \in \mathbb{N}$ with $\gcd(b', b_j) = 1$ such that $b' \mid G^j_{n}$ for all $n \geq n_0'$. Then $\Gamma_j$ can be written as
\begin{equation*}
 \Gamma_j = \left\{ \exp \left( 2 \pi i \frac{c}{b_j^m \beta_j^l} \right) \colon m,l,c \in \mathbb{N} \cup \{0\} \right\}.
\end{equation*}
This completes the proof since $\Gamma_j \cap \Gamma_k = \{1\}$.
\end{proof}

Now we want to construct an isomorphism between $(\mathcal{K}_G, \tau)$ and $([0,1[, T)$, where $T$ is conjugate to $\tau$, i.e. $T=\phi_\beta\circ \tau\circ \phi_\beta^{-1}$. The construction recalls the one considered for $(\mathbb{Z}_b,\tau)$ in Example \ref{Zb}.\\
The first step consists in extending the definition of the Monna map to irrational bases $\beta > 1$. Let
\begin{equation*}
 n = \sum_{j \geq 0} \epsilon_j(n) G_j
\end{equation*}
be the $G$-expansion of an integer $n$. For short we write $\epsilon_j$ and define the $\beta$-adic Monna map $\phi_\beta \colon \mathcal{K}_G \rightarrow \mathbb{R}^+$ as
\begin{equation*}
\phi_\beta(n) = \phi_\beta \left(\sum_{j \geq 0} \epsilon_j(n) G_j \right) = \sum_{j \geq 0} \epsilon_j(n) \beta^{-j-1}\ .
\end{equation*}
Furthermore, we define the radical inverse function as the restriction of $\phi_\beta$ on $\mathcal{K}^0_G$ and similarly we define the pseudo-inverse $\phi_{\beta}^{-1}$. In this context the $\boldsymbol \beta$-adic Halton sequence is given as $(\phi_{\boldsymbol \beta}(n))_{n \in \mathbb{N}} = (\phi_{\beta_1}(n), \ldots, \phi_{\beta_s}(n))_{n \in \mathbb{N}}$, where $\boldsymbol \beta = (\beta_1, \ldots, \beta_s)$ and the $\beta_i$ are characteristic roots of the numeration systems $G^i$, respectively.\\ 

Note that even if one of the Hypotheses \ref{hypA} or \ref{hypB} holds, this does not imply that the image of $\mathcal{K}^0_G$ under $\phi_\beta$ is contained in $[0,1[$ and dense in it.
\begin{lemma}\label{vdC}
 Let $\mathbf{a} = (a_0, \ldots, a_{d-1})$, where the integers $a_0, \ldots, a_{d-1} \geq 0$ are the coefficients defining the numeration system $G$ and assume that the corresponding characteristic root $\beta$ satisfies 
\begin{equation}\label{par}
 \beta = a_0 + \frac{a_1}{\beta} + \ldots + \frac{a_{d - 1}}{\beta^{d - 1}},
\end{equation}
where $a_0 = \lfloor \beta \rfloor$. Furthermore, assume that there is no $\mathbf{b} = (b_0, \ldots, b_{k-1})$ with $k < d$ such that $\beta$ is the characteristic root of the polynomial defined by $\mathbf{b}$. Then $\phi_\beta(\mathbb{N}) \subset [0,1[$ and $\phi_\beta(\mathbb{N}) \not\subset [0,x[$ for all $0 < x < 1$ if and only if $\mathbf{a}$ can be written either as
\begin{align}
 \mathbf{a} &= (a_0, \ldots, a_0) \text{ or as }\label{case1}\\
 \mathbf{a} &= (a_0, a_0 - 1, \ldots, a_0 - 1, a_0),\label{case2}
\end{align}
where $a_0 > 0$.
\end{lemma}
\begin{proof}

It follows from \eqref{par} that
\begin{equation}\label{eins}
 \frac{a_0}{\beta} + \ldots + \frac{a_{d-1}}{\beta^{d}} = 1.
\end{equation}
Furthermore for all admissible representations $(\varepsilon_0, \varepsilon_{1},\varepsilon_{2},\dots)$ we have
\begin{equation}\label{lex}
 (\varepsilon_k,\varepsilon_{k-1},\dots,\varepsilon_0,0^{\infty})< (a_0,a_1,\dots,a_{d-1})^{\infty}
\end{equation}
for every $k$, when $<$ denotes the lexicographic order.\\

We consider a representation given by $\mathbf{c} = (c_0, \ldots, c_{k - 1})^\infty$ for $k > 0$ and assume that there are no positive integers $c'_i$ and $m < k$ such that $\mathbf{c} = (c'_0, \ldots, c'_{m - 1})^\infty$. We obtain
\begin{align*}
 \phi_\beta(\mathbf{c}) &= \sum_{i = 0}^\infty \left( \frac{c_0}{\beta} + \ldots + \frac{c_{k - 1}}{\beta^{k}} \right) \left( \frac{1}{\beta^{k}}\right)^i\\
 &= \left( \frac{c_0}{\beta} + \ldots + \frac{c_{k - 1}}{\beta^{k}} \right) \frac{1}{1 - \frac{1}{\beta^{k}}}\\
 &= \left( \frac{c_0}{\beta} + \ldots + \frac{c_{k - 1}}{\beta^{k}} \right) \frac{\beta^{k}}{\beta^{k} - 1}.
\end{align*}
To finish the proof of the lemma, the following will be shown: the maximum of $\phi_\beta(\mathbf{c})$ (extended over all representations $\mathbf{c}$) is 1, provided that \eqref{case1} or \eqref{case2} are satisfied, i.e.\
\begin{equation}\label{one2}
 \frac{c_0}{\beta} + \ldots + \frac{c_{k - 2}}{\beta^{k-1}} + \frac{c_{k - 1} + 1}{\beta^{k}} = 1.
\end{equation}

We have that $k \geq d$ since $\mathbf{a}$ is by assumption the minimal representation of $\beta$. In the case $k = d$ we get
\begin{equation*}
 \frac{c_0}{\beta} + \ldots + \frac{c_{d - 2}}{\beta^{d-1}} + \frac{c_{d - 1} + 1}{\beta^{d}} =  \frac{a_0}{\beta} + \ldots + \frac{a_{d-1}}{\beta^{d}} = 1, 
\end{equation*}
thus $c_0 = a_0, \ldots, c_{d-2} = a_{d - 2}$ and $c_{d-1} + 1 = a_{d-1}$. Moreover, by \eqref{lex}, we have $\max_{i = 0, \ldots, d-1} c_i \leq a_0$ and thus $a_0 = \max( \max_{i = 0, \ldots, d-2} a_i, a_{d - 1} - 1)$. Let $m \geq 0$ be the maximal integer such that $a_i = a_0$ for $i \leq m$. Then, if $a_{d-1} = a_0 + 1$ we obtain that $\mathbf{c} = (a_0, a_1, \ldots, a_{d-2}, a_0)^\infty$ contains $m + 1$ successive $a_0$'s, thus $\mathbf{c}$ is not admissible and $a_0 = \max_{i = 0, \ldots, d-1} a_i$. Similar arguments yield $k \leq d$ since by assumption $\mathbf{c} \neq (c'_0, \ldots, c'_{m - 1})^\infty$ for $m < k$.\\

Assume now that $k = d$, $a_0 = \max_{i = 0, \ldots, d-1} a_i$ and let $m \geq 0$ be defined as above. By \eqref{eins} and \eqref{lex} $\mathbf{c^*} = (a_0, \ldots, a_{m}, (a_0-1, a_0, \ldots, a_{m})^\infty)$ is admissible and 
\begin{equation*}
 \phi_\beta(\mathbf{c}^*) \geq 1,
\end{equation*}
where equality holds if $m$ is either $0$ or $d - 1$ which are the cases \eqref{case1} and \eqref{case2}. This yields the assertion \eqref{one2}, thus the proof of the lemma is complete. 
\end{proof}

If we drop the condition that $d$ has to be minimal, we obtain the following additional cases for which the above lemma is satisfied:
\begin{equation}
  \mathbf{a} = (a_0, \ldots, a_0, a_0 + 1)\label{case4}
\end{equation}
and
\begin{equation}\label{case3}
 \mathbf{a} = (\mathbf{a}', \ldots, \mathbf{a}', \mathbf{a}''),
\end{equation}
where $a_0 > 0$, $\mathbf{a}', \mathbf{a}''$ are of equal length and of the form
\begin{align*}
 \mathbf{a}' &= (a_0, \ldots, a_0, a_0 - 1),\quad \mathbf{a}'' = (a_0, \ldots, a_0)\text{ or }\\
 \mathbf{a}' &= (a_0, a_0 - 1, \ldots, a_0 - 1),\quad \mathbf{a}'' = (a_0, a_0 - 1, \ldots, a_0 - 1, a_0).
\end{align*}
Note that \eqref{case4} is another way to represent the $(a_0 + 1)$-adic number system, which is a special case of \eqref{case1} and obviously fulfills the lemma. Furthermore, condition \eqref{case3} is a reformulation of \eqref{case1} and \eqref{case2}, thus in the sequel we only consider numeration systems which satisfy \eqref{case1} or \eqref{case2}.

\begin{lemma}\label{lem3}
 Let $G$ be a numeration system of the form \eqref{rec}, assume that the coefficients of the linear recurrence are given by $a_j = a, ~j = 0, \ldots, (d-1),$ for a positive integer $a$ and let $\beta$ denote the corresponding characteristic root. Then $\mu(Z) = \lambda(\phi_\beta(Z))$ for every cylinder set $Z$.
\end{lemma}
\begin{proof}
 Let the cylinder set $Z$ be defined by the fixed digits $\epsilon_0, \ldots, \epsilon_{k-1}$. Assume first that $\epsilon_{k-1} < a$, then $F_{k+r} = (a+1)^r$ for $0 \leq r < d$. Thus, by \eqref{mu}, we obtain that
\begin{equation*}
 \mu(Z) = \beta^{-k}.
\end{equation*}

Consider the $\beta$-adic Monna map of $n \in \mathbb{N}$, i.e.\
\begin{equation*}
 \phi_\beta(n) = \sum_{i = 0}^\infty \frac{\epsilon_i}{\beta^{i+1}}.
\end{equation*}
If $\epsilon_{k-1} < a$ we easily see that $\phi_\beta(Z)$ is dense in 
\begin{equation*}
I = \left[\sum_{i = 0}^{k-1} \frac{\epsilon_i}{\beta^{i+1}}, \sum_{i = 0}^{k-2} \frac{\epsilon_i}{\beta^{i+1}} + \frac{(\epsilon_{k-1} + 1)}{\beta^{k}} \right[
\end{equation*}
and that $\phi_\beta(x') \notin I$ if $x' \notin Z$. Thus $\phi_\beta(Z)$ is $\lambda$-measurable and $\lambda(\phi_\beta(Z)) = \lambda(I) = \beta^{-k}$.\\

Assume now that $Z$ is defined by the fixed digits $\epsilon_{0},\ldots , \epsilon_{k-2}$ and $\epsilon_{k-1} = a$. By the above argument we derive that a cylinder with fixed digits $\epsilon_{0} ,\ldots, \epsilon_{k-2}$, and $\epsilon_{k-2}<a$, has measure $\beta^{-(k-1)}$.\\

Now compute the measure of $Z$:
\begin{equation*}
 \mu(Z) = \beta^{-(k-1)} - (a - 1) \beta^{-k}\ .
\end{equation*}
Next we consider $\phi_\beta(Z)$, hence
\begin{equation*}
\phi_\beta(Z) = \left[\sum_{i = 0}^{k-1} \frac{\epsilon_i}{\beta^{i+1}}, \sum_{i = 0}^{k-3} \frac{\epsilon_i}{\beta^{i+1}} + \frac{(\epsilon_{k-2} + 1)}{\beta^{k}} \right[
\end{equation*}
and thus $\lambda(\phi_\beta(Z)) = \mu(Z)$.\\ 

Let $2 \leq h \leq \min(k, d-1)$ and consider a cylinder set $Z$ with fixed digits $\epsilon_{0},\ldots , \epsilon_{k-h-1} < a$ and $\epsilon_{k-l}=a$ for $l = 1, \ldots, h$. Then, the cylinder with fixed digits $\epsilon_{0},\ldots , \epsilon_{k-h-1}$ has measure $\beta^{-(k-h)}$ and every cylinder with digits $\epsilon_{0},\ldots , \epsilon_{k-h+1}$ has measure $\beta^{-(k-h+2)}$. Thus it follows
\begin{equation*}
 \mu(Z) = \beta^{-(k-h+1)} - (a - 1) \beta^{-(k-h+2)}.
\end{equation*}
Considering $\phi_\beta(Z)$ we have
\begin{equation*}
\phi_\beta(Z) = \left[\sum_{i = 0}^{k-h+1} \frac{\epsilon_i}{\beta^{i+1}}, \sum_{i = 0}^{k-h} \frac{\epsilon_i}{\beta^{i+1}} + \frac{(\epsilon_{k-h-1} + 1)}{\beta^{k-h+2}} \right[,
\end{equation*}
and therefore $\lambda(\phi_\beta(Z)) = \mu(Z)$.
\end{proof}

As mentioned above, a result of Frougny and Solomyak \cite[Lemma 3]{frougny} implies that the dominant root of
\begin{equation*}
 x^2 - a_0 x - a_1, \quad a_0,a_1 \geq 1,
\end{equation*}
is a Pisot number if and only if $a_0 \geq a_1$. Lemma \ref{vdC} shows that the image of $\mathcal{K}_G^0$ under $\phi_\beta$ is not a subset of $[0,1[$, when $a_0 > a_1$. It follows that Lemma \ref{lem3} characterizes all van der Corput-type constructions for $d = 2$.

\begin{theorem}\label{main4.2}
 Let $G^1, \ldots, G^s$ be numeration systems as in Theorem \ref{main4.1} and let $\beta_1, \ldots, \beta_s$ denote the roots of the corresponding characteristic equations. Then the $s$-dimensional $\boldsymbol \beta$-adic Halton sequence $(\phi_{\boldsymbol \beta}(n))_{n \in \mathbb{N}}$ is u.d.\ in $[0,1[^s$.
\end{theorem}
\begin{proof}
  By Lemma \ref{lem3} and the definition of the Monna map we obtain an isometry between the dynamical systems $((\mathcal{K}_{G^1}, \tau_1) \times \ldots \times (\mathcal{K}_{G^s}, \tau_s))$ and $(([0,1[, T_1) \times \ldots \times ([0,1[, T_s))$ where
\begin{equation*}
 T_i \colon [0,1[ \rightarrow [0,1[, \quad T_i(x) := \phi_{\beta_i} \circ \tau_i \circ \phi_{\beta_i}^{-1} (x).
\end{equation*}
Let $\mathbf{T} \mathbf{x} = (T_1 x_1, \ldots, T_s x_s)$ for $\mathbf{x} = (x_1, \ldots, x_s) \in [0,1[^s$. Hence by Birkhoff's ergodic theorem, $(\mathbf{T}^n \mathbf{x})_{n \in \mathbb{N}}$ is u.d.\ in $[0,1[^s$ for all $\mathbf{x} \in [0,1[^s$. In particular, $(\phi_{\boldsymbol \beta} (n))_{n \in \mathbb{N}} = (\mathbf{T}^n \mathbf{0})_{n \in \mathbb{N}}$ is u.d.
\end{proof}

 Note that the classical $\mathbf{b}$-adic Halton sequence with pairwise coprime integer bases $b_1, \ldots, b_s \geq 2$, is a special case of Theorem \ref{main4.2}.

\begin{theorem}\label{101}
 Let the numeration system $G$ be defined by the coefficients $(a_0, a_1, a_2) = (1,0,1)$, let $\beta$ be its characteristic root and $\tau$ the odometer on $G$. Then $\mu(Z) = \lambda(\phi_{\beta}(Z))$ for all cylinder sets $Z$. Thus $T(x) = \phi_\beta \circ \tau \circ \phi_{\beta}^{-1} (x)$ is uniquely ergodic and $(T^n x)_{n \in \mathbb{N}}$ is u.d.\ for all $x$ in $[0,1[$. Furthermore the spectrum of $T$ is given by
\begin{equation}\label{spec}
  \Gamma = \left\{ \exp \left( 2 \pi i \frac{c}{\beta^l} \right) \colon m,l,c \in \mathbb{N} \cup \{0\} \right\}.
\end{equation}

\end{theorem}
\begin{proof}
It is well known that $\beta$ is a Pisot number and equation \eqref{par} holds since $\lfloor \beta \rfloor = 1 = a_0$. Hypothesis \ref{hypA} was proved for this case in \cite[Theorem 4]{bks}. The proof that Hypothesis \ref{hypB} is fulfilled can be found in \cite[Theorem 3]{aki}. Equation \eqref{spec} follows from the proof of Theorem \ref{main4.1}.\\

Now we have to prove that $\phi_\beta$ transports the measure $\mu$ to the Lebesgue measure on $[0,1[$. First we assume $k \geq 3$. Let the cylinder $Z$ be defined by the fixed digits $\epsilon_0, \ldots, \epsilon_{k - 1}$. We consider four different cases. Suppose first $\epsilon_{k-3} = \epsilon_{k-2} = \epsilon_{k-1} = 0$. Then $F_{k,0} = 1, F_{k,1} = 2, F_{k,2} = 3$ and we get by \eqref{mu} that
\begin{equation*}
 \mu(Z) = \beta^{-k}.
\end{equation*}
Furthermore, by the same argument as in the first part of the proof of Theorem \ref{main4.2} we obtain
\begin{equation*}
 \phi_\beta(Z) = \left[ \sum_{i = 0}^{k - 1} \frac{\epsilon_i}{\beta^{i+1}}, \sum_{i = 0}^{k - 2} \frac{\epsilon_i}{\beta^{i+1}} + \frac{(\epsilon_{k - 1} + 1)}{\beta^k} \right[,
\end{equation*}
and thus $\lambda(\phi_\beta(Z)) = \beta^{-k}$.\\
Now let $\epsilon_{k-3} = 1, \epsilon_{k-2} = \epsilon_{k-1} = 0$. Hence $F_{k,0} = 1, F_{k,1} = 2, F_{k,2} = 3$ and $\mu(Z) = \beta^{-k}$. We have
\begin{align*}
 \phi_\beta(Z) &= \left[\sum_{i = 0}^{k - 1} \frac{\epsilon_i}{\beta^{i+1}}, \sum_{i = 0}^{k - 1} \frac{\epsilon_i}{\beta^{i+1}} + \beta^{-k} \sum_{i = 0}^\infty \beta^{-(3 i+1)} \right[\\
  &= \left[\sum_{i = 0}^{k - 1} \frac{\epsilon_i}{\beta^{i+1}}, \sum_{i = 0}^{k - 1} \frac{\epsilon_i}{\beta^{i+1}} + \beta^{-k} \right[,
\end{align*}
thus again $\lambda(\phi_\beta(Z)) = \beta^{-k}$. Now assume $\epsilon_{k-2} = 1, \epsilon_{k-1} = 0$. Hence $F_{k,0} = 1, F_{k,1} = 1, F_{k,2} = 2$ and
\begin{equation*}
 \mu(Z) = \beta^{-k} \frac{\beta^{-2} + 1}{\beta^{-2} + \beta^{-1} + 1}.
\end{equation*}
Similarly as above we obtain
\begin{align*}
 \phi_\beta(Z) &= \left[ \sum_{i = 0}^{k - 1} \frac{\epsilon_i}{ \beta^{i+1}}, \sum_{i = 0}^{k - 1} \frac{\epsilon_i}{\beta^{i+1}} + \beta^{-(k+1)} \sum_{i = 0}^\infty \beta^{-(3 i+1)} \right[\\
  &= \left[\sum_{i = 0}^{k - 1} \frac{\epsilon_i}{\beta^{i+1}}, \sum_{i = 0}^{k - 1} \frac{\epsilon_i}{\beta^{i+1}} + \beta^{-(k+1)} \right[,
\end{align*}
thus $\lambda(\phi_\beta(Z)) = \beta^{-(k+1)}$. Now we have
\begin{equation*}
 \beta^{-(k+1)} = \beta^{-k} \frac{\beta^{-2} + 1}{\beta^{-2} + \beta^{-1} + 1}
\end{equation*}
which is equivalent to $\beta^{-3} + \beta^{-2} + \beta^{-1} = \beta^{-2} + 1$ and holds since $\beta$ is the characteristic root of the polynomial $x^3-x-1=0$.\\
In the last case we assume $\epsilon_{k - 1} = 1$, thus  $F_{k,0} = F_{k,1} = F_{k,2} = 1$ and
\begin{equation*}
 \mu(Z) = \beta^{-k} \frac{1}{\beta^{-2} + \beta^{-1} + 1}.
\end{equation*}
As above we get $\lambda(\phi_\beta(Z)) = \beta^{-(k+2)}$ and the result follows since
\begin{equation*}
 \beta^{-(k+2)} = \beta^{-k} \frac{1}{\beta^{-2} + \beta^{-1} + 1}
\end{equation*}
is equivalent to $\beta^{-3} + \beta^{-1} = 1$. The cases where $k < 3$, follow by the same arguments.
\end{proof}

As a consequence of Theorem \ref{product} we can construct uniformly distributed two-dimensional sequences $(\phi_{\beta_1}(n), \phi_{\beta_2}(n))_{n \in \mathbb{N}}$, where $\beta_1$ is the characteristic root as in Theorem \ref{101}, $\beta_2$ is the characteristic root of a numeration system as in Theorem \ref{main4.1} and $\frac{\beta^k_1}{\beta^l_2} \notin \mathbb{Q}$ for all integers $k,l > 0$.\\

Theorem \ref{101} extends the examples given in \cite[Proposition 13,14]{bg}, where the authors consider $G$-additive functions which lead to u.d.\ point sequences in the unit interval. Furthermore, it is possible to show that the one-dimensional point sequence in the previous theorem is a low-discrepancy sequence by mimicking the proof for the $b$-adic van der Corput sequence, see e.g.\ \cite{bg, carbone, kuipers_niederreiter}.\\

In the last part of this chapter we want to show that the Kakutani-Fibonacci transformation (see Definition \ref{KF}) is in fact uniquely ergodic.\\
Hence, with this different approach we can show that the orbit of $x$ under the transformation is u.d.\ for every $x \in [0,1[$.\\

For proving unique ergodicity of the Kakutani-Fibonacci transformation we need the following lemma.
\begin{lemma}\label{lemma3}
 Let $\beta$ be the golden ratio $\frac{\sqrt{5} + 1}{2}$. Then we have $T x = \phi_\beta \circ \tau \circ \phi_\beta^{-1} x$ for all $\beta$-adic rationals
\begin{equation*}
 x = \sum_{i = 1}^k \frac{\epsilon_i}{\beta^i}
\end{equation*}
with coefficients $\epsilon_i \in \{0, 1\}$.
\end{lemma}
\begin{proof}
 First let us observe that $\alpha=\frac{1}{\beta}$. It follows from \cite[Lemma 3]{AHZ} specializing $L=S=1$ that every positive integer $n$ has a representation of the form 
\begin{equation*}
n=\sum_{i=0}^N\epsilon_it_i\ ,
\end{equation*}
where $t_i$ is the number of total intervals of the $i$-th partition. This implies that $\xi_{1,1}^n=\sum_{i=0}^N\frac{\epsilon_i}{\beta^{i+1}}$. We want to show that $Tx=\phi_\beta\circ \tau\circ \phi_\beta^{-1}(x)$. Then we need to see if $x$ belongs to $I_1$, $I_{2k}$ or $I_{2k+1}$.\\
If $x\in I_1=\left[ 0,\frac{1}{\beta^2} \right[$, then $x=\sum_{j=2}^\infty\frac{\epsilon_j(x)}{\beta^{j+1}}$. In this case
\begin{equation*}
T_1(x)=x+\frac{1}{\beta}=\frac{1}{\beta}+\sum_{j=2}^\infty\frac{\epsilon_j(x)}{\beta^{j+1}}
\end{equation*}
and
\begin{eqnarray*}
\phi_\beta\circ \tau\circ \phi_\beta^{-1}(x)&=&\phi_\beta\circ \tau\left( 00\epsilon_2\epsilon_3\dots \right)\\
&=&\phi_\beta(10\epsilon_2\epsilon_3\dots)\\
&=&\frac{1}{\beta}+\sum_{j=2}^\infty\frac{\epsilon_j}{\beta^{j+1}}\ .
\end{eqnarray*}
Let us fix $k\geq 1$. If $x\in I_{2k}=\left[\sum_{j=0}^{k-1}\alpha^{2j+1}, \sum_{j=0}^{k}\alpha^{2j+1}\right[$, then $x=\sum_{j=0}^{k-1}\alpha^{2j+1}+\sum_{2k+2}^\infty \epsilon_j\alpha^{j+1}$. Hence
\begin{eqnarray*}
T_{2k}(x)&=&x+\alpha^{2k}-\sum_{j=0}^{k-1}\alpha^{2j+1}\\
&=& \sum_{j=0}^{k-1}\alpha^{2j+1}+\sum_{2k+2}^\infty \epsilon_j\alpha^{j+1}\alpha^{2k}-\sum_{j=0}^{k-1}\alpha^{2j+1}\\
&=& \sum_{2k+2}^\infty \epsilon_j\alpha^{j+1}\alpha^{2k}\ .
\end{eqnarray*}
and
\begin{eqnarray*}
\phi_\beta\circ\tau\circ\phi_\beta^{-1}(x)&=&\phi_\beta\circ\tau\left(1010\dots0100\epsilon_{2k+2}\epsilon_{2k+3}\dots \right)\\
&=&\phi_\beta(01101\dots 0100\epsilon_{2k+2}\epsilon_{2k+3}\dots)\\
&=&\dots = \phi_\beta(0\dots 01100\epsilon_{2k+2}\epsilon_{2k+3}\dots)\\
&=&\phi_\beta(0\dots 010\epsilon_{2k+2}\epsilon_{2k+3}\dots)\\
&=& \frac{1}{\beta^{2k}}+\sum_{j=2k+2}^\infty\frac{\epsilon_j}{\beta^{j+1}}\\
&=&\alpha^{2k}+\sum_{j=2k+2}^\infty\epsilon_j\alpha^{j+1}\ .
\end{eqnarray*}
Finally, if $x\in I_{2k+1}=\left[\sum_{j=0}^{k-1}\alpha^{2(j+1)}, \sum_{j=0}^{k}\alpha^{2(j+1)}\right[$, then $x=\sum_{j=0}^{k-1}\alpha^{2(j+1)}$\linebreak $+\sum_{2k+3}^\infty \epsilon_j\alpha^{j+1}$. Hence
\begin{eqnarray*}
T_{2k+1}(x)&=&x+\alpha^{2k+1}-\sum_{j=0}^{k-1}\alpha^{2(j+1)}\\
&=& \sum_{2k+3}^\infty \epsilon_j\alpha^{j+1}+\alpha^{2k+1}\ .
\end{eqnarray*}
and
\begin{eqnarray*}
\phi_\beta\circ\tau\circ\phi_\beta^{-1}(x)&=&\phi_\beta\circ\tau\left(0101\dots0100\epsilon_{2k+3}\epsilon_{2k+4}\dots \right)\\
&=&\phi_\beta(1101\dots 0100\epsilon_{2k+3}\epsilon_{2k+4}\dots)\\
&=&\dots = \phi_\beta(0\dots 01100\epsilon_{2k+3}\epsilon_{2k+4}\dots)\\
&=&\phi_\beta(0\dots 010\epsilon_{2k+3}\epsilon_{2k+4}\dots)\\
&=& \frac{1}{\beta^{2k+1}}+\sum_{j=2k+3}^\infty\frac{\epsilon_j}{\beta^{j+1}}\\
&=&\alpha^{2k+1}+\sum_{j=2k+3}^\infty\epsilon_j\alpha^{j+1}\ .
\end{eqnarray*}
  So we have proved that the two maps are the same on the set of all $\beta$-adic rationals. Combining this with the ergodicity of the Kakutni-Fibonacci transformation we complete the proof.
\end{proof}

\begin{theorem}\label{th3}
 The Kakutani-Fibonacci transformation is uniquely ergodic.
\end{theorem}
\begin{proof}
By Theorem \ref{main4.2}, $\phi_\beta \circ \tau \circ \phi_\beta^{-1}$ is uniquely ergodic, thus by Lemma \ref{lemma3} and the density of the $\beta$-adic rationals $T$ is uniquely ergodic.
\end{proof}

\chapter{Integration with respect to copulas}

\addtocounter{section}{1}
\setcounter{definition}{0}
 \thispagestyle{empty}
As we have already highlighted in the first chapter, there is a strong connection between uniformly distributed sequences and the estimate of integrals, especially of multidimensional ones. In this last chapter we want to focus on the integration of two-dimensional functions with respect to copulas. We already know that a copula is the asymptotic distribution function of a two-dimensional sequence $(x_n, y_n)_{n\in\mathbb{N}}$, with $(x_n)_{n\in\mathbb{N}}, (y_n)_{n\in\mathbb{N}}$ are uniformly distributed sequences. We now provide the general definition and we want to give upper and lower bounds for these integrals. This can be done by drawing a connection to linear assignment problems. These are a family of problems consisting in how to assign $n$ items to other $n$ items in an optimal way. We will discuss this kind of problems and the way to solve them, showing how we can apply the solution for these problems to the original problems of finding bounds for integrals. Of course the second type of problems is discrete, thus this approach gives rise to an approximation technique.  Finally, we apply our approximation technique to problems in financial mathematics and uniform distribution theory, such as the model-independent pricing of first-to-default swaps. We refer to \cite{Joe, nelsen} for a complete introduction to copulas and to \cite{BDM} for a comprehensive treatment of assignment problems.
\begin{definition}[Copula]\label{cop}
 Let $C$ be a positive function on the unit square. Then $C$ is called (two)-copula iff for every $x,y \in [0,1[$
\begin{align*}
 C(x,0) &= C(0,y) = 0,\\
 C(x,1) &= x \text{ and } C(1,y) = y
\end{align*}
and for every $x_1,x_2,y_1,y_2 \in [0,1[$ with $x_2 \geq x_1$ and $y_2 \geq y_1$
\begin{equation}\label{two-inc}
 C(x_2,y_2) - C(x_2, y_1) - C(x_1, y_2) + C(x_1,y_1) \geq 0.
\end{equation}
A function which satisfies \eqref{two-inc} is called two-increasing or supermodular. In the sequel we denote by $\mathcal{C}$ the set of all two-copulas.
\end{definition}
The word copula was first employed in a mathematical or statistical sense by Abe Sklar (1959) in the theorem (which now bears his name) describing the functions that join together one-dimensional distribution
functions to form multivariate distribution functions.\\
Sklar's Theorem, which will be stated below, elucidates the role that copulas play in the relationship between multivariate distribution functions and their univariate margins.
\begin{theorem}[Sklar's Theorem]
Let $H$ be a joint distribution function with margins $F$ and $G$. Then there exists a copula $C$ such that for all $x,y\in \mathbb{R}$, 
\begin{equation}\label{sklar}
H(x,y)=C(F(x),G(y))\ .
\end{equation}
If $F$ and $G$ are continuous, then $C$ is unique; otherwise, $C$ is uniquely determined on Ran $F \times$ Ran $G$, where Ran $F$ is the range of $F$. Conversely, if $C$ is a copula and $F$ and $G$
are distribution functions, then the function $H$ defined by \eqref{sklar} is a joint distribution function with margins $F$ and $G$.
\end{theorem}

\begin{theorem}
For every copula $C$ and every $(u,v) \in [0,1]^2$,
\begin{equation}\label{frechet}
W(u,v) \leq C(u, v)\leq M(u, v)\ ,
\end{equation}
where $W(x,y) = \max(x+y-1,0)$ and $M(x,y) = \min(x,y)$ are called Fr\'echet-Hoeffding lower and upper bounds, respectively.
\end{theorem}
It is well known that the Fr\'echet-Hoeffding lower and upper bounds are copulas in the two dimensional setting. For higher dimensions an analogon of \eqref{frechet} exists, however the lower bound is in general not a copula.\\

Our main goal in this chapter is to provide bounds for
\begin{equation*}
\int_0^1\int_0^1 f(x,y)dC(x,y)\ .
\end{equation*}
Thus, we are interested in bounds of the form
\begin{equation}\label{bounds}
  \int_{[0,1[^2} f(x,y) dC_{\min}(x,y) \leq  \int_{[0,1[^2} f(x,y) dC(x,y) \leq  \int_{[0,1[^2} f(x,y) dC_{\max}(x,y),
\end{equation}
for all $C \in \mathcal{C}$, where $C_{\min}, C_{\max}$ are copulas. \\
A particularly interesting subclass of copulas for our problems are the so-called shuffles of $M$.
\begin{definition}[Shuffles of $M$]\label{shuf}
Let $n \geq 1$, $s = (s_0, \ldots, s_n)$ be a partition of the unit interval with $0 = s_0 < s_1 < \ldots < s_n = 1$, $\sigma$ be a permutation of $S_n = \{1,\ldots,n\}$ and $\omega \colon S_n \rightarrow \{-1, 1\}$. We define the partition $t= (t_0, \ldots, t_n), ~0 = t_0 < t_1 < \ldots < t_n = 1$ such that each $[s_{i-1}, s_i[ \times [t_{\sigma(i)-1}, t_{\sigma(i)}[$ is a square. A copula $C$ is called shuffle of $M$ with parameters $\{n, s, \sigma, \omega\}$ if it is defined in the following way: for all $i \in \{1, \ldots, n\}$ if $\omega(i) = 1$, then $C$ distributes a mass of $s_i - s_{i-1}$ uniformly spread along the diagonal of $[s_{i-1}, s_i[ \times [t_{\sigma(i)-1}, t_{\sigma(i)}[$ and if $\omega(i) = -1$ then $C$ distributes a mass of $s_i - s_{i-1}$ uniformly spread along the antidiagonal of $[s_{i-1}, s_i[ \times [t_{\sigma(i)-1}, t_{\sigma(i)}[$.
\end{definition}
Note that the two Fr\'echet-Hoeffding bounds $W, M$ are trivial shuffles of $M$ with parameters $\{1, (0,1), (1), -1\}$ and $\{1, (0,1), (1), 1\}$, respectively. Furthermore, it is well-known that every copula can be approximated arbitrarily close with respect to the supremum norm by a shuffle of $M$; see e.g.\ \cite[Theorem 3.2.2]{nelsen}. In the sequel we denote by $\pi_n$ the partition of the unit interval which consists of $n$ intervals of equal length.\\

In the next paragraph we want to give a description of the linear assignment problem and of the Hungarian algorithm. Then we will illustrate the close relation of \eqref{bounds} to linear assignment problems.
\paragraph{Description of the linear assignment problem and of the Hungarian algorithm}

Assignment problems deal with the question of how to assign $n$ items (jobs, students) to $n$ other items (machines, tasks).
Since there are in general many assignments possible, we are interested in the best suitable assignment for the problem under investigation. Therefore, we must state our goal by specifying an objective function. Given an $n \times n$ cost matrix $A = (a_{ij})$, where $a_{ij}$ measures the cost of assigning $i$ to $j$, we ask for an assignment with minimum total cost, i.e., the objective function $\sum_{i=1}^n a_{i\sigma(i)}$ is to be minimized. The linear sum assignment problem (LSAP) can then be stated as 
\begin{equation}\label{lin}
\min_{\sigma\in \mathcal{P}}\sum_{i=1}^n a_{i\sigma(i)}\ ,
\end{equation}
where $\sigma$ runs through all possible permutations on $n$ elements.\\
Half a century ago Harold W. Kuhn published a famous article \cite{kuhn} presenting the Hungarian algorithm, the first polynomial-time method for the assignment problem.\\
We briefly describe how it works.
The following algorithm finds an optimal assignment to a given $n \times n$ matrix.
\begin{enumerate}
\item Subtract the smallest entry in each row from all the entries of its row
\item Subtract the smallest entry in each column from all the entries
of its column
\item Draw lines through appropriate rows and columns so that all the zero entries of the cost matrix are covered and the minimum number of such lines is used
\item Test for Optimality: $(i)$ If the minimal number of covering lines is $n$, an optimal assignment of zeros is possible and we are finished. $(ii)$ If
the minimum number of covering lines is less than $n$, an optimal
assignment of zeros is not yet possible. In that case, proceed to Step 5.
\item Determine the smallest entry not covered by any line. Subtract this entry from each uncovered row, and then add it to each covered column. Return to Step 3.
\end{enumerate}
Now we see how to apply this algorithm to our problem.

\begin{theorem}\label{main1}
Let $n \geq 1$, $A = \{a_{i,j}\}_{i,j=1, \ldots,n}$ be a real-valued $n \times n$ matrix and let the function $f$ be defined as
\begin{equation*}
 f(x,y) := a_{i,j}, \quad (x,y) \in \left[ \frac{i-1}{n}, \frac{i}{n} \right [ \times \left[ \frac{j-1}{n}, \frac{j}{n} \right[.
\end{equation*}
 Then the copula $C_{\max}$ which maximizes
\begin{equation}\label{int}
 \max_{C \in \mathcal{C}} \int_{[0,1[^2} f(x,y) dC(x,y)
\end{equation}
 is given as a shuffle of $M$ with parameters $\{n, \pi_n, \sigma^*, 1\}$, where $\sigma^*$ is the permutation which solves the assignment problem
\begin{equation*}
 \max_{\sigma \in \mathcal{P}} \sum_{i = 1}^n a_{i, \sigma(i)}.
\end{equation*}
 Moreover the maximal value of \eqref{int} is given by
\begin{equation}\label{maxval}
  \int_{[0,1[^2} f(x,y) dC_{\max}(x,y) = \frac{1}{n} \sum_{i = 1}^n a_{i, \sigma^*(i)}.
\end{equation}
\end{theorem}

\begin{proof}
Let $\{C_k(x,y), k = 1,\ldots, n! = N\}$ be the set of all shuffles of $M$ with parameters of the form $\{n, \pi_n, \sigma_k, 1\}$ and let $t_k \geq 0, ~k = 1, \ldots, N$, where $\sum_{k = 1}^N t_k = 1$. Then $C'(x,y) = \sum_{k = 1}^N t_k C_k(x,y)$ is always a copula satisfying 
\begin{equation*}
 \int_{[0,1[^2} f(x,y) dC'(x,y) \leq \frac{1}{n} \sum_{i = 1}^n a_{i, \sigma^*(i)},
\end{equation*}
where $\sigma^*$ is given in the statement of the theorem.\\

For an arbitrary copula $C \in \mathcal{C}$ we define the matrix $B_C$ as
\begin{equation*}
 B_C(i,j) = n \int_{\left [ \frac{i-1}{n}, \frac{i}{n} \right [ \times \left [ \frac{j-1}{n}, \frac{j}{n} \right [} dC(x,y).
\end{equation*}
It follows by Definition \ref{cop} that $B_C$ is doubly stochastic and by Definition \ref{shuf} that $B_{C_k}$ is a permutation matrix. Furthermore it follows from the Birkhoff-von Neumann Theorem that the set of doubly stochastic matrices coincides with the convex hull of the set of permutation matrices, see e.g. \cite{Mirsky}. Thus for every $C$ there exist $t_k \geq 0, ~k = 1,\ldots,N $ with $\sum_{k = 1}^N t_k = 1$ such that
\begin{equation*}
 B_C(i,j) = \sum_{k = 1}^N t_k B_{C_k}(i,j), \quad \text{ for every } i,j,
\end{equation*}
and hence
\begin{equation*}
 \int_{[0,1[^2} f(x,y) dC(x,y) = \sum_{k = 1}^N t_k \int_{[0,1[^2} f(x,y) dC_k(x,y) \leq \frac{1}{n} \sum_{i = 1}^n a_{i, \sigma^*(i)}\ .\quad \qedhere
\end{equation*}
\end{proof}

Note that the maximal copula in Theorem \ref{main1} is by no means unique, since for instance the value of the integral in \eqref{int} is independent of the choice of $\omega$.\\ 

Obviously, we can derive a lower bound in Theorem \ref{main1} by considering \linebreak$g(x,y) = -f(x,y)$. Furthermore it is easy to see that Theorem \ref{main1} applies to all functions $f$ which are constant on sets of the form 
\begin{equation*}
 I_{i,j} = \left[ s_i, s_{i+1} \right[ \times \left[ t_j, t_{j+1} \right[, \quad i = 0, \ldots, n-1, ~ j = 0, \ldots, m-1,
\end{equation*}
where $0 = s_0 < s_1 < \ldots < s_n = 1$ and $0 = t_0 < t_1 < \ldots < t_m = 1$ are rational numbers.\\

The following generalization of our approach applies to a wide class of functions on the unit square.

\begin{theorem}\label{main2}
 Let $f$ be a continuous function on $[0,1]^2$, let the sets $I^n_{i,j}$ be given as
\begin{equation*} 
 I^n_{i,j} = \left[ \frac{i-1}{2^n}, \frac{i}{2^n} \right [ \times \left[ \frac{j-1}{2^n}, \frac{j}{2^n} \right[ \quad \text{ for } i,j = 1, \ldots, 2^n,
\end{equation*}
for every $n > 1$ and define the functions $\underline{f}_n, \overline{f}_n$ as
\begin{align}
 \underline{f}_n(x,y) &= \min_{(x,y) \in I^n_{i,j}} f\left( x, y \right), \quad \text{ for all } (x,y) \in I^n_{i,j},\notag\\
 \overline{f}_n(x,y) &= \max_{(x,y) \in I^n_{i,j}} f\left( x, y \right), \quad \text{ for all } (x,y) \in I^n_{i,j}.\label{mini}
\end{align}
 Furthermore let $\underline{C}^n_{\max}, \overline{C}^n_{\max}$ be the copulas which maximize
\begin{equation*}
  \max_{C \in \mathcal{C}} \int_{[0,1[^2} \underline{f}_n(x,y) dC(x,y)\quad \text{ and }\quad \max_{C \in \mathcal{C}} \int_{[0,1[^2} \overline{f}_n(x,y) dC(x,y),
\end{equation*}
respectively. Then
\begin{align}
 \int_{[0,1[^2} \underline{f}_n(x,y) d\underline{C}^n_{\max}(x,y) &\leq \sup_{C \in \mathcal{C}} \int_{[0,1[^2} f(x,y) dC(x,y) \notag\\ 
  &\leq \int_{[0,1[^2} \overline{f}_n(x,y) d\overline{C}^n_{\max}(x,y),\label{bound}
\end{align}
for every $n$, and
\begin{align}
 \lim_{n \rightarrow \infty} \int_{[0,1[^2} \underline{f}_n(x,y) d\underline{C}^n_{\max}(x,y) &= \lim_{n \rightarrow \infty} \int_{[0,1[^2} \overline{f}_n(x,y) d\overline{C}^n_{\max}(x,y) \notag\\ 
&= \sup_{C \in \mathcal{C}} \int_{[0,1[^2} f(x,y) dC(x,y).\label{conv}
\end{align}
\end{theorem}
\begin{proof}
The inequalities in \eqref{bound} follow immediately from the construction of $\underline{f}_n,\overline{f}_n$ and Theorem \ref{main1}. Furthermore since $f$ is continuous on $[0,1]^2$ we have that for every $\epsilon > 0$ there exists an integer $n$ such that
\begin{equation}\label{eps}
 | \overline{f}_n(x,y) - \underline{f}_n(x,y) | < \epsilon, \quad \text{ for all } (x,y) \in [0,1]^2.
\end{equation}
Moreover it follows from Theorem \ref{main1} that for every $n$ we can write
\begin{equation*}
  \int_{[0,1[^2} \underline{f}_n(x,y) d\underline{C}^n(x,y) = \frac{1}{2^n} \sum_{i = 1}^{2^n} a_{i, \sigma^*(i)}
\end{equation*}
for a permutation $\sigma^*$ and a real valued matrix $A = \{a_{i,j}\}_{i,j=1, \ldots,n}$ with
\begin{equation*}
 a_{i,j} = \min_{(x,y) \in I^n_{i,j}} f\left( x, y \right), \quad \text{ for } i,j = 1, \ldots, 2^n.
\end{equation*}
Using \eqref{eps}, we get that
\begin{equation*}
  \int_{[0,1[^2} \overline{f}_n(x,y) d\overline{C}^n(x,y) \leq \int_{[0,1[^2} (\underline{f}_n(x,y) + \epsilon) d\underline{C}^n(x,y) = \frac{1}{2^n} \sum_{i = 1}^{2^n} (a_{i, \sigma^*(i)} + \epsilon)
\end{equation*}
and therefore
\begin{equation*}
  \left | \int_{[0,1[^2} \overline{f}_n(x,y) d\overline{C}^n(x,y) - \int_{[0,1[^2} \underline{f}_n(x,y) d\underline{C}^n(x,y) \right | < \epsilon.
\end{equation*}
Combining this with \eqref{bound}, we get \eqref{conv}.  \qedhere
\end{proof}

The assumption that $f$ is continuous can, perhaps, be relaxed to the case that $f$ is $C$-continuous a.e.\ for all $C \in \mathcal{C}$. This is required to make sure that
\begin{equation*}
 \int_{[0,1[^2} f(x,y) dC(x,y)
\end{equation*}
exists for all $C \in \mathcal{C}$.\\

By defining the functions $\underline{f}_n, \overline{f}_n$ differently, we might get an approximation technique which converges faster to the optimal value, for instance we could use 
\begin{equation*}
 f_n(x,y) = f\left(\frac{i}{2^n}, \frac{j}{2^n} \right), \quad \text{ for all } (x,y) \in I^n_{i,j}.
\end{equation*}
Furthermore the mini- and maximization steps in \eqref{mini} can be time-consu-\linebreak ming, for instance when these problems are not explicitly solvable. However the advantage of the present approach lies in the fact that we get an upper and lower bound of the optimal value for every $n$, which is obviously useful for numerical applications.\\ 

In numerical investigations where \eqref{mini} could not be solved explicitly we used mini- and maximization over a fixed grid in each $I^n_{i,j}$. This results in a fast computation, however we obviously lose the property that we get upper and lower bounds for every $n$.\\

By assuming Lipschitz-continuity of $f$, we can describe the rate of convergence of our method.

\begin{corollary}
 Let the assumptions of Theorem \ref{main2} hold and, in addition assume that $f$ is Lipschitz-continuous on $[0,1]^2$ with constant $L$. Then
\begin{equation*}
 \left | \int_{[0,1[^2} \overline{f}_n(x,y) d\overline{C}^n(x,y) -  \int_{[0,1[^2} \underline{f}(x,y) d\underline{C}(x,y) \right| \leq  L \frac{\sqrt{2}}{2^n}.
\end{equation*}
\end{corollary}
\begin{proof}
 Following the proof of Theorem \ref{main2} and using the Lipschitz-conti-\linebreak nuity of $f$ we get
\begin{equation*}
 | \overline{f}_n(x,y) - \underline{f}_n(x,y) | \leq L \frac{\sqrt{2}}{2^n}, \quad \text{ for all } (x,y) \in [0,1]^2,
\end{equation*}
and therefore
\begin{equation*}
  \left | \int_{[0,1[^2} \overline{f}_n(x,y) d\overline{C}^n(x,y) - \int_{[0,1[^2} \underline{f}_n(x,y) d\underline{C}^n(x,y) \right | \leq  L \frac{\sqrt{2}}{2^n}.
\end{equation*}

\end{proof}

In this section we present two numerical examples in which we apply the approximation technique presented in Theorem \ref{main2}. We use an implementation of the Hungarian Algorithm in MatLab, which makes it possible to derive the solution of the linear assignment problem \eqref{lin} for a given matrix $A$ of size $2^{10} \times 2^{10}$ within seconds. The involved mini- or maximization of the integrand function on a given grid can be done efficiently, since the integrand functions are piecewise smooth.

Now we derive upper bounds for \eqref{inte} by maximizing $g$ over the set of all copulas. This has already been done in \cite{fial} for functions $f$ such that $\frac{\partial^2 f}{\partial_x \partial_y}(x,y)$ has constant sign for all $(x,y) \in [0,1[^2$. Note that this condition is equivalent to the two-increasing property of $f$ provided that $\frac{\partial^2 f}{\partial_x \partial_y}(x,y)$ exists on the unit square.\\ 

As a numerical example, we consider
\begin{equation*}
  \limsup_{N \rightarrow \infty} \frac{1}{N} \sum_{n = 1}^N \sin(\pi (x_n + y_n)).
\end{equation*}
The numerical results are illustrated in Table \ref{tab:t2}. Note that the approximations of the lower bound can be easily computed using the symmetry of the sine function.\\

A further interesting question concerns the sequences $(x_n)_{n > 1}, (y_n)_{n > 1}$ which maximize \eqref{inte}. Let $(x_n)_{n > 1}$ be a u.d.\ sequence and $C(x,y)$ a shuffle of $M$, then it is easy to see that $(f(x_n))_{n > 1}$ is u.d., where $f$ is the support of $C$. Thus if $C$ is the shuffle of $M$ which attains the maximum in \eqref{inte}, an optimal two-dimensional sequence is given as $(x_n, f(x_n))_{n > 1}$, where $(x_n)_{n > 1}$ is an arbitrary u.d.\ sequence. In Figure \ref{fig: max1}, we present the support of the copula which attains the upper bound for the maximum in our approximation when $n = 7$.\\

We point out that in this context the support of a copula is meant as the support of the measure $\mu_C$ induced by the copula, i.e.
\begin{equation*}
{\rm Supp(C)}=\{B\in \mathcal{B}([0,1]^2) : \mu_C(B)>0\}\ .
\end{equation*}

Although we are not able to give a rigorous proof, by increasing $n$ it seems that the copula $C'$ which attains the maximum is the shuffle of $M$ with parameters $\{2, (0,0.75,1), (1), \{\omega(1) = -1, \omega(2) = 1 \}\}$. In this case we have that 
\begin{equation*}
{\rm supp(C')} = \left\{ (x,f'(x)) : x \in \left[0,\frac{3}{4}\right], f'(x) = \frac{3}{4} - x \right\}\cup \left\{ (x,x) : x \in \left[\frac{3}{4},1\right]\right\}
\end{equation*}
and then
{\small
\begin{align*}
 \int_0^1 \int_0^1 \sin(\pi (x + y)) dC'(x,y) &= \int_0^1 \sin(\pi (x + f'(x))) dx\\
&= \int_0^{\frac{3}{4}} \sin(\pi (x + 0.75 - x)) dx + \int_{\frac{3}{4}}^1 \sin(\pi 2 x) dx\\
&= \frac{3}{4 \sqrt{2}} - \frac{1}{2 \pi} \approx 0.371175\ .
\end{align*}}

\begin{table}[!ht]
\centering
\begin{tabular}[h]{|r|r|r|r|r|r|r|}
\hline

	$n$ & 5 & 6 & 7 & 8& 9 & 10\\
\hline
UB		  & 0.3933 & 0.3824 & 0.377 & 0.3741 & 0.3727 & 0.3712\\
\hline
LB 	  & 0.3482 & 0.3598 & 0.3655 & 0.3684 & 0.3698 & 0.3711\\
\hline
\end{tabular}
 \caption{Upper and lower bounds for the maximum in \eqref{inte} with respect to $n$.}
 \label{tab:t2}
\end{table}

\begin{figure}[h!]
 \centering
 \includegraphics[scale=0.7]{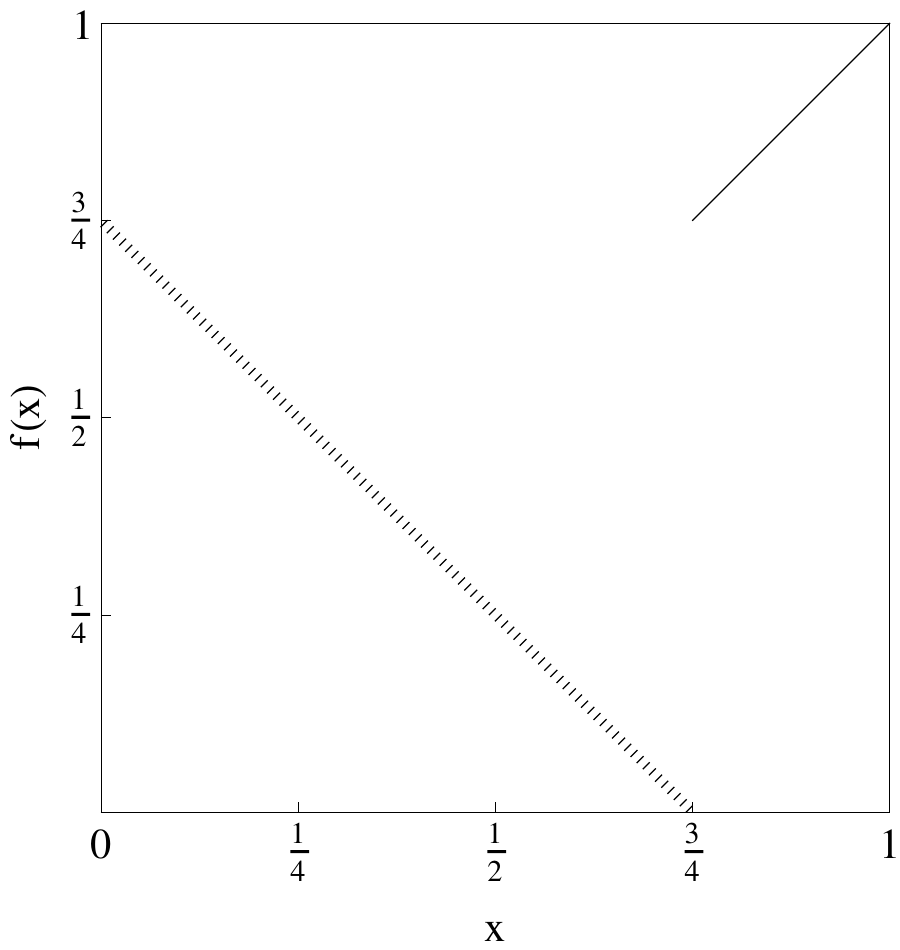}
\caption{Support of copula which attains upper bound for $\sin( \pi (X + Y))$ and $n = 7$.}
\label{fig: max1}
\end{figure}

A first-to-default swap (FTD) is a contract in which a protection seller (PS) insures a protection buyer (PB) against the loss caused by the first default event in a portfolio of risky assets. The PB pays regularly a fixed constant premium to the PS, the so-called spread, until the maturity $T$ of the contract or the first default event, whichever occurs first. In exchange, the PS compensates the loss caused by the default at the time of default.\\  

We assume that the underlying portfolio consists of two risky assets, for which the marginal default distributions are known, but the joint distribution is unknown. We want to derive a worst case bound in this setting. For the valuation of the FTD we follow the paper of Schmidt and Ward \cite{schmidt_ward}. Note that Monte Carlo methods for the evaluation of first-to-default swaps, where the dependences within the portfolio is modeled by a copula, are e.g.\ presented in Aistleitner et al.\ \cite{aht} and Packham and Schmidt \cite{packham}.\\ 

Let $\tau_1, \tau_2$ denote the random default times of the two risky assets, let the notional be equal to one for both assets and $R_i, i = 1,2,$ be the so-called recovery rates, which are the percental amounts of money that can be liquidized in case of the default of an asset. We assume that the distribution of $\tau_i$ is given as
\begin{equation*}
 \mathbb{P}(\tau_i \leq t) = 1 - e^{-\lambda_i t}, \quad t > 0,
\end{equation*}
where the intensity $\lambda_i$ can be derived from the credit default swap market as
\begin{equation*}
 \lambda_i = \frac{s_i}{1 - R_i},
\end{equation*}
and $s_i$ is the premium of an insurance against the default of asset $i$.\\

Now denote by $\tau = \min(\tau_1, \tau_2)$ the first default time in the portfolio, let $0 = t_0 < t_1 < \ldots < t_n = T$ be the payment times of the constant spread and assume that there exists a risk free interest rate $r \geq 0$. Then, to guarantee a fair spread $s$, we obtain that the expected, discounted premium and default payments are equal, i.e.\
\begin{equation*}
 s \sum_{i = 0}^n e^{- r t_i} \mathbb{P}(\tau > t_i) = \sum_{i = 1}^2 \mathbb{E}\left[ (1- R_i) e^{- r \tau} \mathbf{1}_{\{ \tau < T \wedge \tau = \tau_i\}} \right].
\end{equation*}
By the above assumptions we obtain that
\begin{align*}
 &\mathbb{P}(\tau > t_i) = \int_{[0,1[^2} \mathbf{1}_{\left\{f(x, \lambda_1) > t_i ~\wedge ~f(y, \lambda_2) > t_i \right\}} dC(x,y),\\
 &\sum_{i = 1}^2 \mathbb{E}\left[ (1- R_i) e^{- r \tau} \mathbf{1}_{\{ \tau < T ~\wedge ~\tau = \tau_i\}} \right] = \\ 
&\int_{[0,1[^2} e^{-r \min\left( f(x, \lambda_1), f(y, \lambda_2) \right)} \biggl(\mathbf{1}_{\left\{ f(x, \lambda_1) \leq \min\left(f(y, \lambda_2), T\right) \right\}} (1 - R_1) \\ 
  &+ \mathbf{1}_{\left \{f(y, \lambda_2) \leq \min\left(f(x, \lambda_1), T\right) \right\}} (1 - R_2) \biggr) dC(x,y),
\end{align*}
where $f(x, \lambda) = \frac{- \log(1-x)}{\lambda}$ is the inverse distribution function of an exponential distribution with parameter $\lambda$ and $\mathbf{1}_{\{(x,y) \in B\}}$ denotes the characteristic function of set $B \subseteq [0,1[^2$.\\

Now we want to calculate the maximal spread $s$ by maximizing over all copulas. We obtain for the spread that
\begin{align}
 s &= \int_{[0,1[^2} \frac{e^{-r \min\left( f(x, \lambda_1), f(y, \lambda_2) \right)}}{\sum_{i = 0}^n e^{- r t_i} \mathbf{1}_{\left\{f(x, \lambda_1) > t_i ~\wedge ~f(y, \lambda_2) > t_i \right\}} }\notag \\
 &\cdot \biggl(\mathbf{1}_{\left\{ f(x, \lambda_1) \leq \min\left(f(y, \lambda_2), T\right) \right\}} (1 - R_1) \notag\\ 
 &+ \mathbf{1}_{\left \{f(y, \lambda_2) \leq \min\left(f(x, \lambda_1), T\right) \right\}} (1 - R_2)\biggr) dC(x,y). \label{ftd}
\end{align}
Note that the value of the integral is finite since the first payment is made at $t_0 = 0$. Furthermore note that the integrand function in this example is not continuous, thus Theorem \ref{main2} cannot be applied. Nevertheless it is clear that our technique provides upper and lower bounds for the optimal values, and since these bounds converge to each other our approach still works.\\

In Table \ref{tab:t1} we present numerical results for a concrete example with three payment times, $t_i = 0,1,2$. One can observe that the resulting copulas (given in Figures \ref{fig: max2} and \ref{fig: max3} for $n = 7,8$, respectively) are highly irregular in the left upper quarter of the unit square. However for $n = 10$ the upper and lower bounds for the optimal values are almost equal.

\begin{table}[!ht]
\centering
\begin{tabular}[h]{|r|r|r|r|r|r|r|r|}
\hline
	$\lambda_1$ & $\lambda_2$ & $R_1$ & $R_2$ & $T$ & $r$ & $t_i$ &\\
\hline
$\frac{1}{3}$ & $\frac{1}{2}$	& 0.5	& 0.7	& 2 & 0.05 & (0, 1, 2)&\\	
\hline
	$n$ & 3 & 4 & 5 & 6 & 7 & 8& 10\\
\hline
$\overline{UB}$		& 0.3601 & 0.3355 &0.3301 &0.326 & 0.322 & 0.3202  &0.3195\\
\hline
$\overline{LB}$ 	& 0.2956  & 0.3031 & 0.314 & 0.318 & 0.3183 & 0.3189 & 0.3195\\
\hline
$\underline{UB}$		& 0.1714  & 0.1674  & 0.1567 & 0.1535 & 0.1519  &  0.1505  & 0.1498\\
\hline
$\underline{LB}$ 	& 0.1453  & 0.1456 & 0.1458 & 0.1480  & 0.1492 & 0.1492 & 0.1495 \\
\hline
\end{tabular}
 \caption{Approximation of the maximal spread of a FTD, where $\overline{UB}$ and $\overline{LB}$ and $\underline{UB}$ and $\underline{LB}$ denote the values of the upper and the lower bounds of the maximal and minimal value of the integral, and $n$ the fineness of the approximation according to Theorem \ref{main2}.}
 \label{tab:t1}
\end{table}

\begin{figure}[h!]
 \centering
 \includegraphics[scale=0.6]{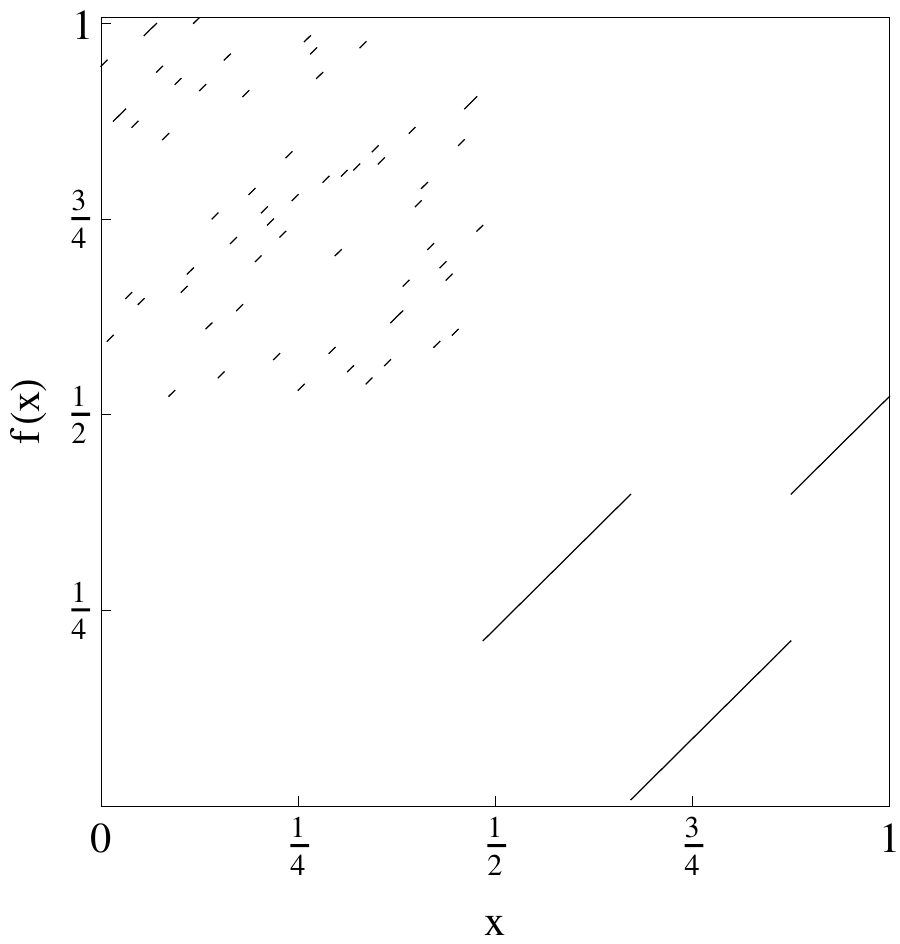}
\caption{Copula which attains the upper bound for the maximal value with $n = 7$.}
\label{fig: max2}
\end{figure}

\begin{figure}[h!]
 \centering
 \includegraphics[scale=0.6]{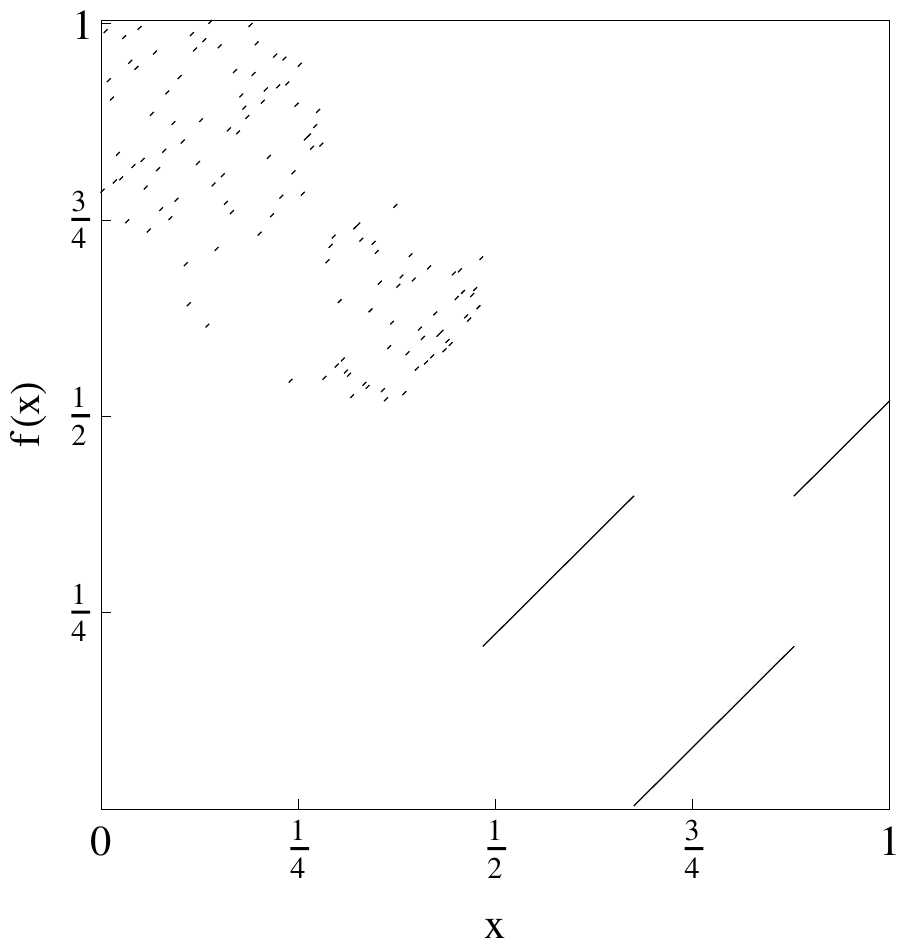}
\caption{Copula which attains the upper bound for the maximal value with $n = 8$.}
\label{fig: max3}
\end{figure}

The method just presented can be used to derive sharp bounds for integrals of piecewise constant functions with respect to copulas. This extends the scientific literature on this topic, which is in general still open. \\ The numerical effectiveness of our method was illustrated in the two numerical examples coming from different branches of applied mathematics just proposed.\\

However, as pointed out also in the previous chapters, an interesting problem concerns the extension of the method to the multidimensional case. This would be of particular interest since finding bounds for multidimensional integrals with respect to copulas has several applications in fields of mathematics such as number theory and actuarial mathematics.\\
Unfortunately, at the moment this method fails to be optimal in higher dimension, since the so-called multi-index assignment problems are in general NP-hard.
We think that one possible way to tackle this kind of problems is to investigate them from a heuristic point of view.
\bibliography{phd_bib}
\bibliographystyle{mri}
\end{document}